\documentclass[11pt,reqno,a4paper]{article}
 \usepackage[english]{babel}
\usepackage[left=2cm,right=2cm,top=2cm,bottom=2cm]{geometry}
\usepackage{amsmath,mathrsfs,stmaryrd}
\usepackage{amsthm}
\usepackage{mathtools}
\usepackage{empheq}
\usepackage{enumerate,enumitem}
\usepackage{dsfont}
\usepackage[hidelinks]{hyperref}
\usepackage{csquotes}
\usepackage{xcolor}
\usepackage{times}
\usepackage{cleveref}
\usepackage{cancel}
\usepackage{cleveref}
\usepackage{verbatim}
\usepackage{mathbbol}

\makeatletter
\@addtoreset{equation}{section}
\makeatother

\theoremstyle{plain}
\newtheorem{theorem}{Theorem}[section] 
\newtheorem{proposition}[theorem]{Proposition}

\newtheorem{assumption}[theorem]{Assumption}
\newtheorem{lemma}[theorem]{Lemma}
\newtheorem{construction}[theorem]{Construction}

\theoremstyle{definition}
\newtheorem{remark}[theorem]{Remark}

\newtheorem{example}[theorem]{Example}   

\newtheorem{definition}[theorem]{Definition}

\newcommand{\R}{\ensuremath{\mathbb R}} 

\newcommand{\dd}{\ensuremath{{\hspace{0.01em}\mathrm d}}}

\newcommand{\V}{\ensuremath{\mathcal V}}
\newcommand{\T}{\ensuremath{{\mathbb T}}} 

\usepackage{amssymb,graphicx}
\newcommand{\bbGamma}{{\mathpalette\makebbGamma\relax}}
\newcommand{\makebbGamma}[2]{%
	\raisebox{\depth}{\scalebox{1}[-1]{$\mathsurround=0pt#1\mathbb{L}$}}%
}
\definecolor{cadmiumgreen}{rgb}{0.0, 0.42, 0.24}
\definecolor{tangelo}{rgb}{0.98, 0.3, 0.0}
\definecolor{shamrockgreen}{rgb}{0.0, 0.62, 0.38}
\definecolor{darkolivegreen}{rgb}{0.33, 0.42, 0.18}
\definecolor{deepmagenta}{rgb}{0.8, 0.0, 0.8}
\definecolor{burgundy}{rgb}{0.5, 0.0, 0.13}
\definecolor{darkbyzantium}{rgb}{0.36, 0.22, 0.33}

\usepackage{authblk}

\usepackage{lipsum}

\newcommand\blfootnote[1]{%
	\begingroup
	\renewcommand\thefootnote{}\footnote{#1}%
	\addtocounter{footnote}{-1}%
	\endgroup
}

\begin{document}

\title{Pathwise central limit theorem and moderate deviations via rough paths for SPDEs with multiplicative noise.} 

 \author{Emanuela Gussetti}
  \affil{\small Bielefeld University, Germany \vspace{-1em} }

\maketitle

\blfootnote{
	\textit{Mail}: emanuela.gussetti@uni-bielefeld.de.
}

\smallskip

\begin{abstract}
We put forward a general framework for the study of a pathwise central limit theorem (CLT) and a moderate deviation principle (MDP) for stochastic partial differential equations perturbed with a small multiplicative linear noise by means of the theory of rough paths. The CLT can be interpreted as the convergence to a pathwise derivative of the It\^o-Lyons map. The result follows by applying a  pathwise Malliavin-like calculus for rough paths and from compactness methods. The convergence in the CLT is quantified by an optimal speed of convergence.
From the exponential equivalence principle and the knowledge of the speed of convergence, we can derive easily a MDP. In particular, we do not apply the weak convergence approach usually employed in this framework. We derive a pathwise CLT and a MDP for the stochastic Landau-Lifschitz-Gilbert equation in one dimension, for the heat equation and for a stochastic reaction-diffusion equation. As a further application, we derive a pathwise convergence to the CLT limit and a corresponding MDP for equations driven by linear It\^o noise.
\end{abstract}


\section{Introduction}
The objective of this work is to establish a central limit theorem (CLT) and the corresponding moderate deviations principle (MDP) for the deviations of the unique solution $u^\epsilon$ to a stochastic partial differential equation perturbed by a small Stratonovich noise from the unique solution $u$ to the deterministic partial differential equation 
\begin{align*}
u^\epsilon_{t}=u^\epsilon_0+\int_{0}^{t}b(u^\epsilon_r)\dd r+\int_{0}^{t} u^\epsilon_r\circ\dd \sqrt{\epsilon} W_r\,,\quad\quad\quad \delta u_{s,t}=\int_{s}^{t}b(u_r)\dd r\,,
\end{align*}
where $W$ is a Brownian motion and the drift $b$ is possibly nonlinear.
 By central limit theorem we mean the limit for $\epsilon$ approaching $0$ of the sequence
\begin{align*}
X^\epsilon:=\frac{u^\epsilon-u}{\sqrt{\epsilon}}\, 
\end{align*}
in a proper norm. In the literature, it is usual to observe a \textit{linearisation} effect in the CLT, namely that the limit $X$ of the sequence $(X^\epsilon)_{\epsilon}$ for vanishing $\epsilon$ is a solution to the equation
\begin{align*}
	X_t=X_0+\int_{0}^{t} b'(u_r)X_r \dd r+ \int_{0}^{t} u_r \circ\dd W_r\,,
\end{align*}
where $b'(u)$ is the Frechét derivative and the stochastic integral is intended in the Stratonovich sense.
Thus, the limit $X$ is the solution to a linear equation with additive noise. This is in contrast to the original equation for $u^\epsilon$, where a multiplicative noise appears and $b$ can be nonlinear. In the classical Stratonovich calculus, the sequence $(X^\epsilon)_{\epsilon}$ converges to $X$ in $\mathcal{L}^2(\Omega)$ and the boundedness of the second moment of the initial condition is required. Some references in the literature for the convergence of stochastic PDEs to the CLT in expectation are the works \cite{wang_zhai_zhang,wang_zhang, Xiong_Zhang}. 

To explain this linearisation effect, consider a map $\Phi$ that associates to the noise the solution to the SPDE, i.e.  $\Phi: W\mapsto \mathrm{Solution \, to\,  SPDE}$. We can rewrite formally each $X^\epsilon$ as
\begin{align*}
X^\epsilon= \frac{u^\epsilon-u}{\sqrt{\epsilon}}=\frac{\Phi({\sqrt{\epsilon}}W)-\Phi(0)}{\sqrt{\epsilon}}\,.
\end{align*}
This leads to the intuition that the limit $X$ of the sequence $(X^\epsilon)_{\epsilon}$ is the derivative, in some sense, of the map $\Phi$ in \enquote{0} in the direction \enquote{$\sqrt{\epsilon}W$}. But the map $\Phi$ is not in general continuous with respect to the uniform topology on the space of continuous functions. We can nevertheless recover the continuity of the map $\Phi$ if we interpret the SPDE above in the rough path framework. 
By interpreting the Stratonovich integral by means of rough paths theory, we consider the pathwise formulation of the equation
\begin{align*}
\delta u^\epsilon_{s,t}=\int_{s}^{t}b(u^\epsilon_r)\dd r+\sqrt{\epsilon}W_{s,t}u^\epsilon_s+\epsilon \mathbb{W}_{s,t} u^\epsilon_s+u^{\epsilon,\natural}_{s,t}\,,
\end{align*}
where $\mathbf{W}\equiv (W,\mathbb{W})$ is the corresponding rough path and $u^{\epsilon,\natural}$ is a remainder term of order $o(|t-s|)$.
Under existence and uniqueness of the solutions $u^\epsilon$ and $u$, there exists a map $\Phi$ from the space of rough paths to the space of continuous functions so that $u^\epsilon=\Phi(\tau_{\sqrt{\epsilon}}\mathbf{W})$, $u^0=\Phi(\mathbf{0})$. Here $\mathbf{W}, \mathbf{0}$ are respectively rough paths lifts of the Brownian motion with Stratonovich lift and of the null path. By $\tau_{\sqrt{\epsilon}}$ we denote a scaling of the rough path, defined rigorously in \eqref{eq:def_space_dependence_rp}.

The map $\Phi$ is the so called It\^o-Lyons map and it is continuous with respect to the noise in the $p$-variation norm on the space of rough paths. By employing a pathwise Malliavin calculus, we can identify rigorously the limit $X$ as a directional derivative of $\Phi$ centred in $\mathbf{0}$ and in the direction $\mathbf{W}$. This identification of the limit $X$ with the derivative of  a map explains the linearization effect observed in the limit equation. 

We prove rigorously the previous considerations for the heat equation, a reaction-diffusion equation and the Landau-Lifschitz-Gilbert equation (LLG). We denote the corresponding drifts generically by $b_i(u)=\Delta u+\gamma_i(u)$ for $i=1,2,3$ and state precise definitions below. We give first the general statement of the theorem and after we comment on the conditions to apply it to the three different SPDEs.
\begin{theorem}\label{teo:CLT_intro}
	(Pathwise central limit theorem). Let $u^0\in H^1$ and let $\mathbf{W}$ be Stratonovich lift of Brownian motion (in the sense of bounded rough drivers). Let $u$ be the unique solution to the equation
	\begin{align*}
	\delta u_{s,t}=\int_{s}^{t}[\Delta u_r+\gamma_i(u_r)]\dd r +W_{s,t} u_s+\mathbb{W}_{s,t} u_s+u^{\natural}_{s,t}
	\end{align*}
	with initial condition $u^0$ in the space $L^\infty(H^1)\cap L^2(H^2)\cap \mathcal{V}^p(L^2)$.
	Denote by $\Phi_i(\mathbf{W})=u$ the associated It\^o-Lyons map.  Then the sequence $(X^{\epsilon})_\epsilon$, defined as
	\begin{align*}
	X^{\epsilon}=\frac{\Phi_i(\tau_{\sqrt{\epsilon}}\mathbf{W})-\Phi_i(\mathbf{0})}{\sqrt{\epsilon}}\,,
	\end{align*}
	converges strongly in $L^\infty(H^1)\cap L^2(H^2)$ and $\mathbb{P}$-a.s. for $\epsilon \rightarrow 0$ to a limit $X$.  In particular $X$ is the pathwise directional derivative of the It\^o-Lyons map $\Phi_i$ in $\mathbf{0}$ in the direction $\mathbf{W}$, where $X$ is the unique solution to 
	\begin{equation*}
	\begin{aligned}
	\delta X_{s,t}=\int_{s}^{t}[\Delta X_r +\gamma_i'(u_r)X_r]\dd r+W_{s,t}u_s+X^\natural_{s,t}\,,
	\end{aligned}
	\end{equation*}
	where $\gamma_i'$ is the Frechét derivative of $\gamma_i$.
	Moreover, the convergence occurs with speed $\sqrt{\epsilon}$, i.e.
	\begin{align*}
	\sup_{0\leq t\leq T}\|X^{\epsilon}_t-D\Phi_i[\mathbf{0}](\mathbf{W})_t\|^2_{H^1}+\int_{0}^{T}\|X^{\epsilon}_r-D\Phi_i[\mathbf{0}](\mathbf{W})_r\|^2_{H^2}\dd r\lesssim_{\mathbf{W}} \sqrt{\epsilon}\,.
	\end{align*}
\end{theorem}
More specifically, we consider the following three SPDEs:
\begin{itemize}
	\item the \textit{heat equation} on $\mathbb{T}^d$, with $d=1,2,3$: we set $\gamma_1\equiv 0$, $u^0\in H^1(\mathbb{T}^d;\mathbb{R}^n)$ and the rough driver lift of Brownian motion with a smoother space component $\mathbf{W}$ as in Definition~\ref{def:rough_driver} (see Section~\ref{sec:CLT_SPDEs}). 
	\item a \textit{reaction-diffusion equation} on $\mathbb{T}^d$, with $d=1,2,3$: we set $\gamma_2(r)=r(1-|r|^2)$ for $r\in \mathbb{R}^n$: $u^0\in H^1(\mathbb{T}^d;\mathbb{R}^n)$ and the rough driver lift of Brownian motion with a smoother space component $\mathbf{W}$ as in Definition~\ref{def:rough_driver} (see Section~\ref{section:reaction_diffusion}). 
	The drift $\Delta X_r+\gamma'_2(u_r)X_r$ in the case of the reaction-diffusion is given by 
	\begin{align*}
	\Delta X_r+X_r-X_r|u_r|^2-2u_rX_r\cdot u_r\,.
	\end{align*}
	\item the \textit{Landau-Lifschitz-Gilbert equation} on $\mathbb{T}^1$: we set $\gamma_3(f):=f\times \partial_x^2 f +f|\partial_x f|^2$ for $f\in H^2$, $u^0\in H^1(\mathbb{T}^d;\mathbb{S}^2)$ and the rough driver lift of Brownian motion with a smoother space component $\mathbf{W}$ as in Section~\ref{sec:LLG}.  The drift $\partial_{x}^{2} X_r+\gamma'_3(u_r)X_r$ in the case of the LLG is given by 
	\begin{align*}
	\partial_{x}^{2} X_r +\partial_x (u_r \times \partial _{x}X_r)+
	\partial_x (X_r\times \partial_{x}u_r)
	+2u_r\partial_x u_r\cdot \partial_x X_r+|\partial_x u_r|^2 X_r\,.
	\end{align*}
\end{itemize}
The approach is therefore suitable for semi-linear equations, like the stochastic reaction-diffusion equation, and for quasi-linear equations, like the stochastic LLG equation. Note that the result in Theorem~\ref{teo:CLT_intro} holds also for initial conditions in $L^2(\mathbb{T}^d,\mathbb{R}^n)$ for the heat equation and the reaction-diffusion equation, with strong convergence to the CLT in $L^\infty(L^2)\cap L^2(H^1)$: we set the initial condition in $H^1(\mathbb{T}^d,\mathbb{R}^n)$ to show the higher order regularity. 

 We remark that the CLT for the LLG equation is new and the pathwise notion of solution plays a fundamental role in establishing this convergence.  
It does not seem to be an easy task to achieve a CLT via the classical Stratonovich calculus. Whereas, via rough paths theory, the problem can be solved. The motivation is that Stratonovich and It\^o's calculus need moment estimates to bound the stochastic integrals via the Burkholder-Davis-Gundy inequality. On the other side, due to the non linearities in the drift of the LLG, it does not seem to possible to close the Gronwall's estimate by adding the expectation to the equation. This problem is not seen by rough paths calculus, where no estimate in expectation is required (see Remark 4.5 in \cite{LLG_inv_measure} and \cite{LLG1D} for an explanation on the role of rough paths calculus in the study of the stochastic LLG equation).

 Besides giving a rigorous interpretation of the linearisation effect, the rough path approach to the CLT leads to the following advantages:
\begin{itemize}
\item[1.]The convergence of $(X^\epsilon)_{\epsilon}$ to the derivative $X$ occurs pathwise, with the optimal speed of convergence $\sqrt{\epsilon}$. \textit{The same pathwise result holds when the heat equation or the reaction-diffusion equation considered above are driven by a multiplicative linear noise, with an integral interpreted in the It\^o sense}. This improves the results present in the literature, for a noise of this form. For example, a solution to the heat equation with rescaled noise, can be interpreted as the unique solution $y^\epsilon$ to
\begin{align*}
\delta y^\epsilon_{s,t}&=\int_{s}^{t} \Delta y^\epsilon_r\dd r+\int_{s}^{t}y^\epsilon_r\sqrt{\epsilon}\dd W_r
=\int_{s}^{t} \Delta y^\epsilon_r\dd r+\frac{\epsilon}{2}\int_{s}^{t}y^\epsilon_r\dd r+\int_{s}^{t}y^\epsilon_r\sqrt{\epsilon}\circ\dd W_r\,,
\end{align*}
where we observe that the scaling of the It\^o-Stratonovich correction is $\epsilon$. By interpreting the Stratonovich integral in the sense of rough paths, we can replicate the proof of Theorem~\ref{teo:CLT_intro} and achieve the convergence of $((y^\epsilon-\Phi_1(\mathbf{0}))/\sqrt{\epsilon})_{\epsilon}$ to the same limit $X$ in $L^\infty(H^1)\cap L^2 (H^2)$ and $\mathbb{P}$-a.s (see the discussion in Remark~\ref{remark:ito_CLT}). 

The pathwise convergence in the CLT does not seem to be easily reachable in the framework of It\^o calculus. This is due to the necessity of estimating the stochastic integrals by means of the Burkholder-Davis-Gundy inequality.
\item[2.] The process $W$ does not need to be a Brownian motion, nor a Gaussian. It is therefore possible to obtain the so called \textit{non central} limit theorems, namely limit theorems with non-Gaussian limit. 
\item[3.] At a pathwise level, no moments of the initial condition are needed (observe that $X_0=0$). At the level of the CLT, the random initial condition does not even need to be measurable.
\item[4.] The noise is fixed at the beginning and it stays in the limit: indeed no Girsanov transform is involved.
\end{itemize}

\paragraph{A Moderate Deviations Principle.} Besides the CLT, one can look at large and moderate deviations from the deterministic solution. This corresponds to the study of the the speed of convergence to $0$ of the limit for $\epsilon\rightarrow 0$ of
\begin{align*}
Y_i^\epsilon:= \frac{\Phi_i(\tau_{\sqrt{\epsilon}}\mathbf{W})-\Phi_i(\mathbf{0})}{\sqrt{\epsilon}\lambda(\epsilon)}\,,
\end{align*}
where $\lambda(\epsilon)>0$ and $\lambda(\epsilon)\sqrt{\epsilon}\rightarrow 0$ and $\lambda(\epsilon)\rightarrow +\infty$ as $\epsilon\rightarrow 0$. For $\lambda(\epsilon)=1/\sqrt{\epsilon}$ we are in the large deviation principle (LDP) regime: the limit above converges pathwise to $0$ and, by means of the continuity of the It\^o-Lyons map, it is possible to establish a speed of convergence.

 In the large deviations we look at deviations of the stochastic solution from the deterministic one and we investigate bounds for the occurrence of rare events. We refer to M.~Ledoux, Z.~Qian, T.~Zhang \cite{Ledoux_Qian_Zhang_p_var} for the LDP via rough paths.
For $\lambda(\epsilon)\in (0,1/\sqrt{\epsilon})$, we are in the moderate deviations principle (MDP) regime: we look at \enquote{less rare} events in comparison to the LDP regime, but still rare in comparison to the average behaviour expressed in the CLT. Intuitively, we are looking at the deviations (or fluctuations) from the CLT behaviour. 
We restrict ourselves to the study of the MDP for SPDEs driven by a Brownian motion. Indeed a LDP for the Brownian motion is known and we can pass this property to the solution to the SPDEs by means of the continuity of the It\^o-Lyons map and the contraction principle. The following MDP result for the SPDEs mentioned above holds.
\begin{theorem}\label{th:MDP_intro} \textit(Moderate deviations principle).
	Let $ \mathbf{W}\equiv(W,\mathbb{W})$ be the Stratonovich lift of Brownian motion (in the sense of rough drivers). The sequence $((\Phi_i(\tau_{\sqrt{\epsilon}}\mathbf{W})-\Phi_i(\mathbf{0}))/\sqrt{\epsilon} \lambda(\epsilon))_\epsilon$ satisfies a large deviation principle on $L^\infty(H^1)\cap L^2(H^2)\cap \mathcal{V}^p(L^2)$ with speed $\lambda(\epsilon)^2$ and with good rate function  
	\begin{align*}
	\tilde{I}(f)=\inf\bigg(\frac{1}{2}\int_{0}^{T}|h(s)|^2 \dd s:h\in \mathcal{H}, \; D\Phi_i[\mathbf
	0](\mathbf{h})=f\bigg)\,, \quad \inf \emptyset:=+\infty\,,
	\end{align*}
	where $\mathcal{H}$ is the space of Cameron-Martin paths with respect to $W$ and $X^{h,i}=D\Phi_i[\mathbf
	0](\mathbf{h})$ is the unique solution to the skeleton equation
	\begin{align*}
	\delta X^{h,i}_{s,t}=\int_{s}^{t} [\Delta X^{h,i}_r +\gamma_i'(\Phi_i(\mathbf{0})_r)X^{h,i}_r]\dd r+h_{s,t}\Phi_i(\mathbf{0})_s+X^{h,i,\natural,}_{s,t}\,.
	\end{align*}
\end{theorem}

The usual approach to LDPs and MDPs for SPDEs is the, so-called, \enquote{weak convergence approach} introduced and studied in P.~Dupuis, R.~S.~Ellis \cite{Dupuis_Ellis}, A.~Budhiraja, P.~Dupuis \cite{budhiraja2000variational}, M.~Boué,  P.~Dupuis \cite{boue_dupuis}. 
Instead, we use a strategy of proving the MDP via the exponential equivalence for SPDEs, as introduced by R.~Wang, T.~Zhang \cite{wang_zhang}, combined with the contraction principle.  
We show that $Y_i^\epsilon$ is exponentially equivalent to $(D\Phi_i[\mathbf{0}](\mathbf{W})/\lambda(\epsilon))_{\epsilon}$. Then, from the contraction principle $(D\Phi_i[\mathbf{0}](\mathbf{W})/\lambda(\epsilon))_{\epsilon}$ satisfies a LDP with a specific rate function, it follows that also $(Y_i^\epsilon)_\epsilon$ satisfies a MDP. 
Note that the random initial condition needs to be measurable for the MDP and have some finite moments. This follows from the fact that the pathwise estimates used in the MDP are dependent on the initial condition (see Lemma~\ref{lemma:exp_equiv}). A MDP for the heat equation and for the reaction-diffusion equation is already known: the proof presented here is different, since it does not rely on the weak convergence approach and it exploits the contraction principle. The MDP for the LLG is new.

We explore now more technical aspects.

\textbf{On the notion of derivative.} The differentiability of the It\^o-Lyons map is a very well studied topic, in particular in the framework of rough differential equations. In this work we rely on a notion of Gateaux-like derivative introduced by Z.~Qian and J.~Tudor \cite{Qian_Tudor}. The authors construct a derivative which extends the notion of Malliavin calculus on the whole space of weakly geometric rough paths and widens the possible directions of derivation beyond the Cameron-Martin directions. Their notion of derivative has a geometric interpretation and they construct a tangent space, where the derivative lies. We follow their considerations regarding the notion of increment and the definition of derivative. We choose a restriction of this notion of derivative to geometric rough paths because we also want to allow differentiability in direction that are neither of complementary Young regularity nor in the subspace of controlled rough paths. We also discuss some further types of directions beyond the Cameron-Martin paths associated with Brownian motion, already present in the literature (see the literature review below). In particular, we need to consider random rough paths which are measurable with respect to the underlying probability space.

\textbf{CLT as convergence to a derivative not centered in $\mathbf{0}$, i.e. $D\Phi[\mathbf{H}](\mathbf{W})$.} Classically, it is usual to look at the deviations of an SPDE from the deterministic solution to the associated PDE. This corresponds to the convergence of the sequence $(X^\epsilon)_{\epsilon}$ to a pathwise derivative in $\mathbf{0}$ in the direction $\mathbf{W}$, denoted by $D\Phi[\mathbf{0}](\mathbf{W})$. Let consider the easy case of a random geometric rough path $\omega\mapsto (B(\omega),\mathbb{B}(\omega))$ and the joint lift with a random path $H$ of complementary Young's regularity with respect to the Brownian motion. In this case, we can study the deviation of the unique solution $u^{\epsilon,\mathbf{H}}$ to 
\begin{align*}
\delta u^{\epsilon,\mathbf{H}}_{s,t}=\int_{s}^{t}b(u^{\epsilon,\mathbf{H}}_r)\dd r+\int_{s}^{t} u^{\epsilon,\mathbf{H}}_r\dd \{\mathbf{H}+\tau_{\sqrt{\epsilon}}\mathbf{W}\}\, .
\end{align*}
from the pathwise unique solution $u^\mathbf{H}$ to the equation
\begin{align}\label{eq:u_H_intro}
\delta u^{\mathbf{H}}_{s,t}=\int_{s}^{t}b(u^{\mathbf{H}}_r)\dd r+\int_{s}^{t} u^{\mathbf{H}}_r\dd \mathbf{H}_r\, .
\end{align}
We can consider also a random rough path $\omega\mapsto\mathbf{H}(\omega)$, where each path $H(\omega)$ is of complementary Young's regularity. In this case, the integral in \eqref{eq:u_H_intro} is a Young's integral and the CLT corresponds to studying the deviations from an SPDE (in place than the deviations from the deterministic solution). 
Note that we are looking at the deviations of a stochastic equation from another stochastic equation. Under well posedness hypothesis, the CLT in this case coincides with the convergence to the pathwise directional derivative $D\Phi[\mathbf{H}](\mathbf{W})$.
The directional derivative $z^\mathbf{H}=D\Phi[\mathbf{H}](\mathbf{W})$ is then the solution to an SDE with both additive and multiplicative noise, given by
\begin{align*}
\delta z^{\mathbf{H}}_{s,t}&=\int_{s}^{t}b'(u^{\mathbf{H}}_r)z^{\mathbf{H}}_r \dd r+W_{s,t}u^{\mathbf{H}}_s+H_{s,t}z^{\mathbf{H}}_{s}+([HW]+[WH])_{s,t}u^{\mathbf{H}}_s+\mathbb{H}_{s,t}z^{\mathbf{H}}_s+(z^{\mathbf{H}})^\natural_{s,r,t}\,,
\end{align*}
where $[HW]$ is the Young's integral of $H$ against $W$ (resp. $[WH]$ is the the Young's integral of $W$ against $H$).
 It is also possible to consider $W, H$ to be two independent Brownian motions and consider the joint lift of the two with respect to the Stratonovich integration theory.
The problem of perturbing a given rough differential equation by another rough path is already present in the literature concerning non linear filtering, also in the rough paths framework (see e.g. J. Diehl, H. Oberhauser, S. Riedel \cite{Diehl_Ober_Riedel}, J.~Diehl, P.~Friz, W.~Stannat \cite{Diehl_Friz_Stannat}). Those equations from non linear filtering can represent a motivation for looking into $D\Phi[\mathbf{H}](\mathbf{W})$ deviations of a rough differential equation driven by $\mathbf{H}$ in the direction $\mathbf{W}$.

\subsubsection{Overview of the literature.}
\textit{CLTs and MDPs.} Some references in the literature for the convergence of stochastic PDEs to the CLT in expectation are the works \cite{wang_zhai_zhang,wang_zhang, Xiong_Zhang}. 
Recently, L.~Galeati and D.~Luo studied a CLT and a LDP in the framework of regularisation via transport noise \cite{Galeati_Dejun}: also in this case, the authors obtain the convergence to the CLT in $\mathcal{L}^p(\Omega)$. We also refer to this paper for a review on the CLTs present in the literature concerning transport noise. The MDPs for stochastic partial differential equations have applications in the field of importance sampling: it is often difficult to simulate numerically events in the large deviations regime: this is due to the non linearity of the rate function. On the contrary, the rate function for MDPs depends on a linear equation: one simulates less rare events in comparison to the LDPs regime, but with potential computational advantages. We refer the interested reader to the works of I.~Gasteratos, M.~Salinis, K.~Spiliopoulos \cite{ioannis_paper_2,ioannis_paper_1}.

\textit{Pathwise Malliavin Calculus.} The differentiability of the It\^o-Lyons map is not a new topic in rough path theory: different types of pathwise Malliavin-like calculi have been introduced. We mention Chapter 11, Chapter 20 in P.~Friz, N.~B.~Victoir \cite{FrizVictoir},  Chapter 11 in P.~Friz, M.~Hairer \cite{FrizHairer} for an introduction. 
We refer to work of  L.~Coutin and A.~Lejay \cite{coutin_lejai_omega_lemma} and the references therein, for further extensions. We also mention the work of G.~Cannizzaro, P.~K.~Friz, P.~Gassiat \cite{cannizzaro_friz_gassiat} in the context of singular SPDEs: the authors introduce a Malliavin calculus in the framework of regularity structures. Y.~Inahama \cite{Inahama_Malliavin} studied the Malliavin differentiability of rough differential equations driven by Gaussian rough paths. In P.~Friz, T.~Nilssen, W.~Stannat \cite{FrizNilssenStannat}, the authors study the stability of semi-linear rough PDEs by means of a pathwise Malliavin calculus.
A different approach to pathwise Malliavin calculus is discussed in Z.~Qian, J.~Tudor \cite{Qian_Tudor}: the authors build a differential structure on the space of weakly geometric rough paths. In contrast to the other works, they highlight the possibility of considering more directions than the Cameron-Martin directions. 

The pathwise Malliavin calculus is normally employed to establish a H\"ormader's theorem \cite{cass_friz_2010,cass_friz_2011,cass_hairer_litterer_tindel,friz_gess_jain_monrad}: the application of Malliavin calculus to the CLT appears to be new in the rough path framework. In the literature, there are results of CLT for stochastic integrals via Malliavin calculus: see for instance D.~Harnett, D.~Nualart  \cite{HarnettNualartCLT}.
In contrast to the present work, the Malliavin calculus is used as a tool, whereas here the pathwise derivative is itself the object of interest in the CLT.

\textit{Beyond the complementary Young regularity paths.} We are interested in lifting two different stochastic processes jointly to the space of rough paths. At the mere path level, we can always lift two weakly geometric rough paths to a joint weakly geometric rough path. This is the content of the extension theorem in T.~Lyons, N.~B.~Victoir \cite{FrizLyonsExtension}. See also the work of L.~Broux, L.~Zambotti \cite{Broux_Zambotti} for a recent approach to the extension theorem and other applications.  But, in order to get measurability of the joint lift, the joint rough path, as a random object, needs to be measurable. For this reason we are interested in measurable joint lifts: the first example is the lift of a random rough path (a rough path valued random variable) jointly with a path of complementary Young regularity: in this case, the lift is unique and measurable (\cite{FrizHairer,FrizVictoir}). J.~Diehl, H.~Oberhauser, S.~Riedel \cite{Diehl_Ober_Riedel} introduce a joint lift between a Brownian motion and a deterministic geometric rough path, by using Stratonovich calculus.
In this last paper, the authors discuss the relation between the joint lift that they can construct and the work of T.~Lyons \cite{lyons1991non_existence} of non existence of path integrals for Brownian motion (with different constructions).
 They also show how to jointly lift a multidimensional Brownian motion with a multidimensional semi-martingale to a geometric rough path (under the assumption that the two are independent). Those lifts to random rough paths are measurable. In P.~Friz, M.~Hairer \cite{FrizHairer}, the authors use the same construction for an It\^o lift on the second iterated integrals. In J.~Diehl, P.~Friz, W.~Stannat \cite{Diehl_Friz_Stannat}, the authors study an SDE driven by an It\^o multiplicative noise and an additional rough path: they apply the problem to non linear filtering. I.~Chevyrev, P.~Friz \cite{chevyrev_friz} show the existence of a joint lift for general semi-martingales (also allowing for jumps). P.~Friz, P.~Zorin-Kranich \cite{FrizKranichranich} prove the existence of a joint lift for semi-martingales, by means of It\^o calculus. The construction relies on the stochastic sewing Lemma in K. L\^e \cite{Le_stoch_sewing}. We mention also the construction in P.~Friz, A.~Hocquet, K. L\^e \cite{Friz_Hocquet_Le}. 

\textit{Applications of the continuity of It\^o-Lyons map.} The main advantage of interpreting stochastic integration via rough paths is the continuity of the so called It\^o-Lyons map: the map which associates to the rough path the integral is continuous. This fact has been employed in different directions: if there exists a unique solution to a rough differential equation driven by a linear multiplicative noise, then a direct consequence is the so-called Wong-Zakai convergence, a support theorem for the solutions and a large deviation principle (in the sense of deviations of the perturbed equation from the deterministic solution). See M.~Ledoux, Z.~Qian, T.~Zhang \cite{Ledoux_Qian_Zhang_p_var}, Chapter 19 in P.~Friz, N.~B.~Victoir \cite{FrizVictoir} and the references therein for a general theory.  

\textit{Energy solutions for SPDEs via rough paths.} We mention \cite{deya2016priori,flandoli_hofmanova_luo_nilssen,LLG1D,hocquet2018ito,hocquet2018quasilinear,hofmanova_leahy_nilssen_AAP,tornstein_embedding} for applications of rough paths theory to the study of energy solutions to SPDEs and their applications. 

\textit{Stochastic Landau-Lifschitz-Gilbert equation on a one dimensional domain.} For existence results of solutions to the stochastic Landau-Lifschitz-Gilbert equation in one dimension and large deviations principles for the LLG we mention the works of Z.~Brzeźniak, B.~Goldys, T.~Jegeraj \cite{brzezniak_LDP} and E.~Gussetti, A.~Hocquet \cite{LLG1D}. For the Wong-Zakai convergence of the solutions to the LLG and the related difficulties we mention the first paper on the topic of  Z.~Brzeźniak, U.~Manna, D.~Mukherjee \cite{brzezniak2019wong} and two subsequent works that address the problem via rough paths theory, from E.~Gussetti, A.~Hocquet  \cite{LLG1D} and K.~Fahim, E.~Hausenblas, D.~Mukherjee \cite{fahim_hausenblas}. For other applications of rough paths theory to the study of the LLG, we mention E.~Gussetti \cite{LLG_inv_measure}, A.~Hocquet, A.~Neamţu \cite{hocquet_neamtu}. We mention  M. Neklyudov, A.~Prohl \cite{neklyudov2013role} for the study of a finite dimensional ensemble of particles and invariant measures. For many problems related to the stochastic LLG, we mention the monograph \cite{Banas_book}.

\subsubsection{Organization of the paper.}In Section~\ref{sec:def_and_notations}, we introduce some basic definitions. In Section~\ref{sec:compatible_direction}, we define the random compatible directions, we give some examples and we introduce the pathwise directional derivative. 
In Section~\ref{sec:CLT_theoric}, we study the CLT for the heat equation, the reaction-diffusion equation and for the stochastic Landau-Lifschitz-Gilbert equation on a one dimensional torus. In Section~\ref{sec:MDP}, we prove a moderate deviation principle for the SPDEs considered before. In Appendix \ref{section:appendix_RP}, we collect some useful results on rough paths theory. In Appendix \ref{Appendix_B_new}, we establish well posedness and continuity of the It\^o-Lyons map for initial condition in $H^1$ for the heat equation. Finally, in Appendix \ref{Appendix_B} we show the well posedness for the limit equation $X=D\Phi[\mathbf{g}](\mathbf{W})$ and we prove the continuity of the map $\mathbf{W}\mapsto X$.

\subsubsection{Acknowledgements.}The author is warmly grateful to Professor M.~Hofmanov\'a for her many advices and constant help while running this project. 
The author is also thankful to F.~Bechtold  and U.~Pappalettera for many interesting discussions and to L.~Galeati for many references on the CLTs. The research was funded by the Deutsche Forschungsgemeinschaft (DFG, German Research Foundation) – SFB 1283/2 2021 – 317210226.

\section{Notations}\label{sec:def_and_notations}
We denote by $\mathbb{N}$ the space of natural numbers $1,2,\dots$, by $\mathbb{N}_0:=\mathbb{N}\cup\{0\}$.
Throughout the paper, a finite but arbitrary time horizon $T>0$ is fixed. 
We denote by $\mathbb{S}^2$ the unit sphere in $\R ^3$ and by $\T^d=\R^d/\mathbb Z^d$ the unit torus.
By $a\cdot b$, we denote the inner product in $\R^n$ for any $n\in\mathbb N$ and $a,b\in \R^n$, and by $|\cdot|$ the norm inherited from it. We will not distinguish between the different dimensions, it will be clear from the context. If $n=3$ we recall the definition
$a\times b:=(a_2b_3-a_3b_2,a_3b_1-a_1b_3,a_1b_2-a_2b_1)$. The space of bounded linear operators from a Banach space $E$ to a Banach space $F$ is denoted by $\mathcal L(E,F)$ and we denote $\mathcal{L}(E):=\mathcal{L}(E,E)$. By $\mathcal{L}_a(\mathbb{R}^3)$ we denote the space of antisymmetric matrices from $\mathbb{R}^3$ to $\mathbb{R}^3$.
For $a\equiv(a_1,\cdots,a_d),b\equiv(b_1,\cdots, b_d)\in\mathbb{R}^d$, we denote $a\otimes b=(a_ib_j)_{i,j}\in \mathcal{L}(\mathbb{R}^d)$ and $a\odot b:=a\otimes b+b\otimes a$. When $E=\R^n,$ we denote by $\mathbf 1_{n\times n}$ the identity map.

\paragraph{Paths and controls.}
Let $I$ be a subinterval of $[0,T]$ and let $I^2:=I\times I$. The shorthand `$\forall s \leq t \in I$' will be used instead of `$\forall (s,t)\in I^2$ such that $s \leq t.$'
Given a one-index map $g$ defined on $I$, we denote by $\delta g_{s,t}:=g_t-g_s$ for $s \leq t \in I$.
If $G$ is defined on $I^2$, we denote by $\delta G_{s,r,t}:=G_{s,t}-G_{s,r}-G_{r,t}$ for $s \leq r\leq t\in I$. 
We call \textit{increment} any two-index map which is given by $\delta g_{s,t}$ for some $g=g_t.$
As it is easy to observe, increments are exactly those 2-index elements $G_{s,t}$ for which $\delta G_{s,r,t}\equiv 0$.

We say that a continuous map $\omega\colon \{s\leq t \in I\}\rightarrow [0,+\infty)$ is a \textit{control on $I$} if $\omega$ is continuous, $\omega(t,t)=0$ for any $t\in I$ and if it is superadditive, i.e.\ for all $s \leq r \leq t$
\begin{align*}
\omega(s,r)+\omega(r,t)\leq \omega (s,t)\,.
\end{align*}
Given a control $\omega$ on $I:=[a,b]$, we also denote $\omega(I):=\omega(a,b)$.

Let $(E,\|\cdot\|_E)$ be a Banach space and $p>0$, we denote by $\mathcal{V}^p_{2}(I;E)$ the set of two-index maps $G\colon\{s\leq t \in I\}\rightarrow E$ that are continuous in both components such that $G_{t,t}=0$ for all $t\in I$ and there exists a control $\omega$ on $I$ so that
\begin{equation} \label{condition_control_1}
\|G_{s,t}\|_E\leq \omega(s,t)^{1/p}
\end{equation}
for all $s\leq t \in I$. 
The space $\mathcal{V}^p_{2}(I;E)$ is equivalently defined as the space of two-index maps of finite $p$-variation, namely $G_{s,t}$ belongs to $\mathcal{V}^p_{2}(I;E)$ if and only if 
\[
\|G\|_{\mathcal{V}_{2}^{p}(I;E)}:=\sup_{\pi}\Big (\sum\nolimits_{[u,v]\in\pi}\|G_{u,v}\|_E^p\Big)^\frac{1}{p}<+\infty\,,
\]
where the supremum is taken over the set of partitions $\pi=\{[t_0,t_1],\dots [t_{n-1},t_{n}]\}$ of $I$.
Moreover, the above semi-norm $\|\cdot\|_{\mathcal{V}^p_{2}(I;E)}$ is in fact equal to the infimum of $\omega(I)^{1/p}$ over the set of controls $\omega$ such that \eqref{condition_control_1} holds (see \cite{hocquet2017energy} or \cite[Paragraph 8.1.1]{FrizVictoir}).	
In a similar fashion we define the space $\mathcal{V}^p(I;E)$ of all continuous paths $g\colon I\to E$ such that $\delta g\in\mathcal{V}^p_{2}(I;E)$. It is endowed with the norm $\|g\|_{\V^p(I;E)}=\sup_{t\in I}\|g_t\|_{E}+\|\delta g\|_{\mathcal{V}_{2}^{p}(I;E)}$.

We will sometimes need to work with local versions of the previous spaces. For that purpose, we define $\mathcal{V}_{2,\mathrm{loc}}^{p}(I;E)$ as the space of two-index maps $G\colon\{s\leq t \in I\}\rightarrow E$ such that there exists a finite covering $(I_k)_{k\in K}$ of $I$, $K\subset \mathbb{N}$, so that $G\in \mathcal V_{2}^{p}(I_k;E)$ for all $k\in K$. 
We define the linear space 
\begin{align*}
\mathcal{V}^{1-}_{2,\mathrm{loc}}(I;E):=\bigcup_{0<p<1} \mathcal{V}_{2,\mathrm{loc}}^{p}(I;E)\,.
\end{align*}
We denote by $\mathcal{V}^p(E)=\mathcal{V}^p([0,T];E),$ 
and by \( \mathcal{V}_2^p(E)=\mathcal{V}_2^p([0,T];E) \). We denote by $C([0,T];E)$ the space of continuous functions defined on $[0,T]$ with values in $E$. 
\paragraph{Sobolev spaces.}
For $n,d\in\mathbb{N} $ we consider the usual Lebesgue spaces $L^p:=L^p(\T^d,\R^n)$, for $p\in[1,+\infty]$ endowed with the norm $\|\cdot\|_{L^p}$ and the classical Sobolev spaces $W^{k,q}:=W^{k,q}(\T^d,\R ^n)$ for integer $q\in [1,+\infty]$ and $k\in\mathbb{N}$ endowed with the norm $\|\cdot\|_{W^{k,q}}$.
We also denote by $H^k:=W^{k,2}(\T^d,\R ^n)$. 
Again, our abbreviations do not distinguish between different (finite) dimensions for the target space.
Sometimes it will be necessary to consider functions taking values in the unit sphere $\mathbb{S}^2\subset\R^3$: for that purpose, we adopt the notation
\begin{equation*}
H^k(\T^d ;\mathbb{S}^2):=H^k(\T^d;\R^3)\cap \{f:\T^d \rightarrow \R^3 \, \textrm{s.t.}\, |f(x)|=1\text{ a.e.}\}\,,
\end{equation*}
for $k\in \mathbb{N}$. 
Finally, we will denote by $L^p(W^{k,q}):=L^p([0,T];W^{k,q}(\T^d ;\R ^n))$. 
\paragraph{Rough paths and rough drivers.}
We introduce the notion of geometric rough path and of controlled rough path. We refer the reader to the monographs \cite{FrizHairer}, \cite{FrizVictoir} for an introduction to rough paths theory. 
\begin{definition}\label{defi-rough-path}
	Fix a time $T>0$ and let  $2\leqslant p<3$. Then we call a \textit{continuous $p$-variation rough path} on $[0,T]$ any pair 
	\begin{equation}\label{p-var-rp}
	\mathbf{X}=(X,\mathbb{X}) \in \mathcal{V}^p_2 (\mathbb{R}^d) \times \mathcal{V}^{p/2}_2 (\mathbb{R}^{d\times d}) 
	\end{equation}
	that satisfies \textit{Chen's relation} 
	\begin{equation}\label{chen-rela}
	\delta \mathbb{X}_{s,r,t}=X_{s,r}\otimes X_{r,t} \, \qquad s<r<t\in [0,T] \,,
	\end{equation}
	 A continuous $p$-variational rough path $\mathbf{X}$ is said to be a \textit{p-geometric rough path} if  for all $s<t\in [0,T]$
    \begin{align*}
    \mathrm{Sym}(\mathbb{X}_{s,t})=\frac{1}{2}X_{s,t}\otimes X_{s,t}\,,
    \end{align*}
    where $ \mathrm{Sym}(\mathbb{X}_{s,t}):=(\mathbb{X}_{s,t}+\mathbb{X}_{s,t}^T)/2$ is the symmetric part of the matrix $\mathbb{X}_{s,t}$.
\end{definition}
Given a rough path $\mathbf{X}\equiv(X,\mathbb{X})$, we refer to $X$ as \textit{first iterated integral} and to $\mathbb{X}$ as \textit{second iterated integral}. Given a path $X\in \mathcal{V}^p(\mathbb{R}^d)$, we say that a rough path $\mathbf{X}\equiv(X,\mathbb{X})$ is the \textit{lift} of $X$ or that $X$ \textit{can be lifted} to a rough path $\mathbf{X}$.
We denote by $\mathcal{RP}^p(\mathbb{R}^d)$ the space of continuous geometric $p$-variation rough paths.
Throughout the paper, we consider rough paths of $p$-variation with $2\leq p<3$ (unless stated otherwise).
\begin{remark}
	Recall that Chen's relation corresponds to some good additivity properties of an integation theory, whereas the geometricity property corresponds to the notion of integration by part.
\end{remark}
\begin{remark}
If $\mathbf{X}$ is a $p$-geometric rough path, then it can be obtained as the limit, for the $p$-variation topology involved in \eqref{p-var-rp}, of a sequence of smooth rough paths $(\mathbf{X}^\epsilon)_{\epsilon}$, that is with $\mathbf{X}^\epsilon=(X^{\epsilon},\mathbb{X}^{\epsilon})$ explicitly defined as
\begin{equation*}
X^{\epsilon}_{s,t}:=\delta x^{\epsilon}_{s,t} \, \qquad \mathbb{X}^{\epsilon}_{s,t}:=\int_s^t \delta x^{\epsilon}_{s,r} \, \otimes dx^{\epsilon}_r \,
\end{equation*}
for some smooth path $x^\epsilon:[0,T] \to \mathbb{R}^d$ (see \cite{FrizHairer}). 
\end{remark}
We define for all $\epsilon>0$ the \textit{dilatation operator} as 
$\tau_{\epsilon}\mathbf{X}:=(\epsilon X,\epsilon^2\mathbb{X})$. We extend Definition~\ref{defi-rough-path} by introducing the space dependence: let $g\in L^\infty(\mathbb{T}^d,\mathbb{R})$ and $\mathbf{X}\equiv(X,\mathbb{X})\in \mathcal{RP}^p(\mathbb{R}^d)$. For a.e. $x\in\mathbb{T}^d$, we define
\begin{align}\label{eq:def_space_dependence_rp}
\mathbf{W}(x)\equiv(W(x),\mathbb{W}(x)):= \tau_{g(x)}\mathbf{X}\equiv(g(x)X,g(x)^2\mathbb{X})\in \mathcal{V}^p_2(\mathbb{R}^d)\times \mathcal{V}^{p/2}_2(\mathbb{R}^{d\times d})\,.
\end{align}
We refer to $\mathbf{W}$ as \textit{rough driver}. Indeed this procedure to include the space dependence into the rough path can be generalised by introducing the formalism of the unbounded rough driver (first introduced in \cite{bailleul2017unbounded}). In this work, we restrict the attention to rough paths with spatial component introduced as in \eqref{eq:def_space_dependence_rp}.  We define the space of rough drivers with spatial component in $L^\infty(\mathbb{T}^d;\mathbb{R})$ by
\begin{align*}
\mathcal{RD}^p(\mathbb{R}^d,L^\infty) :=\{\mathbf{W}=\tau_{g(x)}\mathbf{X}: g\in L^\infty(\mathbb{T}^d;\mathbb{R}),\quad \mathbf{X}\in \mathcal{RP}^p(\mathbb{R}^d)\}\,.
\end{align*}
We endow the space $\mathcal{RD}^p(\mathbb{R}^d,L^\infty)$ with the metric 
\begin{align*}
\rho(\mathbf{G},\mathbf{H}):=\|G-H\|_{\mathcal{V}^p_2(L^\infty)}+\|\mathbb{G}-\mathbb{H}\|_{\mathcal{V}^{p/2}_2(L^\infty)}\,,
\end{align*}
for all $\mathbf{G}\equiv(G,\mathbb{G})$ and $\mathbf{H}\equiv(H,\mathbb{H})$ in $\mathcal{RD}^p(\mathbb{R}^d,L^\infty)$.
\section{Random compatible directions and pathwise directional derivatives.} \label{sec:compatible_direction}
We define what we mean by \textit{increment} of a map $\Phi:\mathcal{RP}^p(\mathbb{R}^d)\rightarrow E$, where $E$ is a real vectorial space. 
Given two $p$-geometric rough paths $\mathbf{X}\equiv(X,\mathbb{X})$ and $\mathbf{Y}\equiv(Y,\mathbb{Y})$, for all $\epsilon>0$  the \textit{increment} of a map $\Phi:\mathcal{RP}^p(\mathbb{R}^d)\rightarrow E$ in $\mathbf{X}$ in the direction $\mathbf{Y}$ is the difference
\begin{align}\label{eq:increment}
\Phi(\{\mathbf{X}+\tau_\epsilon\mathbf{Y}\})-\Phi(\mathbf{X})\,,
\end{align}
where we have not yet defined the operation $\{\,\cdot\,+\, \cdot \, \}$. We need to define the sum \enquote{$\{\mathbf{X}+\tau_\epsilon\mathbf{Y}\}$} such that it remains a geometric rough path.
This problem is already addressed in Z.~Qian, J.~Tudor \cite{Qian_Tudor} in the broader framework of weak geometric rough paths: we follow their considerations regarding the definition of addition of two rough paths (see also \cite{FrizVictoir}), which we recall for the reader's convenience. We mention \cite{bellingeri_friz_paycha_preiss} for a different notion of sum of rough paths (for smooth paths). 

\paragraph{Sum of rough paths.}
Let $\mathbf{V},\mathbf{W}\in\mathcal{RP}^p(\mathbb{R}^d)$. 
Given two $p$-geometric rough paths $\mathbf{V}\equiv(V,\mathbb{V})$ and $\mathbf{W}\equiv(W,\mathbb{W})$, we can define the sum $\mathbf{V}+\mathbf{W}$ element-wise by
\begin{align}\label{eq:operation_plus_false}
\mathbf{V}+\mathbf{W}:=(V+W,\mathbb{V}+\mathbb{W})\,.
\end{align}
But then the sum of the two paths $\mathbf{V}+\mathbf{W}$ does not satisfy Chen's relation \eqref{chen-rela}, thus it is not a rough path in the sense of Definition~\ref{defi-rough-path}. 
 As a consequence the space $\mathcal{RP}^p(\mathbb{R}^d)$ is only a metric space and not a vectorial space. 
To get some intuition on what to expect as second iterated integral, given two smooth paths $X,Y$, we compute
\begin{align*}
\int_{s}^{t}(X+Y)_{s,r}\otimes d(X+Y)_r= \int_{s}^{t}X_{s,r}\otimes dX_r+\int_{s}^{t} X_{s,r}\otimes dY_r+\int_{s}^{t} Y_{s,r}\otimes dX_{r}+\int_{s}^{t}Y_{s,r}\otimes dY_r\,,
\end{align*}
where the integrals are all interpreted as Young's integrals. We see that, by defining $\mathbb{X}:=\int_{s}^{t}X_{s,r}\otimes dX_r$ and $\mathbb{Y}:=\int_{s}^{t}Y_{s,r}\otimes dY_r$, the second iterated integral in \eqref{eq:operation_plus_false} does not take into consideration the integrals
\begin{align*}
[XY]:=\int_{s}^{t} X_{s,r}\otimes dY_r\,, \quad [YX]:=\int_{s}^{t} Y_{s,r}\otimes dX_{s,r}\,,
\end{align*}
which we refer to as \textit{crossed integrals} or \textit{mixed integrals}, that appear naturally from the bilinearity of the integral map. As a consequence, we can define the sum of $(X,\mathbb{X})$ and $(Y,\mathbb{Y})$ as
\begin{align}\label{eq:operation_plus}
\{\mathbf{X}+\mathbf{Y}\}&:=(X+Y,\mathbb{X}+\mathbb{Y}+[XY]+[YX])\,,
\end{align}
which has the advantage of satisfying Chen's relation \eqref{chen-rela}. Of course, there is no information in $(X,\mathbb{X})$ and $(Y,\mathbb{Y})$ that includes the mixed integrals: in this example we have assumed to have in both cases Young's integration and therefore we know how to interpret the mixed integrals. 
For all $\epsilon>0$ the increment has the form
\begin{equation}\label{eq:increment_intro}
\begin{aligned}
\{\mathbf{X}+\tau_\epsilon\mathbf{Y}\}-\mathbf{X}&=(X+\epsilon Y-X,\mathbb{X}+\epsilon^2\mathbb{Y}+\epsilon[XY]+\epsilon[YX]-\mathbb{X})\\
&=(\epsilon Y,\epsilon^2\mathbb{Y}+\epsilon[XY]+\epsilon[YX])\,,
\end{aligned}
\end{equation}
where the minus operation is the element-wise difference.
We remark that the element-wise difference of rough paths is no more a rough path, since it fails to satisfy Chen's relation \eqref{chen-rela}. Thus also the increment in \eqref{eq:increment_intro} is not a rough path. Nevertheless, we will see this kind of structure appearing in the applications.

\paragraph{Random rough paths.}
We consider the rough paths appearing as a random variable on a probability space $(\Omega,\mathcal{F},\mathbb{P})$.
\begin{definition}
Let $(\Omega,\mathcal{F},\mathbb{P})$ be a probability space and $W$ be a stochastic process so that $W(\omega)\in \mathcal{V}^p(\mathbb{R}^d)$ $\mathbb{P}$-a.s. Assume that there exists a random variable $\omega \mapsto \mathbf{W}(\omega)$ so that $\mathbf{W}(\omega)\in \mathcal{RP}^p(\mathbb{R}^d)$  $\mathbb{P}$-a.s. Then the stochastic process $\mathbf{W}$ is a \textit{random geometric $p$-rough path}.
\end{definition}
We state the previous definition for more regular paths.
\begin{definition}
	 Let $(\Omega,\mathcal{F},\mathbb{P})$ be a probability space, $q\in [1,2)$ and $h$ be a stochastic process so that $h(\omega)\in \mathcal{V}^q(\mathbb{R}^d)$ $\mathbb{P}$-a.s. We call the map $\omega\rightarrow h(\omega)$ \textit{random $q$-rough path}.
\end{definition}
\paragraph{Deterministic and random compatible directions.}
\begin{definition}
	Let $(\Omega,\mathcal{F},\mathbb{P})$ be a probability space, $2<q\leq p$ and $\omega\mapsto \mathbf{V}(\omega)\equiv (V(\omega),\mathbb{V}(\omega))$, $\omega\mapsto \mathbf{W}(\omega)\equiv (W(\omega),\mathbb{W}(\omega))$ be two random geometric $p$-rough paths on $(\Omega,\mathcal{F},\mathbb{P})$ with values in $\mathbb{R}^d$. We say that $(\mathbf{V},\mathbf{W})$ are \textit{random $q$-compatible directions} on $(\Omega,\mathcal{F},\mathbb{P})$ if there exists a measurable subset $\Omega_{\mathbf{V},\mathbf{W}}\subset \Omega$ of full measure and a random geometric $q$-rough path $\omega\mapsto \mathbf{Z}\equiv(Z(\omega), \mathbb{Z}(\omega))$ such that $Z(\omega)=(V(\omega),W(\omega))$ and
		\begin{equation*}
			\mathbb{Z}(\omega)= \begin{pmatrix}
			\mathbb{V}_{s,t}(\omega) & \mathbb{L}_{s,t} (\omega)\\
			\mathbb{M}_{s,t}(\omega)& \mathbb{W}_{s,t}(\omega)
			\end{pmatrix}\,,
		\end{equation*}
	for all $\omega\in \Omega_{\mathbf{V},\mathbf{W}}$, where $\mathbb{M}(\omega),\mathbb{L}(\omega)$ are two index maps with values in the space of $d\times d$ matrices.
\end{definition}

\begin{definition}
	Let $2<q\leq p$ and $\mathbf{V}\equiv (V,\mathbb{V})$, $ \mathbf{W}\equiv (W,\mathbb{W})$ be two geometric $p$-rough paths with values in $\mathbb{R}^d$. We say that $(\mathbf{V},\mathbf{W})$ are \textit{deterministic $q$-compatible directions} if there exists a geometric $q$-rough path $\mathbf{Z}\equiv(Z, \mathbb{Z})$ such that $Z=(V,W)$ and
	\begin{equation*}
	\mathbb{Z}= \begin{pmatrix}
	\mathbb{V}_{s,t}& \mathbb{L}_{s,t} \\
	\mathbb{M}_{s,t}& \mathbb{W}_{s,t}
	\end{pmatrix}\,,
	\end{equation*}
	where $\mathbb{M},\mathbb{L}$ are two index maps with values in the space of $d\times d$ matrices.
\end{definition}

\begin{remark}
	\textbf{Existence and non uniqueness of a joint lift.} Given two $p$-geometric rough paths, it is always possible to construct a joint weakly geometric  $q$-rough path $\mathbf{Z}$, for some $2<q<p$. The existence is true in the more general framework of weakly geometric rough paths, as a consequence of the extension theorem of T.~Lyons, N.~B.~Victoir \cite{FrizLyonsExtension} (this construction is used in Z.~Qian, J.~Tudor \cite{Qian_Tudor}). In general, the lift is not unique: the matrices appearing on the diagonal of $\mathbb{Z}$ do not necessarily correspond to $\mathbb{V},\mathbb{W}$ and also the choice of the integrals $\mathbb{M},\mathbb{L}$ is not univocal.
\end{remark}
\begin{remark} \textbf{ $(0,\mathbf{W})$ is a random compatible direction for every rough path $\mathbf{W}$.}
	Assume that $\omega\mapsto\mathbf{W}(\omega)\equiv(W(\omega),\mathbb{W}(\omega))$ is a random $p$-geometric rough path. In this case, the joint lift $\omega\mapsto \mathbb{Z}(\omega)$ can be defined $\mathbb{P}$-a.s. by
	\begin{equation*}
	\begin{aligned}
	Z_{s,t}(\omega)\equiv(0,W_{s,t}(\omega)),\quad \mathbb{Z}_{s,t}(\omega)\equiv \begin{pmatrix}
	0 & 0 \\
	0& \mathbb{W}_{s,t}(\omega)
	\end{pmatrix}\,,
	\end{aligned}
	\end{equation*}
	and $(0,\mathbf{W})$ are random $p$-compatible directions. In particular the set of full measure depends only on the construction of the random $p$-geometric rough path $\mathbf{W}$.
\end{remark}

In the remaining part of this section, we furnish examples of random $q$-compatible directions, for which we can describe the second iterated integral of the joint lift $\mathbf{Z}$ in dependence on the original paths.
\subsection{Some examples of measurable compatible directions.}
\subsubsection{On complementary Young regularity.}
The increment in \eqref{eq:increment} is well defined, provided we can give a meaning to the crossed integrals: in the initial construction \eqref{eq:operation_plus}, every integral is intended as a Young's integral (indeed the paths are assumed to be smooth). However, given two rough paths $\mathbf{V},\mathbf{W}\in \mathcal{RP}^p(\mathbb{R}^d)$, it is not clear how to define $[VW]$ and $[WV]$.
The easiest way to obtain a good definition of $\{\cdot+\cdot\}$ is to consider $\mathbf{V},W$ of \textit{complementary Young regularity}.We refer to \cite{lejai_victorie}.
\begin{construction}\label{constr:Constr_p_q}
	Let $(\Omega,\mathcal{F},\mathbb{P})$ be a probability space, $p\in (2,3)$ and $q\geq 1$ so that $1/p+1/q>1$. Let $\omega\rightarrow \mathbf{V}(\omega)\equiv(V(\omega),\mathbb{V}(\omega))$ be an $\mathbb{R}^d$-valued random geometric $p$-rough path on  $(\Omega,\mathcal{F},\mathbb{P})$. Denote by $\Omega_{\mathbf{V}}$ the set of full measure such that $\mathbf{V}(\omega)$ is a geometric $p$-rough path for all $\omega\in \Omega_{\mathbf{V}}$.
	Let $h$ be an $\mathbb{R}^d$-valued random geometric $q$-rough path on $(\Omega,\mathcal{F},\mathbb{P})$, so that for every $\omega\in \Omega_{\mathbf{V}}$,  $h(\omega)\in \mathcal{V}^q(\mathbb{R}^d)$.
	
	Denote the Young integral by $\mathcal{I}^{Y}(\cdot,\cdot)$ defined as in \eqref{eq:def_Young_integral} and denote by
	\begin{align*}
	[Vh](\omega):=\mathcal{I}^{Y}(V(\omega),h(\omega)) \,, \quad [hV]:= \mathcal{I}^{Y}(h(\omega),V(\omega))\,,\quad [hh](\omega):=\mathcal{I}^{Y}(h(\omega),h(\omega)) \,.
	\end{align*}
	Then for all $\omega \in \Omega_{\mathbf{V}}$ there exists the joint lift
	$\omega\mapsto \mathbf{Z}(\omega)=(Z(\omega),\mathbb{Z}(\omega))$, where 
	\begin{equation*}
	\begin{aligned}
	Z_{s,t}(\omega)\equiv(V_{s,t}(\omega),h_{s,t}(\omega)),\quad \mathbb{Z}_{s,t}(\omega)\equiv \begin{pmatrix}
	\mathbb{V}_{s,t}(\omega) & [Vh]_{s,t} (\omega)\\
	[hV]_{s,t}(\omega)& [hh]_{s,t}(\omega)
	\end{pmatrix}\,.
	\end{aligned}
	\end{equation*}
		In particular $\omega\mapsto\mathbf{Z}(\omega)$ is a random geometric $p$-rough path on $(\Omega,\mathcal{F},\mathbb{P})$ and $(\mathbf{V},h)$ are random $p$-compatible directions.
\end{construction}
In addition to having an explicit representation of $\mathbb{Z}$ in terms of $V,h$, the following continuity property holds (in the notations of Construction \ref{constr:Constr_p_q}): for all $\omega \in \Omega_{\mathbf{V}}$, 
\begin{align*}
	\mathcal{V}^q(\mathbb{R}^d)&\longrightarrow \mathcal{RP}^p(\mathbb{R}^d)\\
	h(\omega )&\mapsto \mathbf{Z}(\omega)
\end{align*}
is locally Lipschitz continuous. 

\subsubsection{Compatible directions of Brownian motions beyond the Cameron-Martin directions. }
We look at other directions beyond stochastic processes whose paths are of complementary Young regularity. We are interested in compatible directions of the Brownian motion: we can fix the Stratonovich integration theory and look at semi-martingales.
The problem of considering the joint lift of two semi-martingales to a rough path has been addressed widely in the literature. We refer the reader to L.~Coutin, A.~Lejay \cite{coutin_lejai_semiMTG}, I.~Chevyrev, P.~K.~Friz \cite{chevyrev_friz} (for semi-martingales with jumps) and we recall some properties of the joint lift for continuous semi-martingales in the following construction (see also  \cite{Diehl_Ober_Riedel}). 
\begin{construction}\label{constr:constr_at_level} (\textbf{A construction for Brownian motion with respect to Stratonovich integration}) \\
	Let $\mathcal{M}([0,T];\mathbb{R}^d)$ be the set of $\mathbb{R}^d$-valued continuous $(\mathcal{F}_t)_{t\in [0,T]}$-semimartingales on $(\Omega, \mathcal{F},(\mathcal{F}_t)_t,\mathbb{P})$. For all $N,M\in \mathcal{M}$ we denote by $\mathcal{I}^\mathrm{Str}(N,M)$ the Stratonovich integral of $N$ against $M$. 
	Let $B,M\in\mathcal{M}([0,T];\mathbb{R}^d)$, where $B$ is an $(\mathcal{F}_t)_t$-Brownian motion. Assume further that $B$ and $M$ are independent.
	\begin{itemize}
			\item $\mathbf{B}\equiv(B,\mathbb{B})$ is a random $p$-geometric rough path, where the (random) second iterated integral $\mathbb{B}:=  \mathcal{I}^{\mathrm{Str}}(B,B)$.
	
			\item $\mathbf{M}\equiv(M,\mathbb{M})$ is a random $p$-geometric rough path, where $\mathbb{M}:=  \mathcal{I}^{\mathrm{Str}}(M,M)$.
	
			\item for $\mathbb{P}$-a.e. $\omega\in\Omega_{B,M}$, the couple $\mathbf{Z}(\omega)=(Z(\omega),\mathbb{Z}(\omega))$ defined by
			\begin{equation}\label{eq:Z,boldZ}
			\begin{aligned}
			Z_{s,t}(\omega)\equiv(B_{s,t}(\omega),M_{s,t}(\omega)),\quad \mathbb{Z}_{s,t}(\omega)\equiv \begin{pmatrix}
			\mathbb{B}_{s,t}(\omega) & [BM]_{s,t} (\omega)\\
			[MB]_{s,t}(\omega)& \mathbb{M}_{s,t}(\omega)
			\end{pmatrix}\,,
			\end{aligned}
			\end{equation}
			is a continuous geometric $q$-rough path for $q\leq p$,
			where 
			\begin{align*}
			 [BM]:=\mathcal{I}^{\mathrm{Str}}(B,M) \,, \quad [MB]:= \mathcal{I}^{\mathrm{Str}}(M,B)\,,
			\end{align*}
			\item $Z$ and $\mathbb{Z}$ are stochastic processes adapted to $(\mathcal{F}_t)_t$.
	\end{itemize}
\end{construction}
With the previous Construction \ref{constr:constr_at_level}, it is possible to construct joint lifts of Brownian motion and other semi-martingales, which allow to differentiate in other directions beyond the complementary Young regularity paths. On the other hand, no continuity property of the joint lift $\mathbf{Z}$ in terms of $M,B$ can be established in general.
\begin{remark}
\textit{On controlled rough paths.} For the reader already acquainted with rough paths theory, by considering random controlled rough paths, it is possible to establish continuity properties of the integrator with respect to the integrand pathwise or with respect to the $\mathcal{L}^q(\Omega)$- norm (see e.g. \cite{Friz_Hocquet_Le}, \cite{FrizKranichranich}). 
\end{remark}

We present another example of joint lift of paths, usually seen in filtering problems.
\begin{example} \textit{Brownian motion and a process which is a deterministic rough path.}
	In \cite{Diehl_Ober_Riedel}, the authors introduce a random geometric rough path which is the joint lift of a Brownian motion $B$ and a constant stochastic process $\eta$, such that for all $\omega \in \Omega $ the image $\eta(\omega)$ can be lifted to a continuous geometric $p$-rough path. In this construction is that the joint lift has the form \eqref{eq:Z,boldZ} with respect to the Stratonovich integration. This case follows under Construction \ref{constr:constr_at_level} only if the second iterated integral of $\mathbf{\eta}$ can be constructed as the limit of the Stratonovich sums and the convergence occurs in probability: this is the case, indeed the Stratonovich integration and the It\^o integration coincide from the independence of the two processes. The joint limit is a random $q$-geometric rough path (with $2<q\leq p$), but this has been improved as showed in \cite{FrizKranichranich}: the authors show that the lift is a random $p$-geometric rough path.
	This example includes also the case of deterministic Cameron-Martin directions.
\end{example}
The examples discussed in this section are only a small and not exhaustive blink on the possibilities to build compatible directions that involve Brownian motion. 
\subsection{Pathwise directional derivative.}\label{sec:path_directional_derivative} 
The aim of this section is to introduce a notion of derivative for a map $\Phi:\mathcal{RP}^p(\mathbb{R}^d)\rightarrow \mathcal{Y}$, where $\mathcal{Y}$ is  Banach space. The space $\mathcal{RP}^p(\mathbb{R}^d)$, however, is not a Banach space (since it is not a vectorial space) and we can not employ the classical Fréchet derivative or Gateaux derivative for this purpose. The problem has already been addressed in \cite{Qian_Tudor}, where the authors furnish a geometric perspective on the differentiation for the function $\Phi$. We partly follow the considerations in \cite{Qian_Tudor}. 
Let $(\mathbf{X},\mathbf{Y})$ be a random $p$-compatible directions with joint lift $\mathbf{Z}$. Then we can define $\{\mathbf{X}+\mathbf{Y}\}$ as random $q$-rough path for some $2<q\leq p$.  
\paragraph{Pathwise directional derivative.} 
In Definition~\ref{def:def_derivative_quian_tudor} we define a notion of directional derivative:
\begin{definition}\label{def:def_derivative_quian_tudor}(\textit{Pathwise directional derivative})
	Let $p\in[2,3)$ and $(\mathbf{X},\mathbf{Y})$ be deterministic compatible directions with joint lift $\mathbf{Z}$. We say that a map $\Phi:\mathcal{RD}^p(\mathbb{R}^d)\rightarrow \mathcal{Y}$ is \textit{pathwise differentiable} in $\mathbf{X}$ in the direction $\mathbf{Y}$ if for $\epsilon\rightarrow 0$, the limit 
	\begin{align*}
	D\Phi[\mathbf{X}](\mathbf{Y}):=\lim_{\epsilon\rightarrow 0} \frac{ \Phi(\{\mathbf{X}+\tau_{\epsilon} \mathbf{Y}\})-\Phi(\mathbf{X})}{\epsilon}\, 
	\end{align*}
	exists in $\mathcal{Y}$.	
\end{definition}
\begin{remark}
	In comparison to Malliavin's calculus, where the differentiation directions are the ones belonging to the Cameron-Martin space, we can consider less regular directions (\cite{Qian_Tudor}). For directions of complementary Young regularity, this definition of derivative coincides with the notion of $\mathcal{H}$-differentiability of the flow.
\end{remark}

As an example, we compute the pathwise directional derivative of  both of the iterated integrals of a rough path $(W,\mathbb{W})$. 
\begin{example}
	Let $(\mathbf{G},\mathbf{H})$ be deterministic $p$-compatible directions with joint lift $\mathbf{Z}$. Consider the projection map $\pi_1:\mathbf{W}\equiv(W,\mathbb{W})\mapsto W$, where $W\in \mathcal{V}^p_2(\mathbb{R}^d)$. Then
	it is immediate that $D \pi_1[\mathbf{H}](\mathbf{G}) =G$.
	Consider now $\pi_2:\mathbf{W}\mapsto \mathbb{W}$, where $\mathbb{W}\in \mathcal{V}^{p/2}_2(\mathbb{R}^{d\times d})$. The pathwise directional derivative of $\pi_2$ in $\mathbf{H}$ in the direction $\mathbf{G}$ has the form, for all $s\leq t\in [0,T]$ 
	\begin{align*}
	D\pi_2[\mathbf{H}](\mathbf{G})_{s,t}=\lim_{\epsilon\rightarrow 0}\frac{ [HH]_{s,t}+\epsilon[GH]_{s,t}+\epsilon[HG]_{s,t}+\epsilon^2[GG]_{s,t}-[HH]_{s,t}}{\epsilon}=[GH]_{s,t}+[HG]_{s,t}\,.
	\end{align*}
\end{example}

\section{Central limit theorem.}\label{sec:CLT_theoric}
Fix $\epsilon>0$ and let $(\mathbf{H},\mathbf{G})$ be compatible directions with joint lift $\mathbf{Z}$. Let $\mathcal{X}$ be a Banach space and $\Phi:\mathcal{RP}^p(\mathbb{R}^d)\rightarrow C([0,T];\mathcal{X})$ be the continuous It\^o-Lyons map associated to the unique solution to a SPDE. We interpret the CLT as the convergence for $\epsilon \rightarrow 0$ of the sequence of random variables $(X^\epsilon)_{\epsilon>0}$, given by
\begin{equation}\label{eq:Y_epsilon}
X^\epsilon
:=\frac{\Phi(\{\mathbf{H}+\tau_{\sqrt{\epsilon}}\mathbf{G}\})-\Phi(\mathbf{H})}{\sqrt{\epsilon}}\,,
\end{equation}
to the pathwise directional derivative $D\Phi[\mathbf{H}](\mathbf{G})$.

We apply these considerations to the stochastic heat equation in $d=1,2,3$, to stochastic reaction-diffusion equation in $d=1,2,3$ and to the stochastic Landau-Lifschitz-Gilbert equation on $\mathbb{T}$. The procedure is general and can be applied to other equations as well.

\subsection{Central limit theorem for the heat  equations.}\label{sec:CLT_SPDEs}
We look at the CLT for the heat equation. 
We preliminarily introduce some Hilbert spaces. We consider the physical dimensions $d=1,2,3$, $n\in \mathbb{N}$ and denote by $H^1:=H^1(\mathbb{T}^d;\mathbb{R}^n)$, $H^2=H^2(\mathbb{T}^d;\mathbb{R}^n)$ and $E:=W^{2,\infty}(\mathbb{T}^d;\mathbb{R}^n)$. 
Consider the drift $b:H^2\rightarrow L^2$ defined by $b(v):=\Delta v$.
We want to study existence and uniqueness of a solution $u$ to equation
\begin{equation}\label{eq:heat_equation}
\delta u_{s,t}=\int_{s}^{t} b(u_r)\dd r+G_{s,t}u_s+\mathbb{G}_{s,t}u_s+u^\natural_{s,t}\,,
\end{equation}
where $u_0=u^0\in H$ and $\mathbf{G}\equiv (G,\mathbb{G})$ is a random geometric $p$-rough driver as in the following Definition~\ref{def:rough_driver}.
\begin{definition}\label{def:rough_driver}
Let $p\in [2,3)$. A pair of two-index maps $(G,\mathbb{G})$ is  called a continuous $p$-rough driver if 
\begin{align*}
\mathbf{G}\equiv (G,\mathbb{G}) \in \mathcal{V}_2^{p}(\mathcal{L}(E))\times \mathcal{V}_2^{p/2}(\mathcal{L}(E))\,,
\end{align*}
and the Chen's relation holds, i.e. for all $s\leq r\leq t\in [0,T]$
\begin{align*}
\delta G_{s,r,t}=0\,,\quad\quad \delta \mathbb{G}_{s,r,t}= G_{r,t}G_{s,r} \,.
\end{align*}
If it also holds that for all $s\leq t\in [0,T]$
\begin{align*}
\mathrm{Sym}(\mathbb{G}_{s,t})=\frac{1}{2}G_{s,t}G_{s,t}\,,
\end{align*}
then $\mathbf{G}$ is a geometric $p$-rough driver.
\end{definition}
In this Section~\ref{sec:CLT_SPDEs} we consider the space of rough drivers with spatial component
\begin{align*}
\mathcal{RD}^p(\mathbb{R}^d,W^{k,\infty}) :=\{\mathbf{W}=\tau_{g(x)}\mathbf{X}: g\in W^{k,\infty}(\mathbb{T}^d;\mathbb{R}),\quad \mathbf{X}\in \mathcal{V}_2^{p}(\mathcal{L}(\mathbb{R}^d))\times \mathcal{V}_2^{p/2}(\mathcal{L}(\mathbb{R}^d))\}\,,
\end{align*} 
for a fixed $k\in \mathbb{N}_0$.
 The definition of random geometric $p$-rough driver follows analogously to the definition of random geometric $p$-rough path. Notice that the definition of increment with respect to the time component of the noise can be defined analogously.
\subsubsection{Wong-Zakai convergence, existence and uniqueness of a strong solution in $L^\infty(H^1)\cap L^2(H^2)\cap \mathcal{V}^p(L^2)$.}
In Appendix~\ref{Appendix_B_new}, we show the well posedness of the heat equation, where a solution is intended in the sense of Definition~\ref{def:sol_rough}.
\begin{definition}\label{def:sol_rough}
	Let $\mathbf{G}\equiv (G,\mathbb{G})\in \mathcal{RD}^p(\mathbb{R}^n;E) $ and  $p\in [2,3)$. We say that $u\in L^\infty(H^1)\cap L^2(H^2)\cap \mathcal{V}^p(L^2)$ is a solution to \eqref{eq:heat_equation} with initial condition $u^0\in H^1$ if there exists a two index map $u^\natural\in \mathcal{V}^{p/3}_2(L^2)$ defined implicitly by the equality in $L^2$
	\begin{equation}\label{eq:equazione_heat_def_1}
	\delta u_{s,t}=\int_{s}^{t}b(u_r)\dd r+G_{s,t}u_s+\mathbb{G}_{s,t}u_s+u^\natural_{s,t}\,,
	\end{equation}
	with $u_0=u^0$ and so that the energy inequality holds
	\begin{align*}
	\|u\|^2_{L^\infty(H^1)}+\|u\|^2_{L^2(H^2)}\lesssim_{\mathbf{G}}\|u^0\|^{2}_{H^1}\,.
	\end{align*}
\end{definition}
We denote by $u=\pi_1(u^0,\mathbf{G})$ a solution to \eqref{eq:equazione_heat_def_1} in the sense of Definition~\ref{def:sol_rough} started in $u^0$ and driven by $\mathbf{G}$. The results in Appendix~\ref{Appendix_B_new}
are summarised in Proposition~\ref{pro:lip_semilinear_1}.
\begin{proposition}\label{pro:lip_semilinear_1}
	Let $\mathbf{G}\in \mathcal{RD}^p(\mathbb{R}^n;E)$ and $u^0\in H^1$. There exists a unique solution $u=\pi_1(u^0,\mathbf{G})$ to \eqref{eq:heat_equation} in the sense of Definition~\ref{def:sol_rough}. Moreover, the It\^o-Lyons map associated to the unique solution $u$ to \eqref{eq:heat_equation} 
	\begin{align*}
	\Phi_1:\quad\mathcal{RD}^p(\mathbb{R}^n;E)&\longrightarrow  L^\infty(H^1)\cap L^2(H^2) \cap \mathcal{V}^p(L^2)\\
	\mathbf{G}&\longmapsto u=\pi_1(u^0;\mathbf{G})\, 
	\end{align*}
	is locally Lipschitz continuous, in the sense that for all $\mathbf{G}, \mathbf{H}\in \mathcal{RD}^p(\mathbb{R}^n;E)$ there exists a constant $C\equiv C(\|u^0\|_{H},\mathbf{G},\mathbf{H},T)>0$ so that 
	\begin{align*}
	\|\Phi_1(\mathbf{G})-\Phi_1(\mathbf{H})\|_{L^\infty(H^1)\cap L^2(H^2) \cap \mathcal{V}^p(L^2)}\lesssim C\,\rho(\mathbf{G},\mathbf{H})\,.
	\end{align*}
\end{proposition}

\subsubsection{Central limit theorem for the heat equations.}
We are in the setting of the previous section.
Note that $b'(y)z=b(z)$ for all $y,z\in H^1$, where the derivative is intended in the Frechét sense. We introduce now a definition of solution needed for the CLT.
\begin{definition}\label{def:sol_additive}
	Let $(\mathbf{G},\mathbf{W})$ be random $p$-compatible directions with joint lift $\mathbf{Z}$ (where the directions are rough drivers). Let $u^0\in H^1$ and $u\in L^\infty(H^1)\cap L^2(H^2) \cap\mathcal{V}^p(L^2)$ be the unique solution in the sense of Definition~\ref{def:sol_rough} to
	\begin{align*}
	\delta u_{s,t}=\int_{s}^{t}b(u_r)\dd r+G_{s,t}u_s+\mathbb{G}_{s,t}u_s+u^{\natural}_{s,t}
	\end{align*}
	in $L^2$ and such that $u_0=u^0$.
	We say that $X\in L^\infty(H^1)\cap L^2(H^2)\cap \mathcal{V}^p(L^2)$ is a solution to \eqref{equazione_r_1} with initial condition $X_0=0$ if there exists a two index map $X^\natural\in \mathcal{V}^{p/3}_2(L^2)$ defined implicitly by the equality in $L^2$
	\begin{equation}\label{equazione_r_1}
	\delta X_{s,t}=\int_{s}^{t}b'(u_r)X_r\dd r+G_{s,t}X_s+\mathbb{G}_{s,t}X_s+([WG]+[GW])_{s,t}u_s+W_{s,t}u_s+X^\natural_{s,t}\,.
	\end{equation}
\end{definition}
Consider now the sequence $(X^\epsilon)_{\epsilon>0}$ in $\mathcal{Y}:=L^\infty(H^1)\cap L^2(H^2)\cap\mathcal{V}^p(L^2)$, defined by
\begin{align*}\label{eq:x_epsilon}
X^\epsilon
:=\frac{u^\epsilon-u}{\sqrt{\epsilon}}\,,
\end{align*}
where $u^\epsilon=\pi_1(u^0,\{\mathbf{G}+\tau_{\sqrt{\epsilon}}\mathbf{W}\})$ and $u=\pi_1(u^0,\mathbf{G})$.
Recall that $X^\epsilon$ fulfils
\begin{align*}
\delta X^\epsilon_{s,t}&=\int_{s}^{t} \frac{b(u^\epsilon_r)-b(u_r)}{\sqrt{\epsilon}} \dd r+\frac{(G_{s,t}+\sqrt{\epsilon}W_{s,t})u^\epsilon_s-G_{s,t}u_s}{\sqrt{\epsilon}}\\
&\quad+\frac{(\epsilon\mathbb{W}_{s,t}+\sqrt{\epsilon}[WG]_{s,t}+\sqrt{\epsilon}[GW]_{s,t}+\mathbb{G}_{s,t})u^\epsilon_s-\mathbb{G}u_s}{\sqrt{\epsilon}}+ X^{\epsilon,\natural}_{s,t}\,.
\end{align*}
We state now the main result of the section. The proof follows from techniques analogous to Proposition~\ref{pro:lip_semilinear}.
\begin{theorem}\label{teo:CLT}
	(Central limit theorem for strong solutions to the heat equation). Let $u^0\in H^1$ and let $(\mathbf{G},\mathbf{W})$ be compatible directions with joint lift $\mathbf{Z}$ (where the directions are rough drivers). Let $u=\pi_1(u^0,\mathbf{G})$ be the unique solution in the sense of Definition~\ref{def:sol_rough} to
	\begin{align*}
	\delta u_{s,t}=\int_{s}^{t}b(u_r)\dd r +G_{s,t} u_s+\mathbb{G}_{s,t} u_s+u^{\natural,\mathbf{G}}_{s,t}
	\end{align*}
	in the space $L^\infty(H^1)\cap L^2(H^2)\cap \mathcal{V}^p(L^2)$, where $b$ follows the assumptions of this section.
	Denote by $\Phi_1(\mathbf{G})=\pi_1(u^0,\mathbf{G})$.  Then the sequence $(X^\epsilon)_\epsilon$, defined as
	\begin{align*}
	X^\epsilon=\frac{\Phi_1(\{\mathbf{G}+\tau_{\sqrt{\epsilon}}\mathbf{W}\})-\Phi_1(\mathbf{G})}{\sqrt{\epsilon}}\,,
	\end{align*}
	converges strongly in $L^\infty(H^1)\cap L^2(H^2)$ for $\epsilon \rightarrow 0$ to a limit $X$.  In particular
	\begin{align*}
	X=D\Phi_1[\mathbf{G}](\mathbf{W})\,,
	\end{align*}
	i.e. $X$ is the pathwise directional derivative of the It\^o-Lyons map $\Phi_1$ in $\mathbf{G}$ in the direction $\mathbf{W}$. Moreover $X$ is the unique solution to \eqref{equazione_r_1}  
	in the sense of Definition~\ref{def:sol_additive}.
	In addition, the convergence occurs with optimal speed $\sqrt{\epsilon}$, i.e.
	\begin{align*}
	\sup_{0\leq t\leq T}\|X^\epsilon_t-D\Phi_1[\mathbf{G}](\mathbf{W})_t\|^2_{H^1}+\int_{0}^{T}\|X^\epsilon_r-D\Phi_1[\mathbf{G}](\mathbf{W})_r\|^2_{H^2}\dd r\lesssim \sqrt{\epsilon}\omega_{\mathbf{Z}}^{1/p}\,,
	\end{align*}
	where $\omega_{\mathbf{Z}}:=\|\mathbf{Z}\|^p_{\mathcal{RD}^p(W^{2,\infty})}$ 
\end{theorem}
\begin{remark}\label{rem:linearity_map}
\textbf{Linear dependence of $X$ on $W$.}
Observe that the map $\mathbf{W}\equiv(W,\mathbb{W})\mapsto X$ depends only on the first iterated integral $W$. Let $(\mathbf{G},\mathbf{W}^1)$ and $(\mathbf{G},\mathbf{W}^2)$ be two compatible directions such that $\alpha [G W^1]+\beta[G W^2]=[G(\alpha W^1+\beta W^2)] $ and $\alpha[W^1G]+\beta[W^2G]=[(\alpha W^1+\beta W^2)G]$, for all $\alpha,\beta \in \mathbb{R}$. In view of the independence of the limit equation on the second interated integral, we can denote by $D\Phi_1[\mathbf{G}](W^1)$ and by $D\Phi_1[\mathbf{G}](W^2)$ the respective solutions to \eqref{equazione_r_1}. From the linearity of the drift and the bi-linearity of the mixed integrals, it follows that $\alpha D\Phi_1[\mathbf{G}](W^1)+\beta D\Phi_1[\mathbf{G}](W^2)=D\Phi_1[\mathbf{G}](\alpha W^1+\beta W^2)$ for all $\alpha,\beta\in \mathbb{R}$, namely that the map $W\mapsto D\Phi_1[\mathbf{G}](\mathbf{W}) $ is linear.  
\end{remark}
\begin{remark}
\label{remark:ito_CLT} \textbf{On the pathwise convergence of the heat equation driven by an It\^o  integral to the CLT limit.}
This remark is a direct consequence of Remark~\ref{rem:linearity_map}.
It is usual in the classical theory of CLT for SPDEs to consider equations driven by an It\^o-type multiplicative noise. The previous analysis of convergence to the CLT is pursued for a Stratonovich-type integral, but we can infer the pathwise convergence of the increments to the CLT in the case of It\^o stochastic integral. Indeed, if $W$ is a Stratonovich lift of the Brownian motion, the pathwise CLT $D\Phi[\mathbf{0}](W)$ does not depend on the second iterated integral. Hence the limit is the same for every other geometric lift of the Brownian motion as well as for the It\^o's integral (this is due to the fact that the It\^o-Stratonovich correction is quadratic in $\sqrt{\epsilon}$).
\end{remark}

\begin{proof} (\textit{of Theorem \ref{teo:CLT}})
	For the existence and uniqueness of a unique solution $X$ to the limit equation, we refer to Appendix~\ref{Appendix_B} (note that the assumptions on the drift are verified since $b'(u)X=b(X)$).
	From Proposition~\ref{pro:lip_semilinear}, the map $\Phi_1$ is locally Lipschitz continuous and the sequence $(X^\epsilon)_{\epsilon}$ is uniformly bounded in $L^\infty(H^1)\cap L^2(H^2)\cap \mathcal{V}^p(L^2)$. From Lemma~\ref{lemma:embedding_tornstein}, there exists a subsequence converging strongly in $L^2(H^1)\cap C([0,T];L^2)$ to a limit $X$. We prove the pathwise convergence in  $L^\infty(H^1)\cap L^2(H^2)$ for $\epsilon \rightarrow 0$. What follows holds for all $s\leq r\leq t\in [0,T]$. We consider the equation $Z^\epsilon:=X^\epsilon-X$ and the squared equation from the product formula in Proposition~\ref{pro:product}, 
	\begin{equation}\label{eq:X_eps_menus_X_k_square_semilinear}
		\begin{aligned}
		\delta\big(\partial^k_x Z^\epsilon\otimes \partial^k_x Z^\epsilon\big)_{s,t}
		=\mathcal{D}(k,k)_{s,t}+I_{s,t}+\mathbb{I}_{s,t}+\tilde{\mathbb{I}}_{s,t}+\big(\partial^k_x Z^\epsilon\big)^{\otimes 2,\natural}_{s,t}\,,
		\end{aligned}
	\end{equation}
	where for $k\in \{0,1\}$ we denote 
	\begin{align*}
	\mathcal{D}_{s,t}(k,k)&:=\int_s^t \partial_{x}^k\bigg(\frac{b(u^\epsilon_r)-b(u_r)}{\sqrt{\epsilon}} - b'(u_r)X_r \bigg)\odot\partial_{x}^k(X^\epsilon_r-X_r) \dd r\,,
	\end{align*}
	\begin{equation*}
	\begin{aligned}
	\textcolor{black}{I}_{s,t}&:=\sum_{n=0}^{k} \binom{k}{n} [\sqrt{\epsilon}\partial_{x}^{k-n}W_{s,t}\partial_{x}^{n} X^\epsilon_s+\partial_{x}^{k-n}G_{s,t}\partial_{x}^{n} (X^\epsilon_s-X_s)]\odot \partial_{x}^k(X^\epsilon_s-X_s)\,,
	\end{aligned}
	\end{equation*}
	\begin{equation*}
	\begin{aligned}
	\mathbb{I}_{s,t}&:=\sum_{n=0}^{k} \binom{k}{n}(\partial_{x}^{k-n}\sqrt{\epsilon}\mathbb{W}_{s,t}\partial_{x}^{n} u^\epsilon_s)\odot\partial_{x}^k(X^\epsilon_s-X_s)\\
	&\quad+\sqrt{\epsilon}\sum_{n=0}^{k} \binom{k}{n}\left[\partial_{x}^{k-n}[[GW]+[WG]]_{s,t}\partial_{x}^{n} X^\epsilon\right]\odot\partial_{x}^k(X^\epsilon_s-X_s)\\
	&\quad+\sum_{n=0}^{k} \binom{k}{n}\partial_{x}^{k-n}\mathbb{G}_{s,t}\partial_{x}^{n} Z^\epsilon_s\odot\partial_{x}^k(X^\epsilon_s-X_s)=:\mathbb{I}^1_{s,t}+\mathbb{I}^2_{s,t}+\mathbb{I}^3_{s,t}\,,
	\end{aligned}
	\end{equation*}

	\begin{equation*}
	\begin{aligned}
	\tilde{\mathbb{I}}_{s,t}:=\sum_{n=0}^{k} \binom{k}{n} [\sqrt{\epsilon}\partial_{x}^{k-n}W_{s,t}\partial_{x}^{n} &X^\epsilon_s+\partial_{x}^{k-n}G_{s,t}\partial_{x}^{n} (X^\epsilon_s-X_s)]\\
		&\otimes\sum_{m=0}^{k} \binom{k}{m}[\sqrt{\epsilon}\partial_{x}^{k-m}W_{s,t}\partial_{x}^{m} X^\epsilon_s+\partial_{x}^{k-m}G_{s,t}\partial_{x}^{m} (X^\epsilon_s-X_s)]\,.
	\end{aligned}
	\end{equation*}
	We are again interested in estimating the remainder term via the sewing Lemma~\ref{lemma_sewing}: we rewrite $\delta \big(\partial^k_x Z^\epsilon\big)^{\otimes 2,\natural}$ as a linear combination of elements of at least $p/3$-variation,
	\begin{align*}
		\delta \big(\partial^k_x Z^\epsilon\big)^{\otimes 2,\natural}_{s,r,t}=-\delta I_{s,r,t}-\delta \mathbb{I}_{s,r,t}-\delta \tilde{\mathbb{I}}_{s,r,t}\,,
	\end{align*}
	where $\delta I_{s,r,t}$ is given by
	\begin{align*}
		\delta I_{s,r,t}=\sum_{n=0}^{k} \binom{k}{n} [-\sqrt{\epsilon}\partial_{x}^{k-n}W_{r,t}\delta(\partial_{x}^{n} X^\epsilon\odot \partial_{x}^kZ^\epsilon)_{s,r}-\partial_{x}^{k-n}G_{r,t}\delta(\partial_{x}^{n} Z^\epsilon\odot \partial_{x}^kZ^\epsilon)_{s,r}]\,.
	\end{align*}
	We estimate $\delta\mathbb{I}=\delta\mathbb{I}^1+\delta \mathbb{I}^2+\delta\mathbb{I}^3$ separately for each addend. We write explicitly $\delta\mathbb{I}^1_{s,r,t}$ as
	\begin{align*}
		\delta\mathbb{I}^1_{s,r,t}&= \sum_{n=0}^{k} \binom{k}{n}\sqrt{\epsilon}\partial_{x}^{k-n}\delta\mathbb{W}_{s,r,t}\partial_{x}^{n} u^\epsilon_s\odot\partial_{x}^kZ^\epsilon_s-\sum_{n=0}^{k} \binom{k}{n}\partial_{x}^{k-n}\sqrt{\epsilon}\mathbb{W}_{r,t}\delta(\partial_{x}^{n} u^\epsilon\odot\partial_{x}^kZ^\epsilon)_{s,r}\,.
	\end{align*}
	We rewrite $\delta\mathbb{I}^2_{s,r,t}$ as
	\begin{align*}
	\delta \mathbb{I}^2_{s,r,t}&=\sqrt{\epsilon}\sum_{n=0}^{k} \binom{k}{n}\partial_{x}^{k-n}\delta([GW]+[WG])_{s,r,t}\partial_{x}^{n} X^\epsilon_s\odot\partial_{x}^kZ^\epsilon_s\\
	&\quad-\sqrt{\epsilon}\sum_{n=0}^{k} \binom{k}{n}\partial_{x}^{k-n}[[GW]+[WG]]_{r,t}\delta(\partial_{x}^{n} X^\epsilon\odot\partial_{x}^kZ^\epsilon)_{s,r}\,,
	\end{align*}
	and finally $\delta\mathbb{I}^3_{s,r,t}$ becomes
	\begin{align*}
	\delta\mathbb{I}^3_{s,r,t}&=\sum_{n=0}^{k} \binom{k}{n}\partial_{x}^{k-n}\delta[\mathbb{G}]_{s,r,t}\partial_{x}^{n} Z^\epsilon_s\odot\partial_{x}^kZ^\epsilon_s-\sum_{n=0}^{k} \binom{k}{n}\partial_{x}^{k-n}[\mathbb{G}]_{r,t}\delta(\partial_{x}^{n} Z^\epsilon\odot\partial_{x}^kZ^\epsilon)_{s,r}\,.
	\end{align*}
	We rewrite now $\delta\tilde{\mathbb{I}}$ as 
	\begin{align*}
	\delta\tilde{\mathbb{I}}=\delta \sum_{n=0}^{k} \sum_{m=0}^{k} \binom{k}{n}\binom{k}{m} [\sqrt{\epsilon}\partial_{x}^{k-n}W_{s,t}\partial_{x}^{n} &X^\epsilon_s+\partial_{x}^{k-n}G_{s,t}\partial_{x}^{n} Z^\epsilon_s]\\
	&\otimes[\sqrt{\epsilon}\partial_{x}^{k-m}W_{s,t}\partial_{x}^{m} X^\epsilon_s+\partial_{x}^{k-m}G_{s,t}\partial_{x}^{m} Z^\epsilon_s]\,.
	\end{align*}
	From the linearity of $\delta(\cdot )_{s,r,t}$, we can look only at the argument of the sum and we split it into two addends
	\begin{align*}
	& [\sqrt{\epsilon}\partial_{x}^{k-n}W_{s,t}\partial_{x}^{n} X^\epsilon_s+\partial_{x}^{k-n}G_{s,t}\partial_{x}^{n} Z^\epsilon_s]\otimes \sqrt{\epsilon}\partial_{x}^{k-m}W_{s,t}\partial_{x}^{m} X^\epsilon_s \\
	 &+[\sqrt{\epsilon}\partial_{x}^{k-n}W_{s,t}\partial_{x}^{n} X^\epsilon_s+\partial_{x}^{k-n}G_{s,t}\partial_{x}^{n} Z^\epsilon_s]\otimes \partial_{x}^{k-m}G_{s,t}\partial_{x}^{m} Z^\epsilon_s =:\Gamma^1_{s,t}+\Gamma^2_{s,t} \, .
	\end{align*}
	 By recalling  the algebraic relation \eqref{eq;algebraic_relation}, we rewrite $\delta \Gamma^1_{s,r,t}$ as
	 \begin{align*}
	 \delta \Gamma^1_{s,r,t}&=\epsilon[\partial_{x}^{k-n}W_{s,r}\partial_{x}^{n} X^\epsilon_s\otimes \partial_{x}^{k-m}W_{r,t}\partial_{x}^{m}X^\epsilon_s+\partial_{x}^{k-n}W_{r,t}\partial_{x}^{n} X^\epsilon_s\otimes \partial_{x}^{k-m}W_{s,r}\partial_{x}^{m}X^\epsilon_s\\
	 &\quad\quad+\partial_{x}^{k-n}W_{r,t}\partial_{x}^{n} X^\epsilon_s\otimes \partial_{x}^{k-m}W_{r,t}\partial_{x}^{m}X^\epsilon_s-\partial_{x}^{k-n}W_{r,t}\partial_{x}^{n} X^\epsilon_r\otimes \partial_{x}^{k-m}W_{r,t}\partial_{x}^{m}X^\epsilon_r]\\
	 &\quad +\sqrt{\epsilon}[\partial_{x}^{k-n}G_{s,r}\partial_{x}^{n} Z^\epsilon_s\otimes \partial_{x}^{k-m}W_{r,t}\partial_{x}^{m} X^\epsilon_s+\partial_{x}^{k-n}G_{r,t}\partial_{x}^{n} Z^\epsilon_s\otimes \partial_{x}^{k-m}W_{s,r}\partial_{x}^{m} X^\epsilon_s\\
	 &\quad\quad+\partial_{x}^{k-n}G_{r,t}\partial_{x}^{n} Z^\epsilon_s\otimes \partial_{x}^{k-m}W_{r,t}\partial_{x}^{m} X^\epsilon_s-\partial_{x}^{k-n}G_{r,t}\partial_{x}^{n} Z^\epsilon_r\otimes \partial_{x}^{k-m}W_{r,t}\partial_{x}^{m} X^\epsilon_r ]\\
	 &=:\Gamma^{1,1}_{s,r,t}+\Gamma^{1,2}_{s,r,t}+\Gamma^{1,3}_{s,r,t}+\Gamma^{1,4}_{s,r,t}\,.
	 \end{align*}
	Again by means of \eqref{eq;algebraic_relation}, we rewrite $\delta \Gamma^2_{s,r,t}$ as 
	\begin{align*}
		\delta \Gamma^2_{s,r,t}&=\sqrt{\epsilon}[\partial_{x}^{k-n}W_{s,r}\partial_{x}^{n} X^\epsilon_s\otimes \partial_{x}^{k-m}G_{r,t}\partial_{x}^{m} Z^\epsilon_s+\partial_{x}^{k-n}W_{r,t}\partial_{x}^{n} X^\epsilon_s\otimes \partial_{x}^{k-m}G_{s,r}\partial_{x}^{m} Z^\epsilon_s\\
		&\quad\quad+\partial_{x}^{k-n}W_{r,t}\partial_{x}^{n} X^\epsilon_s\otimes \partial_{x}^{k-m}G_{r,t}\partial_{x}^{m} Z^\epsilon_s-\partial_{x}^{k-n}W_{r,t}\partial_{x}^{n} X^\epsilon_r\otimes \partial_{x}^{k-m}G_{r,t}\partial_{x}^{m} Z^\epsilon_r]\\
		&\quad +\partial_{x}^{k-n}G_{s,r}\partial_{x}^{n} Z^\epsilon_s\otimes \partial_{x}^{k-m}G_{r,t}\partial_{x}^{m} Z^\epsilon_s+\partial_{x}^{k-n}G_{r,t}\partial_{x}^{n} Z^\epsilon_s\otimes \partial_{x}^{k-m}G_{s,r}\partial_{x}^{m} Z^\epsilon_s\\
		&\quad\quad+\partial_{x}^{k-n}G_{r,t}\partial_{x}^{n} Z^\epsilon_s\otimes \partial_{x}^{k-m}G_{r,t}\partial_{x}^{m} Z^\epsilon_s-\partial_{x}^{k-n}G_{r,t}\partial_{x}^{n} Z^\epsilon_r\otimes \partial_{x}^{k-m}G_{r,t}\partial_{x}^{m} Z^\epsilon_r\\
		&=:\Gamma^{2,1}_{s,r,t}+\Gamma^{2,2}_{s,r,t}+\Gamma^{2,3}_{s,r,t}+\Gamma^{2,4}_{s,r,t}\,.
	\end{align*}
	Notice that the following equality holds
	\begin{align}\label{eq:compensazione_a_croce}
	\Gamma^{1,1}_{s,r,t}+\Gamma^{1,3}_{s,r,t}+\Gamma^{2,1}_{s,r,t}+\Gamma^{2,3}_{s,r,t}=\partial_x^k[\sqrt{\epsilon}W_{r,t} X^\epsilon_s+G_{r,t}Z^\epsilon_s]\odot \partial_{x}^k[\sqrt{\epsilon}W_{s,r}X^\epsilon_s+G_{s,r}Z^\epsilon_s]\,.
	\end{align}
	We need to compensate the following terms of at most $p/2$-variation, 
	\begin{align*}
		&\sqrt{\epsilon}\sum_{n=0}^{k} \binom{k}{n}\partial_{x}^{k-n}\delta([GW]+[WG])_{s,r,t}\partial_{x}^{n} X^\epsilon_s\odot\partial_{x}^kZ^\epsilon_s+\sum_{n=0}^{k} \binom{k}{n}\partial_{x}^{k-n}\delta[\mathbb{G}]_{s,r,t}\partial_{x}^{n} Z^\epsilon_s\odot\partial_{x}^kZ^\epsilon_s\\
		&+\sqrt{\epsilon}\partial_{x}^{k}[\delta\mathbb{W}_{s,r,t} u^\epsilon_s]\odot\partial_{x}^kZ^\epsilon_s+\partial_x^k[\sqrt{\epsilon}W_{r,t} X^\epsilon_s+G_{r,t}Z^\epsilon_s]\odot \partial_{x}^k[\sqrt{\epsilon}W_{s,r}X^\epsilon_s+G_{s,r}Z^\epsilon_s]\,.
	\end{align*}
	 We add and subtract from $\delta(\partial_{x}^{n} X^\epsilon\odot \partial_{x}^kZ^\epsilon)_{s,r}$ the elements of at most $p$-variation, which we denote by $T_1$, in such a way that $\delta(\partial_{x}^{n} X^\epsilon\odot \partial_{x}^kZ^\epsilon)_{s,r}-T_1$ is of $p/2$-variation. We explicit $T_1$, 
	\begin{align*}
		T_1&:=\partial_{x}^{n} X^\epsilon_s\odot\partial_{x}^k[\sqrt{\epsilon}W_{s,r}X^\epsilon_s+G_{s,r}Z^\epsilon_s]+\partial_{x}^{n} [W_{s,r}u^\epsilon_s+G_{s,r}X^\epsilon_s]\odot \partial_{x}^kZ^\epsilon_s\,.
	\end{align*}
	Analogously, we subtract from $\delta(\partial_{x}^{n} Z^\epsilon\odot \partial_{x}^kZ^\epsilon)_{s,r}$ the terms of only $p/2$-variation, 
	\begin{align*}
		T_2:=\partial_{x}^nZ^\epsilon\odot \partial_{x}^k[\sqrt{\epsilon}W_{s,r}X^\epsilon_s+G_{s,r}Z^\epsilon_s]+\partial_{x}^n[\sqrt{\epsilon}W_{s,r}X^\epsilon_s+G_{s,r}Z^\epsilon_s]\odot\partial_{x}^kZ^\epsilon\,.
	\end{align*}
	We write the compensation explicitly, 
	\begin{align*}
	\sum_{n=0}^{k} \binom{k}{n} [-\sqrt{\epsilon}\partial_{x}^{k-n}W_{r,t}T_1-\partial_{x}^{k-n}G_{r,t}T_2]&=-\sqrt{\epsilon}\partial_{x}^{k}[W_{r,t} X^\epsilon_s]\odot\partial_{x}^k[\sqrt{\epsilon}W_{s,r}X^\epsilon_s+G_{s,r}Z^\epsilon_s]\\
	&-\sqrt{\epsilon}\sum_{n=0}^{k} \binom{k}{n}\partial_{x}^{k-n}W_{r,t}\partial_{x}^{n} [W_{s,r}u^\epsilon_s+G_{s,r}X^\epsilon_s]\odot \partial_{x}^kZ^\epsilon_s\\
	&-\partial_{x}^{k}[G_{r,t}Z^\epsilon]\odot \partial_{x}^k[\sqrt{\epsilon}W_{s,r}X^\epsilon_s+G_{s,r}Z^\epsilon_s]\\
	&-\sum_{n=0}^{k} \binom{k}{n}\partial_{x}^{k-n}G_{r,t}\partial_{x}^n[\sqrt{\epsilon}W_{s,r}X^\epsilon_s+G_{s,r}Z^\epsilon_s]\odot\partial_{x}^kZ^\epsilon\,.
	\end{align*}
	Notice that the sum of the first and the third sum in the above expression coincides with
	\begin{align*}
		-\partial_x^k[\sqrt{\epsilon}W_{r,t} X^\epsilon_s+G_{r,t}Z^\epsilon]\odot \partial_{x}^k[\sqrt{\epsilon}W_{s,r}X^\epsilon_s+G_{s,r}Z^\epsilon_s]\,,
	\end{align*}
	which erases with \eqref{eq:compensazione_a_croce}. The sum of the second and the fourth sum can be rewritten as
	\begin{align*}
		&-\sqrt{\epsilon}  \partial_{x}^{k}[W_{r,t}W_{s,r}u^\epsilon_s+W_{r,t}G_{s,r}X^\epsilon_s]\odot \partial_{x}^kZ^\epsilon_s-\partial_{x}^{k}[\sqrt{\epsilon}G_{r,t}W_{s,r}X^\epsilon_s+G_{r,t}G_{s,r}Z^\epsilon_s]\odot\partial_{x}^kZ^\epsilon\\
		&=-\sqrt{\epsilon}  \partial_{x}^{k}[\delta \mathbb{W}_{s,r,t}u^\epsilon_s+\delta[GW]_{s,r,t}X^\epsilon_s]\odot \partial_{x}^kZ^\epsilon_s -\partial_{x}^{k}[\sqrt{\epsilon}\delta[WG]_{s,r,t}X^\epsilon_s+\delta \mathbb{G}_{s,r,t}Z^\epsilon_s]\odot\partial_{x}^kZ^\epsilon\,,
	\end{align*}
	which cancels with the terms of at most $p/2$-variation appearing in $\delta \mathbb{I}^1_{s,r,t},\delta \mathbb{I}^2_{s,r,t},\delta \mathbb{I}^3_{s,r,t}$. 
	Therefore we can rewrite $\delta \big(\partial^k_x Z^\epsilon\big)^{\otimes 2,\natural}_{s,r,t}$ as a sum of terms of at least $p/3$-variation. Thus, we need to bound the drifts of  $\delta (Z^\epsilon\otimes Z^\epsilon)$, $\delta( X^\epsilon\otimes X^\epsilon)$, $\delta (Z^\epsilon\otimes X^\epsilon)$. We show the estimate of the drift in $\delta(\partial_{x} Z^\epsilon\otimes \partial_{x} Z^\epsilon)$, the other terms behave similarly. For all $\phi\in W^{1,\infty}$, it holds
	\begin{align*}
	\left|\Biggl\langle \partial_x\left[b'(u_r)X_r-\frac{b(u^\epsilon_r)-b(u_r)}{\sqrt{\epsilon}} \right]\odot \partial_x Z^\epsilon_r,\phi\Biggr\rangle\right|&\lesssim \|\phi\|_{L^\infty}\|\partial^{2}_x Z^\epsilon_r \|_{L^2}^2\\
	&\quad+\|\partial_x\phi\|_{L^\infty} \| Z^\epsilon_r \|_{L^2}\|\partial_x Z^\epsilon_r \|_{L^2}\,.
	\end{align*}
	From the rough standard machinery as in Proposition~\ref{pro:lip_semilinear} (by using the sewing Lemma~\ref{lemma_sewing} and the conditions on the drift), the remainder estimate holds
		\begin{align*}
	\|\big(\partial^k_x Z^\epsilon\big)^{\otimes 2,\natural}_{s,t}\|_{(W^{2,\infty})^*}\lesssim \omega^{1/p}(s,t) \int_{s}^{t}\|\partial^2_x Z^\epsilon_r\|^2_{L^2}\dd r +\omega^{3/p}(s,t)\|Z^\epsilon_r\|^2_{L^\infty([s,t];H^1)}
	+\sqrt{\epsilon}\omega^{3/p}(s,t)\,,
	\end{align*}
	where $\omega$ is  a control depending on $\mathbf{Z},u^\epsilon,u,X^\epsilon,X, Z^\epsilon$. By testing \eqref{eq:X_eps_menus_X_k_square_semilinear} by $\mathbf{1}$ and from the rough Gronwall's Lemma~\ref{lem:gronwall}, we conclude that $(X^\epsilon)_{\epsilon}$ converges strongly to $X$ in $L^\infty(H^1)\cap L^2(H^2)$ with speed of convergence $\sqrt{
	\epsilon}$.
\end{proof}

\subsection{Further examples.}
We discuss how to extend the Wong-Zakai convergence and the CLT to a stochastic reaction-diffusion equation on $\mathbb{T}^d$, for $d=1,2,3$, and how to achieve a CLT for the stochastic Landau-Lifschitz-Gilbert equation on a one dimensional torus. 
\subsubsection{A stochastic reaction-diffusion equation.}\label{section:reaction_diffusion}
We introduce the stochastic reaction-diffusion equation with multiplicative linear noise
\begin{equation}\label{eq:stoch_react_diff}
\delta u_{s,t}=\int_{s}^{t}[\Delta u_r+u_r(1-|u_r|^2)]\dd r+G_{s,t}u_s+\mathbb{G}_{s,t}u_s+u^\natural_{s,t}\,,
\end{equation}
where $u_0=u^0\in H^1(\mathbb{T}^d,\mathbb{R}^n)$, for $d=1,2,3$.  We give the notion of solution to \eqref{eq:stoch_react_diff}.
\begin{definition}\label{def:sol_stoch_reaction_diff}
	Let $\mathbf{G}\equiv (G,\mathbb{G})\in \mathcal{RD}^p(\mathbb{R}^n;W^{1,\infty}) $ and  $p\in [2,3)$. We say that $u\in L^\infty(H^1)\cap L^2(H^2)\cap \mathcal{V}^p(L^2)$ is a solution with initial condition $u^0\in H^1$ if there exists a two index map $u^\natural\in \mathcal{V}^{p/3}_2(L^2)$ defined implicitly by the equality \eqref{eq:stoch_react_diff} in $L^2$, with $u_0=u^0$ as an equality in $L^2$ and the energy inequality holds, for some $\epsilon>0$
	\begin{align*}
	\|u\|^2_{L^\infty(H^1)}+\|u\|^2_{L^2(H^2)}\lesssim_{\mathbf{G},u^0}\|u^0\|^{2}_{L^2}+ \|u^0\|^{\epsilon}_{L^2}\|\nabla u^0\|^2_{L^2}\,.
	\end{align*}
\end{definition}
By denoting $\gamma_2(u):=u(1-|u|^2)$ and by $b_2(u):=\Delta u+ \gamma_2(u)$, we establish existence and uniqueness of strong solutions to the stochastic reaction-diffusion equation \eqref{eq:stoch_react_diff}. We denote by $u=\pi_2(u^0,\mathbf{G})$ the solution to equation \eqref{eq:stoch_react_diff} in the sense of Definition~\ref{def:sol_stoch_reaction_diff} started in $u^0\in H^1$ and driven by a rough driver $\mathbf{G}$. 
\paragraph{A priori estimates and existence.} Let $(G^n,\mathbb{G}^n)$ be a smooth approximation of the geometric rough river $(G,\mathbb{G})$. From a perturbative argument on the classical reaction-diffusion equation, there exists a unique solution $u^n$ to the approximate problem
\begin{align*}
\delta u^n_{s,t}=\int_{s}^{t}[\Delta u^n_r+u^n(1-|u^n|^2)]\dd r+G^n_{s,t}u^n_s+\mathbb{G}^n_{s,t}u^n_s+u^{n,\natural}_{s,t}\,,
\end{align*}
with $u^n_0=u^0$. We show that $(u^n)_n$ is uniformly bounded in $L^\infty(H^1)\cap L^2(H^2)\cap L^4(L^4)\cap \mathcal{V}^p(L^2)$ and the following uniform bound holds
\begin{align}\label{eq:RD_estimate_a_priori}
\sup_{0\leq t\leq T}\|u^n_t\|_{H^1}^2+\int_{0}^{T} \left[ \| u^n_r\|^{2}_{H^2}+ \|u^n\|^4_{L^4}+\int_{\mathbb{T}^d} |u^n_r|^2|\nabla u^n_r|^2\dd x \right]\dd r \leq C(u^0,\mathbf{G},T)\,,
\end{align}
where $C\equiv C(u^0,\mathbf{G})$ is a constant depending polynomially on the $H^1$-norm of the initial condition and on the $p$-variation norm of the noise $\mathbf{G}$. The proof follows from an application of Proposition~\ref{pro:drift_heat_remainder} to the drift of the reaction-diffusion equation: we compute explicitly the norms appearing in the estimate of the remainder. The estimate of the Laplacian coincides with Proposition~\ref{pro:drift_heat_remainder}. We look at the reaction term: for all $\phi\in W^{1,\infty}$ 
\begin{align*}
\int_{0}^{T}|\langle(1-|u_r|^2) u_r \otimes u_r,\phi \rangle |\dd r\lesssim \|\phi\|_{L^\infty} \int_{0}^{T} [\|u_r\|^2_{L^2}+\|u_r\|_{L^4}^4]\dd r\,.
\end{align*}
In conclusion, from the rough Gronwall's Lemma~\ref{lem:gronwall}, Proposition~\ref{pro:remainder_drift_abstract} and the previous considerations on the drift, the following estimate holds
\begin{align}\label{eq:energy_0_level_RD}
\sup_{0\leq t\leq T}\|u^n_t\|_{L^2}^2+\int_{0}^{T} \left[\|\nabla u^n_r\|^{2}_{L^2}+\|u^n\|^4_{L^4}\right]\dd r \lesssim_{\tau}  \exp\left(\frac{\bar{\omega}}{\tau}\right)[|u^{n,0}|^2+\omega_{\mathbf{G}}^{1/p}(0,T)]\,,
\end{align}
We turn now to the reaction term appearing in $\delta (\partial_x u\odot \partial_x u)$. For all $\phi\in W^{1,\infty}$ 
\begin{align*}
\int_{0}^{T}|\langle\partial_x[(1-|u_r|^2) u_r] \otimes\partial_x  u_r,\phi \rangle |\dd r \lesssim \|\phi\|_{L^\infty} \int_{0}^{T}\left [\|\nabla u_r\|^2_{L^2}+\int_{\mathbb{T}^d}|u_r|^2|\nabla u_r|^2\dd x \right]\dd r\,.
\end{align*}
We apply again the rough Gronwall's Lemma~\ref{lem:gronwall}, Proposition~\ref{pro:remainder_drift_abstract}, \eqref{eq:energy_0_level_RD} and the previous considerations on the drift estimates, obtaining
\begin{align*}
\sup_{0\leq t\leq T}\|\nabla u^n_t\|^2_{L^2}+\int_{0}^{T}\left[\|\Delta u^n_r\|^2_{L^2}+\int_{0}^{T}\int_{\mathbb{T}^d}|u^n|^2|\nabla u^n|^2\dd x\right]\dd r\lesssim \exp \left(\omega\right)[\|\nabla u^{n,0}\|^2_{L^2}+\omega^{1/p}_{\mathbf{G}}(0,T)]\,.
\end{align*}
The sequence $(u^n)_n$ is uniformly bounded in $L^\infty(H^1)\cap L^2(H^2)\cap L^4(L^4)\cap \mathcal{V}^p(L^2)$ and \eqref{eq:RD_estimate_a_priori} holds. By a standard procedure, it is possible to identify the limit in the non linearity and show that there exists a limit $u$ which is a solution to \eqref{eq:stoch_react_diff}. Moreover, by the lower semi-continuity of the norm, each solution $u$ satisfies the a priori estimate \eqref{eq:RD_estimate_a_priori}. 
\paragraph{Local Lipschitz continuity of the It\^o-Lyons map and uniqueness.}
By following the computations in Proposition~\ref{pro:lip_semilinear}, the uniqueness of $u$ and the Lipschitz continuity of the It\^o-Lyons map follow. Indeed, we can look only at the drift, since the general structure is the same. Notice again that the behaviour of the Laplacian coincides with the linear equation, thus we concentrate on the reaction term. Given two solutions $u=\pi_2(u^0,\mathbf{G}),v=\pi_2(u^0,\mathbf{W})$ to the reaction-diffusion equation, we can analyse the drift of $\delta (u-v)\odot (u-v)$ tested by $\mathbf{1}$,
\begin{align*}
&\quad-\langle[(1-|u_r|^2) u_r-(1-|v_r|^2) v_r] \odot (u_r-v_r),\mathbf{1}\rangle \\
& = -\langle[(1-|u_r|^2) (u_r-v_r)+[(1-|u_r|^2-1+|v_r|^2) ]v_r] \odot (u_r-v_r),\mathbf{1}\rangle\\
&= -\langle(1-|u_r|^2) (u_r-v_r)\odot (u_r-v_r),\mathbf{1}\rangle-\langle (|v_r|^2-|u_r|^2) v_r\odot (u_r-v_r),\mathbf{1}\rangle\\
&= -\langle(1-|u_r|^2) (u_r-v_r)\odot (u_r-v_r),\mathbf{1}\rangle-\langle (|v_r|^2-|u_r|^2) v_r\odot (u_r-v_r),\mathbf{1}\rangle_{L^2}\\
&= -\langle(1-|u_r|^2) (u_r-v_r), (u_r-v_r) \rangle-\langle (|v_r|^2-|u_r|^2) v_r,(u_r-v_r)\rangle_{L^2}\,.
\end{align*}
In the notations of Proposition~\ref{pro:lip_semilinear}, $z:=u-v$ and $\mathbf{Z}:=\mathbf{G}-\mathbf{W}$. Since $u,v$ are solutions, they fulfil the bound \eqref{eq:RD_estimate_a_priori}: $\|u\|_{L^\infty(H^1)\cap L^2(H^2)}+\|v\|_{L^\infty(H^1)\cap L^2(H^2)}\lesssim C(u^0,\mathbf{G},\mathbf{W})$, 
\begin{align*}
&\quad\int_{0}^{t}\langle[(1-|u_r|^2) u_r-(1-|v_r|^2) v_r] \odot z_r,\mathbf{1}\rangle \dd r \\
&\leq \int_{0}^{t}\left[-\int_{\mathbb{T}^d} |u_r|^2|z_r|^2 \dd x+\|z_r\|^2_{L^2}+\|z_r\|_{L^2}^2[\|v_r\|^2_{L^\infty}+\|u_r\|^2_{L^\infty}]\right] \dd r\,.
\end{align*}
Recall that from Agmon's inequality in two dimensions and three dimensions, it holds that for every $f\in H^2$,
\begin{align}\label{eq:agmon_inequality}
\|f\|_{L^\infty}\lesssim \|f\|_{H^1}^{1/2}\|f\|_{H^2}^{1/2}\,,
\end{align}
We compute the upper bound 
\begin{align*}
&-\int_{0}^{t}\int_{\mathbb{T}^d} |u_r|^2|z_r|^2 \dd x \dd r +\int_{0}^{t}\left[ \|z_r\|^2_{L^2}+\|z_r\|_{L^2}^2[\|v_r\|^2_{L^\infty}+\|u_r\|^2_{L^\infty}]\right] \dd r\\
&\lesssim -\int_{0}^{t}\int_{\mathbb{T}^d} |u_r|^2|z_r|^2 \dd x \dd r+\left[1+\|v_r\|^2_{L^2(L^\infty)}+\|u_r\|^2_{L^2(L^\infty)}\right] \|z\|^2_{L^\infty(L^2)} \,,
\end{align*}
where we use the Agmon's inequality \eqref{eq:agmon_inequality} to obtain that 
\begin{align}\label{eq:stima_fatta _con Agmon}
\|v_r\|^2_{L^2(L^\infty)}\leq \int_{0}^{t}\|v_r\|_{H^1}\|u_r\|_{H^2}\dd r\leq  \int_{0}^{t}[\|v_r\|_{H^1}^2+\|u_r\|_{H^2}^2]\dd r\lesssim C(u^0,\mathbf{G},\mathbf{W})\,.
\end{align}
Hence, the drift of the reaction-diffusion equation has the form
\begin{align*}
\int_{s}^{t} \left[\|\nabla [u_r-v_r]\|^2_{L^2}+\int_{\mathbb{T}^d}|u_r|^2|u_r-v_r|^2\dd x \right]\dd r \leq C(u^0,\mathbf{G},\mathbf{W}) (t-s)\sup_{s\leq r\leq t}\|u_r-v_r\|^2_{L^2}\,.
\end{align*}
By observing \eqref{eq:remainder_WZ}, we see that we need to estimate the drifts coming from $\delta (z\odot z), \delta (z\odot u) $: for all $\phi\in W^{1,\infty}$ and for all $s\leq t\in [0,T]$
\begin{align*}
\int_{s}^{t}|\langle z_r\odot z_r-(u_r|u_r|^2-v_r|v_r|^2)\odot z_r,\phi \rangle |\dd r\lesssim C(u^0,\mathbf{G},\mathbf{W}) \|\phi\|_{L^\infty}(t-s)[\|z\|^2_{L^\infty(L^2)}+\|z\|^2_{L^\infty(L^2)} ]\,.
\end{align*}
The other drift can be estimates similarly: in contrast to the estimate of $\delta (z\odot z) $, the estimate of the drift of $\delta (z\odot u) $ does not need to present a \enquote{$z$} term on the right hand side. Indeed, in the original estimate \eqref{eq:remainder_WZ} the elements of the form  $\delta (z\odot u) $ are already multiplied by $\mathbf{Z}$. If we set $\mathbf{G}=\mathbf{W}$, from the previous considerations and the rough Gronwall's Lemma we can conclude that there exists a unique solution $u$ to \eqref{eq:stoch_react_diff}. By considering different noises instead, we conclude that the It\^o-Lyons map associated to the equation is locally Lipschitz continuous as a map from the space of rough drivers $\mathcal{RD}^p(\mathbb{R}^d;W^{1,\infty})$ to $L^\infty(L^2)\cap L^2(H^1)$. We prove the improved regularity, by looking at the estimate for the remainder of the derivative of the equation. The drifts appearing in the remainder estimate in \eqref{eq:remainder_WZ} come from $\delta (z\odot z), \delta (z\odot u),\delta (\partial_x z\odot z), \delta (z\odot \partial_x u),\delta (\partial_xz\odot u),\delta (\partial_xz\odot \partial_xz) $. We compute explicitly the remainder associated to $\delta (\partial_xz\odot \partial_xz)$ (besides the Laplacian, since the estimate is analogous as in the heat equation): for all $\phi \in W^{1,\infty}$ and for all $s\leq t\in [0,T]$ 
\begin{align*}
	&\quad\int_{s}^{t}|\langle \partial_x z_r\odot \partial_x z_r-\partial_x(u_r|u_r|^2-v_r|v_r|^2)\odot \partial_x z_r,\phi \rangle |\dd r\\
	&\leq \|\phi\|_{L^\infty}(t-s)[\|\nabla z\|^2_{L^\infty(L^2)}+[\|v\|^2_{L^2(L^\infty)}+\|u\|^2_{L^2(L^\infty)}]\|z\|^2_{L^\infty(H^1)}\\
	&\quad +[\|\partial_x u\|^2_{L^\infty(L^2)}+\|\partial_x v\|^2_{L^\infty(L^2)}]\|z_r\|^2_{L^2(H^2)} ]\,,
\end{align*}
where we can complete the estimate with \eqref{eq:stima_fatta _con Agmon}.
By applying the rough Gronwall's Lemma~\ref{lem:gronwall} and the previous considerations, the It\^o-Lyons map $\Phi_2$ associated to the reaction-diffusion equation is locally Lipschitz continuous as a map from the space of rough drivers to $L^\infty(H^1)\cap L^2(H^2)$.

\paragraph{Existence and uniqueness of a solution to the limit equation $X$.}
Denote by $\bar{u}=\Phi_2(\mathbf{0})$ the unique solution to the deterministic stochastic reaction-diffusion equation (the result follows also from the previous existence and uniqueness proof for the stochastic equation with null noise). The Frechét derivative of the drift has the form 
\begin{align*}
b_2'(\bar{u})X=\Delta X+X-X|\bar{u}|^2-2\bar{u}X\cdot \bar{u}\,.
\end{align*}
Analogously to Theorem~\ref{teo:CLT}, we can infer a CLT for the stochastic reaction-diffusion equation \eqref{eq:stoch_react_diff}. We first discuss the existence and uniqueness of the limit equation: we verify Assumption \ref{assumption:drift_linear_b_prime}. It follows that for all $\phi \in W^{1,\infty}$, for all $f,h\in L^\infty(H^1)\cap L^2(H^2)$
\begin{align*}
\int_{s}^{t}|\langle [b_2'(\bar{u})h_r]\otimes  f_r,\phi\rangle|\dd r &\leq \|\nabla h\|_{L^2(L^2)}[\|\nabla f\|_{L^2(L^2)}+\| f\|_{L^\infty(L^2)}]\|\phi\|_{W^{1,\infty}}\\
&\quad +(1+\|\bar{u}\|^2_{L^2(L^\infty)})\|h\|_{L^\infty(L^2)}\|f\|_{L^\infty(L^2)}\|\phi\|_{W^{1,\infty}}\,,
\end{align*}
and analogously that
\begin{align*}
\int_{s}^{t}|\langle\nabla  [b_2'(\bar{u})h_r]\otimes  \nabla f_r,\phi\rangle|\dd r &\leq \|\Delta h_r\|_{L^2(L^2)} \|\nabla f_r\|_{L^2(H^1)} \|\phi\|_{W^{1,\infty}} \\
&\quad + \|\nabla h_r\|_{L^\infty(L^2)} \|\nabla f_r\|_{L^\infty(L^2)} \|\phi\|_{W^{1,\infty}}(t-s)\\
&\quad + \|\bar{u}\|^2_{L^2(L^\infty )}\|\nabla h_r\|_{L^\infty(L^2)} \|\nabla f_r\|_{L^\infty(L^2)} \|\phi\|_{W^{1,\infty}}\\
&\quad + \|\bar{u}\|_{L^\infty(L^2 )}\|\bar{u}\|_{L^2(L^\infty )}\|h\|_{L^\infty(L^2 )}\|f\|_{L^2(L^\infty )}\|\phi\|_{W^{1,\infty}} \,.
\end{align*}
It also holds that
\begin{align*}
	\int_{0}^{t}\langle b_2'(\bar{u}_r) f_r,f_r \rangle\dd r =-\int_{0}^{t}\|\nabla f_r\|^2_{L^2}\dd r+\int_{0}^{t}\| f_r\|^2_{L^2}\dd r-\int_{0}^{t}\| \bar{u}\cdot f_r\|^2_{L^2}\dd r-\int_{0}^{t}\||\bar{u}_r|| f_r|\|^2_{L^2}\dd r\,,
\end{align*}
as well as
\begin{align*}
\int_{0}^{t}\langle \nabla [b_2'(\bar{u}_r) f_r],\nabla f_r \rangle\dd r &\leq -\int_{0}^{t}\|\Delta  f_r\|^2_{L^2}\dd r+[1+\|\bar{u}\|_{L^\infty(H^1)}]\int_{0}^{t}\|  f_r\|^2_{H^1}\dd r- \int_{0}^{t}\| \bar{u}_r\cdot \nabla f_r\|^2_{L^2}\dd r\,.
\end{align*}
In conclusion, Assumption \ref{assumption:drift_linear_b_prime} holds and existence, uniqueness, pathwise Wong-Zakai, a large deviation principle and a support theorem hold for the solution $X$ to the limit equation.

\paragraph{CLT for the reaction-diffusion equation and $X=D\Phi_2[\mathbf{0}](\mathbf{W})$.}
We prove the CLT for the reaction-diffusion equation. Denote by $u^\epsilon=\Phi_2(\tau_{\sqrt{\epsilon}}\mathbf{W} )$ and by $X^\epsilon:=(u^\epsilon-\bar{u})/\sqrt{\epsilon}$.
We follow Theorem~\ref{teo:CLT} and, since there are changes only in the non linear drift terms, we write only that those elements. From Agmon's inequality \eqref{eq:agmon_inequality}, we obtain that the following elements are uniformly bounded by a constant $C>0$ depending polynomially on the initial condition $u^0$, on the noises and exponentially on time
\begin{align*}
\|\bar{u}\|^2_{L^2(L^\infty)}, \|u^\epsilon\|^2_{L^2(L^\infty)}, \| X\|_{L^2(L^\infty)}, \| X^\epsilon \|_{L^2(L^\infty)}\leq C\,.
\end{align*}
The above estimate would be finer, but it still holds in one dimension (use Agmon's inequality \eqref{eq:agmon_inequality}).  In the main body of the proof, we need to estimate the drift of $\delta (X^\epsilon- X)^{\otimes 2}$
\begin{align*}
&\quad\int_{s}^{t}\Biggl\langle \left[ \frac{ u^\epsilon|u^\epsilon|-\bar{u}|\bar{u}|^2}{\sqrt{\epsilon}} -X|\bar{u}|^2-2\bar{u}X\cdot \bar{u}\right]\otimes (X^\epsilon -X), \phi\Biggr\rangle\dd r \\
&=\int_{s}^{t}\Biggl\langle \left[ \frac{ (u^\epsilon-\bar{u})|u^\epsilon|^2-\bar{u}[|u^\epsilon|^2-|\bar{u}|^2]}{\sqrt{\epsilon}} -X|\bar{u}|^2-2\bar{u}X\cdot \bar{u}\right]\otimes (X^\epsilon -X), \phi\Biggr\rangle\dd r \\
&=\int_{s}^{t}\Biggl\langle \left[  (X^\epsilon-X)|u^\epsilon|^2+X[|u^\epsilon|^2-|\bar{u}|^2]-\frac{\bar{u}[|u^\epsilon|^2-|\bar{u}|^2]}{\sqrt{\epsilon}} -2\bar{u}X\cdot \bar{u}\right]\otimes (X^\epsilon -X), \phi\Biggr\rangle\dd r \,.
\end{align*}
We observe that the following equality holds 
\begin{align*}
&\quad-\frac{\bar{u}[|u^\epsilon|^2-|\bar{u}|^2]}{\sqrt{\epsilon}} -2\bar{u}X\cdot \bar{u}\\
&=  -\frac{\bar{u}[u^\epsilon-\bar{u}]\cdot[u^\epsilon+\bar{u}]}{\sqrt{\epsilon}} -2\bar{u}X\cdot \bar{u}\\
&=-\frac{\bar{u}[|u^\epsilon|-|\bar{u}|]|u^\epsilon|}{\sqrt{\epsilon}}-\frac{\bar{u}[|u^\epsilon|-|\bar{u}|]|\bar{u}|}{\sqrt{\epsilon}} -2\bar{u}X\cdot \bar{u}\\
&=-\frac{\bar{u}[|u^\epsilon|-|\bar{u}|][|u^\epsilon|-|\bar{u}|]}{\sqrt{\epsilon}}-\frac{\bar{u}[|u^\epsilon|-|\bar{u}|]|\bar{u}|}{\sqrt{\epsilon}}-\frac{\bar{u}[|u^\epsilon|-|\bar{u}|]|\bar{u}|}{\sqrt{\epsilon}} -2\bar{u}X\cdot \bar{u}\\
&=-\frac{\bar{u}[|u^\epsilon|-|\bar{u}|][|u^\epsilon|-|\bar{u}|]}{\sqrt{\epsilon}}-
2\bar{u}\left[\frac{[|u^\epsilon|-|\bar{u}|]} {\sqrt{\epsilon}}-X\right]\cdot \bar{u}\,.
\end{align*}
Hence, we can estimate  
\begin{align*}
&\quad \int_{s}^{t}\Biggl\langle \left[ \frac{ u^\epsilon|u^\epsilon|-\bar{u}|\bar{u}|^2}{\sqrt{\epsilon}} -X|\bar{u}|^2-2\bar{u}X\cdot \bar{u}\right]\otimes (X^\epsilon -X), \phi\Biggr\rangle\dd r \\
&\lesssim \|u^\epsilon\|^2_{L^2(L^\infty)}\|X^\epsilon-X\|_{L^\infty(L^2)}^2\|\phi\|_{L^\infty} \\
&\quad +\sqrt{\epsilon}\|X\|_{L^2(L^2)}\|X^\epsilon\|_{L^2(L^\infty)} \|X^\epsilon-X\|_{L^\infty(L^2)}\|\phi\|_{L^\infty}\\
&\quad + \sqrt{\epsilon}\|X^\epsilon\|_{L^2(L^\infty)}\|X^\epsilon\|_{L^\infty(L^2)}\|\bar{u}\|_{L^2(L^\infty)}\|X^\epsilon-X\|_{L^\infty(L^2)}\|\phi\|_{L^\infty}\\
&\quad + \|\bar{u}\|_{L^2(L^\infty)}^2\|X^\epsilon-X\|_{L^2(L^2)}\|\phi\|_{L^\infty}\\
&\lesssim \|\phi\|_{L^\infty} [\epsilon+\|X^\epsilon-X\|_{L^\infty(L^2)}^2] \,.
\end{align*}
Set $Z^\epsilon:=X^\epsilon-X$. We compute the drift for the equation for the derivative:
\begin{align*}
&\quad \int_{s}^{t}\Biggl\langle \nabla \left[ \frac{ u^\epsilon|u^\epsilon|-\bar{u}|\bar{u}|^2}{\sqrt{\epsilon}} -X|\bar{u}|^2-2\bar{u}X\cdot \bar{u}\right]\otimes \nabla (X^\epsilon -X), \phi\Biggr\rangle\dd r \\
&\lesssim \|u^\epsilon\|^2_{L^2(L^\infty)}\|\nabla Z^\epsilon\|_{L^\infty(L^2)}^2\|\phi\|_{L^\infty} +\|\nabla Z^\epsilon\|_{L^\infty(L^2)}\|Z^\epsilon\|_{L^\infty(L^2)}\|u^\epsilon\cdot \nabla u^\epsilon\|_{L^1(L^1)} \|\phi\|_{L^\infty}\\
&\quad+\sqrt{\epsilon} \|\nabla X\|_{L^\infty(L^2)}\|X^\epsilon\|_{L^2(L^\infty)}\|u^\epsilon+\bar{u}\|_{L^2(L^\infty)}\|\nabla Z^\epsilon\|_{L^\infty(L^2)}\|\phi\|_{L^\infty}\\ &\quad+\sqrt{\epsilon} \| X\|_{L^2(L^\infty)}\|\nabla X^\epsilon\|_{L^2(L^2)}\|u^\epsilon+\bar{u}\|_{L^2(L^\infty)}\|\nabla Z^\epsilon\|_{L^\infty(L^2)}\|\phi\|_{L^\infty} \\
&\quad+\sqrt{\epsilon} \| X\|_{L^2(L^\infty)}\| X^\epsilon\|_{L^2(L^\infty)}\|\nabla u^\epsilon+\nabla \bar{u}\|_{L^\infty(L^2)}\|\nabla Z^\epsilon\|_{L^\infty(L^2)}\|\phi\|_{L^\infty} \\
&\quad + \sqrt{\epsilon} \| \nabla \bar{u}\|_{L^\infty(L^2)}\|X^\epsilon\|_{L^2(L^\infty)}\|X^\epsilon\|_{L^2(L^\infty)}\|\nabla Z^\epsilon\|_{L^\infty(L^2)}\|\phi\|_{L^\infty} \\
&\quad + 2 \sqrt{\epsilon} \|  \bar{u}\|_{L^2(L^\infty)}\|\nabla X^\epsilon\|_{L^\infty(L^2)}\|X^\epsilon\|_{L^2(L^\infty)}\|\nabla Z^\epsilon\|_{L^\infty(L^2)}\|\phi\|_{L^\infty} \\
&\quad +2\|\nabla \bar{u}\|_{L^\infty(L^2)}\|Z^\epsilon\|_{L^2(L^\infty)}\|\nabla Z^\epsilon\|_{L^\infty(L^2)} \| \bar{u}\|_{L^\infty(L^\infty)} \|\phi\|_{L^\infty} \\
&\quad +\|\bar{u}\|^2_{L^2(L^\infty)}\|\nabla Z^\epsilon\|^2_{L^\infty(L^2)} \|\phi\|_{L^\infty} \\
&\lesssim \|\phi\|_{L^\infty}[\epsilon+ \|Z^\epsilon\|^2_{L^\infty(H^1)}]\,.
\end{align*}
The above estimates of the drift allow to close the estimate of the remainder, via the rough standard machinery. We test the equation by $\mathbf{1}$ and from Gronwall's Lemma~\ref{lem:gronwall},
\begin{align*}
\sup_{0\leq t\leq T}\|Z^\epsilon_t\|^2_{H^1}+\int_{0}^{T}\|\nabla Z^\epsilon_r\|^2_{H^1}\dd r\lesssim_{\mathbf{W}} \epsilon +\omega^{1/p}\sqrt{\epsilon}\,.
\end{align*}
In conclusion, the convergence of the sequence $(X^\epsilon)_{\epsilon}$ to the limit $X$ occurs in $L^\infty(H^1)\cap L^2(H^2)\cap \mathcal{V}^p(L^2)$ $\mathbb{P}$-a.s. with speed $\sqrt{\epsilon}$. As already observed in Remark~\ref{remark:ito_CLT}, \textit{also the solution to the reaction-diffusion equation driven a multiplicative liner noise interpreted in the It\^o's sense converges pathwise to the same limit equation.}

\subsection{On the CLT in the optimal space for the stochastic Landau-Lifschitz-Gilbert equation on $\mathbb{T}$.} \label{sec:LLG}
We study the CLT for the stochastic Landau-Lifschitz-Gilbert equation on the one dimensional torus, which we recall
\begin{align}\label{eq:rLLG}
\delta u_{s,t}=\int_{s}^{t}b_3(u_r)\dd r+W_{s,t}u_s+\mathbb{W}_{s,t}u_s+u^\natural_{s,t}\,,
\end{align}
where $b_3(u):=\partial_x^2 u+\gamma_3(u)$ and $\gamma_3(u):=u\times\partial^2_x u+u|\partial_x u|^2$ (the equation is quasi-linear). Note that the equation is often studied in its equivalent form, where the drift $b_3(u)=u\times\partial^2_x u-u\times(u\times\partial^2_x u)$.
The rough path $\mathbf{W}\equiv(W,\mathbb{W})$ is constructed in such a way that the solution does not escape the sphere $\mathbb{S}^2$ of $\mathbb{R}^3$, namely for a given $\mathbf{H}\equiv(H,\mathbb{H})\in \mathcal{RD}^p(\mathbb{R}^3;H^{k+1})$, with $k\in \mathbb{N}$, we define 
\begin{equation*}
W:= \begin{pmatrix}
0 & H^3 & -H^2\\
-H^3 & 0 & H^1\\
H^2 & -H^1 & 0
\end{pmatrix}
\in \V^{p}(0,T;H^{k+1}(\T ;\mathcal L_a(\R^3)))\,,
\end{equation*}
\begin{equation*}
\label{murky_formula}
\mathbb{W}:=
\begin{pmatrix}
-\mathbb{H}^{3,3} -\mathbb{H}^{2,2}&
\mathbb{H}^{1,2} &
\mathbb{H}^{1,3}\\
\mathbb{H}^{2,1} &
-\mathbb{H}^{3,3} -\mathbb{H}^{1,1}&
\mathbb{H}^{2,3}\\
\mathbb{H}^{3,1} &
\mathbb{H}^{3,2} &
-\mathbb{H}^{2,2} -\mathbb{H}^{1,1}\\
\end{pmatrix}
\in \V_2^{p/2}(0,T;H^{k+1}(\mathbb{T} ;\mathcal L(\R^3)))\,.
\end{equation*}
We denote the space of such \enquote{spherical} rough drivers by $\mathcal{RD}^p_S(\mathbb{R}^3;H^{k+1})$.
We refer the reader to the construction in  \cite{LLG1D} for more details (see also \cite{LLG_inv_measure}).
Let $u^0\in H^k(\mathbb{T},\mathbb{S}^2)$, then as proved in \cite{LLG1D} there exists a unique solution $u=\pi_3(u^0,\mathbf{W})\in L^\infty(H^{k})\cap L^2(H^{k+1})\cap\mathcal{V}^p(H^{k-1})$ to \eqref{eq:rLLG} such that the It\^o-Lyons map $\Phi_3:\mathcal{RD}_S^p(\mathbb{R}^3;H^{k+1})\rightarrow L^\infty(H^{k})\cap L^2(H^{k+1})$ is locally Lipschitz continuous. If the rough driver is identically $0$, then we recover existence and uniqueness of $\bar{u}=\pi_3(u^0,\mathbf{0})$.
From the same theorem, for every $\epsilon>0$ there exists a unique solution $u^\epsilon=\pi_3(u^0;\tau_{\sqrt{\epsilon}}\mathbf{W})$. We aim to prove a central limit theorem for the equation \eqref{eq:rLLG}, namely that the sequence $(X^\epsilon)_\epsilon$, defined by $X^\epsilon:=(u^\epsilon-\bar{u})/\sqrt{\epsilon}$,
converges to a limit $X$ in some space, where $X$ is the unique pathwise solution to
\begin{align}\label{eq:LLG_X_CLT}
\delta X_{s,t}=\int_{s}^{t}b_3'(\bar{u}_r)X_r \dd r+W_{s,t}\bar{u}_s+X^\natural_{s,t}\,,
\end{align}
where $b_3'(\bar{u}_r)X_r$ is given by 
\begin{align*}
\partial_{x}^{2} X_r +\partial_x (\bar{u}_r \times \partial _{x}X_r)+
\partial_x (X_r\times \partial_{x}\bar{u}_r)
+2\bar{u}_r \partial_x \bar{u}_r\cdot \partial_x X_r+ |\partial_x \bar{u}_r|^2 X_r\, 
\end{align*}
and $X_0=0$. Observe that from the spherical constrain $|\bar{u}_t|=1$ for a.e. $t\in [0,T]$, $\mathbb{P}$-a.s. Moreover, $\|\partial_x \bar{u}\|_{L^\infty}\lesssim \|\bar{u}\|_{H^2}$ from the one-dimensional embedding, hence $b_3'(\bar{u}_r)\cdot$ is a linear operator with coefficients in $L^\infty(L^\infty)$. The integral appearing in \eqref{eq:LLG_X_CLT} is of one-variation: from a perturbation argument of the classical theory, there exists a pathwise unique solution to \eqref{eq:LLG_X_CLT} with initial condition $X_0=0$. One can also follow Appendix~\ref{Appendix_B} for a complete proof of well posedness of $Y$ and continuity with respect to $\mathbf{G}$ (note that the drift $b_3'(\bar{u})$ fulfils the Assumption \ref{assumption:drift_linear_b_prime} in Appendix~\ref{Appendix_B}).

By employing the methodology already introduced in \cite{LLG1D} for the Wong-Zakai convergence, we can establish a pathwise CLT in $L^\infty(H^k)\cap L^2(H^{k+1})\cap \mathcal{V}^p(H^{k-1})$ with speed of convergence $\sqrt{\epsilon}$. First we introduce some notations.
Consider the equation for the difference $X^\epsilon-X$, given by
\begin{align}\label{eq:Xeps-X}
\delta (X^\epsilon-X)_{s,t}= \int_{s}^{t} \bigg(\frac{b_3(u^\epsilon_r)-b_3(\bar{u}_r)}{\sqrt{\epsilon}}- b_3'(\bar{u}_r)X_r \bigg)\dd r+W_{s,t}(u^\epsilon_s-\bar{u}_s)+\sqrt{\epsilon}\mathbb{W}_{s,t}u^\epsilon_s+(X^\epsilon-X)^\natural_{s,t}\,.
\end{align}
Since we want to show the convergence in $L^\infty(H^k)\cap L^2(H^{k+1})$, we look at the equation for the $k$-th derivative
\begin{align*}
\delta\partial^k_x (X^\epsilon-X)_{s,t}&= \int_{s}^{t} \partial^k_x\bigg(\frac{b_3(u^\epsilon_r)-b_3(\bar{u}_r)}{\sqrt{\epsilon}}- b_3'(\bar{u}_r)X_r \bigg)\dd r\\
&\quad+\partial^k_x\big(W_{s,t}(u^\epsilon_s-\bar{u}_s)+\sqrt{\epsilon}\mathbb{W}_{s,t}u^\epsilon_s\big)+\partial_x^k(X^\epsilon-X)^\natural_{s,t}\,.
\end{align*}
 We consider then the equation for $\partial^k_x (X^\epsilon-X)\otimes \partial^k_x (X^\epsilon-X) $, by means of the product formula in Proposition~\ref{pro:product},
\begin{align}\label{eq:X_eps_menus_X_k_square}
\delta\big(\partial^k_x (X^\epsilon-X)\otimes \partial^k_x (X^\epsilon-X)\big)_{s,t}=\int_{s}^{t}\mathcal{D}(k,k)\dd r+I_{s,t}+\mathbb{I}_{s,t}+\tilde{\mathbb{I}}_{s,t}+\big(\partial^k_x (X^\epsilon-X)\big)^{\otimes 2,\natural}_{s,t}\,,
\end{align}
where we define
\begin{align*}
\int_{s}^{t}\mathcal{D}(k,k)\dd r&:=\int_s^t \left(\partial_{x}^k\bigg(\frac{b_3(u^\epsilon_r)-b_3(\bar{u}_r)}{\sqrt{\epsilon}}- b_3'(\bar{u}_r)X_r \bigg)\odot\partial_{x}^k(X^\epsilon_r-X_r) \right)\dd r\,,
\end{align*}
\begin{equation*}
\begin{aligned}
\textcolor{black}{I}_{s,t}&:=\sum_{n=0}^{k} \binom{k}{n} (\partial_{x}^{k-n}W_{s,t}\partial_{x}^{n} u^\epsilon_s-\partial_{x}^{k-n}W_{s,t}\partial_{x}^{n} \bar{u}_s)\odot \partial_{x}^k(X^\epsilon_s-X_s)\,,
\end{aligned}
\end{equation*}
\begin{equation*}
\begin{aligned}
\mathbb{I}_{s,t}&:=\sum_{n=0}^{k} \binom{k}{n}(\partial_{x}^{k-n}\sqrt{\epsilon}\mathbb{W}_{s,t}\partial_{x}^{n} u^\epsilon_s)\odot\partial_{x}^k(X^\epsilon_s-X_s)\,,
\end{aligned}
\end{equation*}
\begin{equation*}
\begin{aligned}
\tilde{\mathbb{I}}_{s,t}&:=\textcolor{black}{\sum_{n=0}^{k} \binom{k}{n} (\partial_{x}^{k-n}W_{s,t}\partial_{x}^{n} u^\epsilon_s-\partial_{x}^{k-n}W_{s,t}\partial_{x}^{n} \bar{u}_s)\otimes\sum_{m=0}^{k} \binom{k}{m}(\partial_{x}^{k-m}W_{s,t}\partial_{x}^{m} u^\epsilon_s-\partial_{x}^{k-m}W_{s,t}\partial_{x}^{m} \bar{u}_s)}\,.
\end{aligned}
\end{equation*}
Recall that $u^\epsilon-\bar{u}=\sqrt{\epsilon}X^\epsilon$, hence the terms $I,\mathbb{I},\tilde{\mathbb{I}}$ are multiplied by $\sqrt{\epsilon}$. Moreover, from the local Lipschitz continuity of the It\^o-Lyons map, $X^\epsilon$ is uniformly bounded in $\epsilon$: besides the drift and the remainder, each term has already speed $\sqrt{\epsilon}$. Before the proof of the CLT for the LLG, we estimate the drift term $\mathcal{D}(k,k)$ in Proposition~\ref{pro:drift_LLG}.
\begin{proposition}\label{pro:drift_LLG}
	The drift in the main $L^2$-estimate can be estimated as follows 
	\begin{align*}
	\int_{s}^{t}\int_{\mathbb{T}}\langle \mathcal{D}(k,k), \mathbf{1}\rangle\dd x\dd r\lesssim_{\mathbf{G}} \epsilon g\left(\|X^\epsilon\|_{L^\infty(H^k)\cap L^2(H^{k+1})}\right)+\|X^\epsilon-X\|^2_{L^\infty(H^k)\cap L^2(H^{k+1})}\,,
	\end{align*}
	where $g$ is a polynomial function.
\end{proposition}
\begin{proof}
	We divide the drift into four parts
	\begin{align*}
	\int_{s}^{t}\int_{\mathbb{T}}\mathcal{D}(k,k) \dd x \dd r=\int_{s}^{t}\langle\partial_x^k( D_{1,r}+ D_{2,r}+D_{3,r}), \partial_x^k(X^\epsilon_r-X_r)\rangle \dd r\,.
	\end{align*}
	where
	\begin{align*}
		&D_{1,r}:= \partial_x^2 (X^\epsilon_r-X_r)\,, \; D_{2,r}:=\frac{u^\epsilon_r \times \partial_x^2 u^\epsilon_r -\bar{u}_r\times \partial_x^2\bar{u}_r}{\sqrt{\epsilon}}-\partial_x (\bar{u}_r \times \partial _{x}X_r)-
		\partial_x (X_r\times \partial_{x}\bar{u}_r)\,,\\
		&D_{3,r}:= \frac{u^\epsilon_r|\partial_x u^\epsilon_r|^2-\bar{u}_r|\partial_x \bar{u}_r|^2}{\sqrt{\epsilon}}-2\bar{u}_r \partial_x \bar{u}_r\cdot \partial_x X_r- |\partial_x \bar{u}_r|^2 X_r\,.
	\end{align*}
	We analyse each element of the drift separately. From $\partial_x^k D_1$, we obtain a term with negative sign
	\begin{align*}
	\int_{s}^{t}\langle \partial_x^k D_{1,r}, \partial_x^k(X^\epsilon_r-X_r)\rangle \dd r=-\int_{s}^{t}\|\partial_x^{k+1}(X^\epsilon_r-X_r)\|_{L^2}^2 \dd r\,,
	\end{align*}
	which is the regularising term of the equation. From algebraic manipulations, 
	\begin{align*}
	\int_{s}^{t}\langle \partial_x^k D_{2,r}, \partial_x^k(X^\epsilon_r-X_r)\rangle \dd r&=-\int_{s}^{t}\langle \sqrt{\epsilon}\partial^{k}_x(X^\epsilon\times\partial_x  X^\epsilon)+\partial^{k}_x((X^\epsilon-X)\times\partial_x  \bar{u}),\partial_x^{k+1}(X^\epsilon_r-X_r)\rangle \dd r\\
	&\quad-\int_{s}^{t}\partial^{k}_x(\bar{u}\times\partial_x (X^\epsilon-X)), \partial_x^{k+1}(X^\epsilon_r-X_r)\rangle \dd r\,.
	\end{align*}
	We can then estimate, with computations analogous to the proof of Proposition 4.3 in \cite{LLG1D}, $\partial_x^k D_2$ by
	\begin{align*}
	\langle \partial_x^k D_2, \partial_x^k(X^\epsilon-X)\rangle&\leq \sqrt{\epsilon}[\|X^\epsilon\|_{H^k}\| X^\epsilon\|_{H^2}+\|   X^\epsilon\|_{H^{k+1}}\|X^\epsilon\|_{H^1}]\|X^\epsilon-X\|_{H^{k+1}}\\
	&\quad+[\|X^\epsilon-X\|_{H^k}\| \bar{u}\|_{H^2}+\| \bar{u}\|_{H^{k+1}}\|X^\epsilon-X\|_{H^1}]\|X^\epsilon-X\|_{H^{k+1}}\\
	&\quad+[\|  \bar{u}\|_{H^{k}}\|X^\epsilon-X\|_{H^2}+\|X^\epsilon-X\|_{H^{k+1}}\| \bar{u}\|_{H^1}]\|X^\epsilon-X\|_{H^{k+1}}\,.
	\end{align*}
	In conclusion, we estimate the term $\partial_x^k D_2$ by
	\begin{align*}
	\int_{s}^{t}\langle \partial_x^k D_{2,r}, \partial_x^k(X^\epsilon_r-X_r)\rangle \dd r&\lesssim\epsilon \|X^\epsilon\|^2_{L^\infty(H^k)}\|X^\epsilon\|^2_{L^2(H^{k+1})}+\|X^\epsilon-X\|^2_{L^2(H^{k+1})}\\
	&\, +\|X^\epsilon-X\|^2_{L^\infty(H^k)}\|\bar{u}\|_{L^2(H^{k+1})}^2+\|\bar{u}\|^2_{L^\infty(H^k)}\|X^\epsilon-X\|^2_{L^2(H^{k+1})}\,.
	\end{align*}
	We address now $\partial_x^k D_3$ :
	\begin{align*}
	\int_{s}^{t}\langle \partial_x^k D_{3,r},& \partial_x^k(X^\epsilon-X)\rangle \dd r\\
	&=-\int_{s}^{t}\langle \partial_x^{k-1}[(X^\epsilon-X)|\partial_x \bar{u}|^2+\bar{u}\partial_x (X^\epsilon-X)\cdot\partial_x \bar{u}], \partial_x^{k+1}(X^\epsilon-X)\rangle \dd r\\
	&\quad -\int_{s}^{t}\langle \partial_x^{k-1}[\sqrt{\epsilon}X^\epsilon\partial_x u^\epsilon\cdot\partial_x X^\epsilon+\sqrt{\epsilon}\bar{u}|\partial_x X^\epsilon|^2], \partial_x^{k+1}(X^\epsilon-X)\rangle \dd r\\
	&\quad -\int_{s}^{t}\langle \partial_x^{k-1}[\bar{u}\partial_x \bar{u}\cdot\partial_x (X^\epsilon-X)], \partial_x^{k+1}(X^\epsilon-X)\rangle \dd r\,.
	\end{align*}
	For $f,h,\ell\in H^{k+1}$, we observe that
	\begin{align*}
	\|\partial_x^{k-1}[f(\partial_x \ell,\partial_x h)]\|_{L^2}&\lesssim \|f\partial_x \ell\cdot\|_{L^\infty}\|\partial_x h\|_{H^{k-1}}+\|f\partial_x \ell\cdot\|_{H^{k-1}}\|\partial_x h\|_{L^\infty}\\
	&\lesssim \|f\|_{H^1}\|\ell\|_{H^2}\|h\|_{H^{k}}+[\|f\|_{H^{k-1}}\|\ell\|_{H^2}+\|f\|_{H^1}\|\ell\|_{H^k}]\|h\|_{H^2}\,.
	\end{align*}
	Thus, by applying it to each element of $\partial_x^k D_3$, it follows that for a.e. $r\in [0,T]$
	\begin{align*}
	\langle \partial_x^k D_{3,r}, \partial_x^k(X^\epsilon-X)\rangle
	&\lesssim \|X^\epsilon-X\|^2_{H^{k+1}}+\|X^\epsilon-X\|_{H^1}^2\| \bar{u}\|_{H^2}^2\|\bar{u}\|_{H^k}^2\\
	&\quad+\|X^\epsilon-X\|_{H^{k-1}}^2\| \bar{u}\|_{H^2}^2\|\bar{u}\|_{H^2}^2+\|X^\epsilon-X\|_{H^{1}}^2\|
	\bar{u}\|_{H^k}^2\|\bar{u}\|_{H^2}^2\\
	&\quad+\|\bar{u}\|_{H^1}^2\|X^\epsilon-X\|_{H^2}^2\|\bar{u}\|_{H^k}^2\\
	&\quad+\|\bar{u}\|_{H^{k-1}}^2\|X^\epsilon-X\|_{H^2}^2\|
	\bar{u}\|_{H^2}^2+\|\bar{u}\|_{H^{1}}^2\|X^\epsilon-X\|_{H^k}^2\|\bar{u}\|_{H^2}^2\\
	&\quad+\epsilon\|X^\epsilon\|_{H^1}^2\|u^\epsilon\|_{H^2}^2\|X^\epsilon\|_{H^k}^2\\
	&\quad+\epsilon\|X^\epsilon\|_{H^{k-1}}^2\|u^\epsilon\|_{H^2}^2\|
	X^\epsilon\|_{H^2}^2+\epsilon\|X^\epsilon\|_{H^{1}}^2\|u^\epsilon\|_{H^k}^2\|X^\epsilon\|_{H^2}^2\\
	&\quad+\epsilon\|\bar{u}\|_{H^1}^2\|X^\epsilon\|_{H^2}^2\|X^\epsilon\|_{H^k}^2\\
	&\quad+\epsilon\|\bar{u}\|_{H^{k-1}}^2\|X^\epsilon\|_{H^2}^2\|
	X^\epsilon\|_{H^2}^2+\epsilon\|\bar{u}\|_{H^{1}}^2\|X^\epsilon\|_{H^k}^2\|X^\epsilon\|_{H^2}^2\\
	&\quad+\|\bar{u}\|_{H^1}^2\|\bar{u}\|_{H^2}^2\|X^\epsilon-X\|_{H^k}^2\\
	&\quad+\|\bar{u}\|_{H^{k-1}}^2\|\bar{u}\|_{H^2}^2\|
	X^\epsilon-X\|_{H^2}^2+\|\bar{u}\|_{H^{1}}^2\|\bar{u}\|_{H^k}^2\|X^\epsilon-X\|_{H^2}^2.
	\end{align*}
	Thus, by first applying the above inequality and then from H\"older's inequality in time, we conclude that
	\begin{align*}
	\int_{s}^{t} \langle \partial_x^k D_{3,r}, \partial_x^k(X^\epsilon_r-X_r)\rangle \dd r
	&\lesssim_{u^0} \|X^\epsilon-X\|^2_{L^2(s,t;H^{k+1})}+\|X^\epsilon-X\|^2_{L^\infty(s,t;H^k)}\\
	&\quad+\epsilon g(X^\epsilon, X,u^\epsilon, \bar{u})\,,
	\end{align*}
	where $g(X^\epsilon, X,u^\epsilon, \bar{u})$ is a polynomial function uniformly  bounded in $\epsilon$. The constant depends on some positive powers of the initial condition, since $\|\bar{u}\|_{L^\infty(H^k)\cap L^2(H^{k+1})}\leq \|u^0\|_{H^k}$.
\end{proof}

We state a pathwise CLT for the stochastic Landau-Lifschitz-Gilbert equation.
\begin{theorem}\label{th:local_lip_cont_der_LLG}
	Let $\mathbf{W}\in \mathcal{RD}^p_S(\mathbb{R}^3;H^{k+1})$ be a random $p$-geometric rough driver. The sequence $(X^\epsilon)_\epsilon$ converges $\mathbb{P}$-a.s. in $L^\infty(H^k)\cap L^2(H^{k+1})\cap \mathcal{V}^p(H^{k-1})$ to a process $X$ with speed of convergence $\sqrt{\epsilon}$. The limit process $X$ is the pathwise directional derivative of $\Phi_3$ in $\mathbf{0}$ in the direction $\mathbf{W}$, i.e. $X=D\Phi_3[\mathbf{0}](\mathbf{W})$ and it solves the equation \eqref{eq:LLG_X_CLT}. 
\end{theorem}

\begin{proof} We start by looking at the solution to equation \eqref{eq:Xeps-X}: we show pathwise convergence of $X^\epsilon$ to $X$ in $\mathcal{Y}:=L^\infty(H^k)\cap L^2(H^{k+1})$.
	We test equation \eqref{eq:X_eps_menus_X_k_square} by $\mathbf{1}$ and we study separately the elements appearing in
	\begin{align*}
	\langle\delta(\partial^k_x (X^\epsilon-X)\otimes \partial^k_x (X^\epsilon-X))_{s,t},\mathbf{1}\rangle =\langle\mathcal{D}(k,k)_{s,t}+I_{s,t}+\mathbb{I}_{s,t}+\tilde{\mathbb{I}}_{s,t},\mathbf{1}\rangle +\langle\big(\partial^k_x (X^\epsilon-X)\big)^{\otimes 2,\natural}_{s,t},\mathbf{1}\rangle\,.
	\end{align*}
	Since $\mathbf{W}\in \mathcal{RD}_S^p(H^{k+1})$, $X^\epsilon, u^\epsilon\in L^\infty(H^k)$, the noise terms can be estimated by
	\begin{equation}\label{eq:noise_terms}
	\begin{aligned}
	\int_{\mathbb{T}}\langle I_{s,t}+\mathbb{I}_{s,t}+\tilde{\mathbb{I}}_{s,t}, \mathbf{1}\rangle\dd x\lesssim_k &\sqrt{\epsilon} \omega^{1/p}_{ \mathbf{W}; H^{k+1}}(s,t)[\|X^\epsilon\|^2_{L^\infty(H^k)}+\|u^\epsilon\|^2_{L^\infty(H^k)}]\\
	&+\omega^{1/p}_{ \mathbf{W}; H^{k+1}}(s,t)\|X^\epsilon-X\|^2_{L^\infty([s,t];H^k)}\,.
	\end{aligned}
	\end{equation}
	We follow similar computations as in Theorem 5.1 \cite{LLG1D}. Define $Z^\epsilon:=X^\epsilon-X$. We write down explicitly $\delta \partial^k_x Z^{\epsilon,\otimes 2,\natural}_{s,r,t}$ as
	\begin{align*}
	\delta \partial^k_x Z^{\epsilon,\otimes 2,\natural}_{s,r,t}=-\delta I_{s,r,t}-\delta \mathbb{I}_{s,r,t}-\delta\tilde{\mathbb{I}}_{s,r,t}\,,
	\end{align*}
	where 
	\begin{align*}
	-\delta I_{s,r,t}=\sqrt{\epsilon}\sum_{n=0}^{k} \binom{k}{n} \partial_{x}^{k-n}W_{r,t}\delta(\partial_{x}^{n} X^\epsilon\odot \partial_{x}^k Z^\epsilon)_{s,r}\,,
	\end{align*}
	\begin{align*}
	-\delta \mathbb{I}_{s,r,t}=\sqrt{\epsilon}\sum_{n=0}^{k} \binom{k}{n}\partial_{x}^{k-n}\mathbb{W}_{r,t}\delta(\partial_{x}^{n} u^\epsilon\odot\partial_{x}^kZ^\epsilon)_{s,r}-\sqrt{\epsilon}\sum_{n=0}^{k} \binom{k}{n}\partial_{x}^{k-n}\delta\mathbb{W}_{s,r,t}\partial_{x}^{n} u^\epsilon_s\odot\partial_{x}^kZ^\epsilon_s\,,
	\end{align*}
	We compute $\delta \tilde{\mathbb{I}}_{s,r,t}$ in some steps. From the linearity of $\delta(\cdot)_{s,r,t}$, 
	\begin{align}\label{eq:delta_I_tilde}
	-\delta \tilde{\mathbb{I}}_{s,r,t}&= -\epsilon\delta\left(\sum_{n=0}^{k} \binom{k}{n} \partial_{x}^{k-n}W_{s,t}\partial_{x}^{n} X^\epsilon_s\otimes\sum_{m=0}^{k} \binom{k}{m}\partial_{x}^{k-m}W_{s,t}\partial_{x}^{m} X^\epsilon_s\right)_{s,r,t}\nonumber\\
	&= -\epsilon\sum_{n=0}^{k} \sum_{m=0}^{k} \binom{k}{n} \binom{k}{m} \delta\left(\partial_{x}^{k-n}W_{s,t}\partial_{x}^{n} X^\epsilon_s\otimes\partial_{x}^{k-m}W_{s,t}\partial_{x}^{m} X^\epsilon_s\right)_{s,r,t}\,.
	\end{align}
	We rewrite the main term in the sum \eqref{eq:delta_I_tilde}. 
	Introduce the shorthand notations $A:=\partial_{x}^{k-n}W,\, B:=\partial_{x}^{k-m}W,\, a:=\partial_{x}^{n} X^\epsilon_s, b:=\partial_{x}^{m} X^\epsilon_s$ and we apply \eqref{eq;algebraic_relation}.
	We aim to apply the sewing Lemma~\ref{lemma_sewing}, thus we need to represent the remainder as a linear combination of terms of $p/3$-variation or more regular. Thus we need to compensate the following terms of only $p/2$-variation 
	\begin{equation}\label{eq;terms_p_2_var}
	\begin{aligned}
	&-\sqrt{\epsilon}\sum_{n=0}^{k} \binom{k}{n}\partial_{x}^{k-n}\delta\mathbb{W}_{s,r,t}\partial_{x}^{n} u^\epsilon_s\odot\partial_{x}^kZ^\epsilon_s \\
	&-\epsilon\sum_{n=0}^{k} \sum_{m=0}^{k}\binom{k}{n}\binom{k}{m} 
	[\partial_{x}^{k-n}W_{s,r}\partial_{x}^{n} X^\epsilon_s\otimes \partial_{x}^{k-m}W_{r,t}\partial_{x}^{m} X^\epsilon_s+\partial_{x}^{k-n}W_{r,t}\partial_{x}^{n} X^\epsilon_s\otimes \partial_{x}^{k-m}W_{s,r}\partial_{x}^{m} X^\epsilon_s]\\
	&= -\sqrt{\epsilon} \partial_x^n [\delta \mathbb{W}_{s,r,t} u^\epsilon_s\odot \partial_{x}^k Z^\epsilon_s ] -\epsilon \partial_x^k[W_{r,t}X^\epsilon_s]\odot \partial_x^k[W_{s,r}X^\epsilon_s]\,,\,.
	\end{aligned}
	\end{equation}
	By adding and subtracting from $\delta(\partial_{x}^{n} X^\epsilon\odot \partial_{x}^k Z^\epsilon)_{s,r}$ the terms of at most $p$-variation, given by
	\begin{align*}
	T_1:=\sum_{\ell=0}^{n}\binom{n}{\ell} \partial_x^{n-\ell}W_{s,r}\partial_x^\ell u^\epsilon_s \odot \partial_{x}^k Z^\epsilon_s+\sqrt{\epsilon}\sum_{j=0}^{k}\binom{k}{j}\partial_{x}^{n} X^\epsilon_s\odot\partial_x^{k-j}W_{s,r}\partial_x^j X^\epsilon_s  \,.
	\end{align*}
	The terms in \eqref{eq;terms_p_2_var} cancel with
	\begin{align*}
	\sqrt{\epsilon}\sum_{n=0}^{k} \binom{k}{n} \partial_{x}^{k-n}W_{r,t}T_1&=\sqrt{\epsilon}\sum_{n=0}^{k} \binom{k}{n} \partial_{x}^{k-n}W_{r,t}\sum_{\ell=0}^{n}\binom{n}{\ell} \partial_x^{n-\ell}W_{s,r}\partial_x^\ell u^\epsilon_s \odot \partial_{x}^k Z^\epsilon_s\\
	&\quad+\epsilon\sum_{n=0}^{k} \binom{k}{n} \partial_{x}^{k-n}W_{r,t}\sum_{j=0}^{k}\binom{k}{j}\partial_{x}^{n} X^\epsilon_s \odot \partial_x^{k-j}W_{s,r}\partial_x^j X^\epsilon_s\\
	&= \sqrt{\epsilon} \partial_x^n [\delta \mathbb{W}_{s,r,t} u^\epsilon_s\odot \partial_{x}^k Z^\epsilon_s ] +\epsilon \partial_x^k[W_{r,t}X^\epsilon_s]\odot \partial_x^k[W_{s,r}X^\epsilon_s]\,,
	\end{align*}
	indeed the first sum is erased by the first sum in \eqref{eq;terms_p_2_var}. The second sum cancels with the second sum in \eqref{eq;terms_p_2_var}.
	In conclusion, the remainder can be written as a linear combination of elements of at least $p/3$-variation,
	\begin{align*}
	\delta \partial^k_x Z^{\epsilon,\otimes 2,\natural}_{s,r,t}&=\sqrt{\epsilon}\sum_{n=0}^{k} \binom{k}{n} \partial_{x}^{k-n}W_{r,t}[\delta(\partial_{x}^{n} X^\epsilon\odot \partial_{x}^k Z^\epsilon)_{s,r}-T_1]\\
	&\quad+\sqrt{\epsilon}\sum_{n=0}^{k} \binom{k}{n}\partial_{x}^{k-n}\mathbb{W}_{r,t}\delta(\partial_{x}^{n} u^\epsilon\odot\partial_{x}^kZ^\epsilon)_{s,r}\\
	&\quad+\epsilon[\partial_{x}^{k}[W_{r,t} X^\epsilon_s]\otimes\partial_x^k[ W_{r,t}X^\epsilon_s]-\partial_{x}^{k}[W_{r,t} X^\epsilon_r]\otimes\partial_x^k[ W_{r,t}X^\epsilon_r]]\,.
	\end{align*}
	Observe that each component of the remainder is multiplied by $\sqrt{\epsilon}$. We estimate $\delta \partial^k_x Z^{\epsilon,\otimes 2,\natural}_{s,r,t}$ in $(W^{1,\infty})^*$ via the rough standard machinery (we use the Sewing Lemma~\ref{lemma_sewing}), we obtain the following estimation of the remainder 
	\begin{align}\label{eq:remi_noise}
	\langle\big(\partial^k_x (X^\epsilon-X)\big)^{\otimes 2,\natural}_{s,t},\mathbf{1}\rangle\lesssim_k \sqrt{\epsilon}\omega^{1/p}_{ \mathbf{W}; H^{k+1}}(s,t)[\omega^{2/p}_{ \mathbf{W}; H^1}+\gamma](s,t)\,,
	\end{align}
	where $\gamma$ is bounded uniformly in $\epsilon\in (0,1]$ and its a linear combination of maps of $p/2$-variation defined by
	\begin{align*}
	\gamma(s,t):=\omega_{\mathbf{W};H^{k+1}}^{1/p}(s,t)\omega_{X^\epsilon;H^k}^{1/p}(s,t)+\omega_{\mathbf{W}}^{2/p}\left[(1+\epsilon)\|X^\epsilon\|^2_{L^\infty(H^{k})}+\|u^\epsilon\|^2_{L^\infty(H^{k})}+\|u\|^2_{L^\infty(H^{k})}\right]\,.
	\end{align*}
	From  Proposition~\ref{pro:drift_LLG}, the contributions in \eqref{eq:noise_terms} and \eqref{eq:remi_noise}, we derive that for a universal constant $C>0$ (which depends on $k$ and on some powers of the initial condition $\|u^0\|_{H^k}$) 
	\begin{align*}
	\delta(\|X^\epsilon-X\|^2_{H^k})_{s,t}+\int_{s}^{t}\|X_r^\epsilon-X_r&\|^2_{H^{k+1}}\dd r \leq C\gamma(s,t)\|X^\epsilon-X\|^2_{L^2(s,t;H^{k+1})}+C(t-s)\|X^\epsilon-X\|^2_{L^\infty(s,t;H^k)}\\
	&+ 	\;	\sqrt{\epsilon} C\omega^{1/p}_{ \mathbf{W}; H^{k+1}}(s,t)[\|X^\epsilon\|^2_{L^\infty(H^k)}+\|u^\epsilon\|^2_{L^\infty(H^k)}+\omega^{2/p}_{ \mathbf{W}; H^1}+\omega_{\tilde{\mathcal{D}}}]\,.
	\end{align*}
	Note that $u^\epsilon, X^\epsilon$ are uniformly bounded in $\epsilon$, thus by choosing a partition small enough such that $\gamma(s,t) C<1$ and by the rough Gronwall Lemma~\ref{lem:gronwall}, we conclude that
	\begin{align*}
	\|X^\epsilon-X\|^2_{L^\infty(H^k)}+\int_{0}^{T}\|X_r^\epsilon-X_r\|^2_{H^{k+1}}\dd r&\lesssim_{k,u^0} 		\sqrt{\epsilon} \omega^{1/p}_{ \mathbf{W}; H^{k+1}}(0,T)
	+\sqrt{\epsilon}\omega^{1/p}_{ \mathbf{W}; H^{k+1}}(0,T)\,,
	\end{align*}
	thus $X^\epsilon$ converges to $X$ in $L^\infty(H^k)\cap L^2(H^{k+1})$ with speed of convergence $\sqrt{\epsilon}$, which is optimal. 
	The pathwise convergence implies also convergence in probability and in law (since the driving path is measurable).
\end{proof}

\section{Moderate deviation principle}\label{sec:MDP}
	In this section we study the MDP for the sequence $(Y^\epsilon)_{\epsilon>0}$ introduced in \eqref{eq:Y_epsilon}, with the following scaling limits for the map $\epsilon\mapsto\lambda(\epsilon)$ for $\epsilon\mapsto0$:
	\begin{equation}\label{MDP:lambda_epsilon_right_section}
	\lambda(\epsilon)\rightarrow +\infty\,, \quad \sqrt{\epsilon}\lambda(\epsilon)\rightarrow 0\,.
	\end{equation}
	The MDP for the SPDEs is inherited from a LDP for the driving noise, from the exponential equivalence and from the contraction principle. Thus the driving noise itself should satisfy a LDP. We do not expect every (random) geometric $p$-rough path $\mathbf{W}$ to satisfy a MDP (as well as a LDP), indeed it is a property depending on the distribution of the stochastic process. In this work we consider only the case in which $\mathbf{W}\equiv (W,\mathbb{W})$ is, $\omega$-wise, the Stratonovich lift of a Brownian motion. In this case, the following exponential boundedness condition holds with respect to to the $p$-variation norm (see Corollary 13.15 of \cite{FrizVictoir}). For $p\in (2,3)$, there exists $\eta>0$ independent on time such that 
	\begin{align}\label{eq:exponential_bound}
		\mathbb{E}\left[\exp\left(\frac{\eta}{T^{1-\frac{2}{p}}}\|\mathbf{W}\|_{p-\textrm{var};[0,T]}^2\right)\right]<+\infty\,.
	\end{align}
      Analogously, property \eqref{eq:exponential_bound} holds for some classes of $p$-variation Gaussian processes as stated in Corollary~15.23~(ii) in \cite{FrizVictoir} (we always consider only $p\in (2,3)$): it is possible to consider also such noises, with some small modifications in the setting below.
     
     Let $\mathcal{Y}\subset C([0,T];\mathcal{X})$ be a Polish space endowed with the $\sigma$-algebra of the Borelians of $\mathcal{Y}$, denoted by $\mathcal{B}_{\mathcal{Y}}$. Let $E$ be a Banach space.
     Let $\Phi:\mathcal{RP}^p(\mathbb{R}^d;E)\rightarrow \mathcal{Y}$ be a continuous map, which admits a rough pathwise directional derivative in $\mathbf{0}$ in the direction $\mathbf{W}$.
     Define the family of stochastic processes $(X^\epsilon)_{\epsilon}$ defined on $(\Omega,\mathcal{F},\mathbb{P})$ and with values in $\mathcal{Y}$ by 
	\begin{align*}
	X^\epsilon(\omega):=\frac{\Phi(\tau_{\sqrt{\epsilon}}\mathbf{W}(\omega))-\Phi(\mathbf{0})}{\sqrt{\epsilon}}\,,
	\end{align*}
	$\mathbb{P}$-a.s. Since $\Phi$ admits a pathwise directional derivative in $\mathbf{0}$ in the direction $\mathbf{W}$, the sequence $(X^\epsilon)_{\epsilon}$ converges pathwise to $D\Phi[\mathbf{0}](\mathbf{W})$ in $\mathcal{Y}$ by definition.
We recall the definition of moderate deviation principle applied to the measures $(\mu_\epsilon)_{\epsilon}$, where $\mu_\epsilon:=\mathbb{P}\circ (X^\epsilon/\lambda(\epsilon))^{-1} $ is a probability measure on $(\mathcal{Y},\mathcal{B}_{\mathcal{Y}})$. We define the \textit{Cameron-Martin space} of the Brownian motion $\mathbf{W}$ as the subset of $W^{1,2}([0,T];\mathbb{R}^3)$ whose paths are null at $0$, i.e.
\begin{equation*}
\mathcal{H}:=\Bigl\{ \int_{0}^{\cdot}\dot{h}_t \dd t: \dot{h}\in L^2([0,T];\mathbb{R}^3), \,h_0=0 \Bigr\}\,,
\end{equation*}
endowed with the inner product $\langle g,h\rangle=\langle\dot{g},\dot{h} \rangle $.
We denote by $\mathbf{h}_{s,t}\equiv(\delta h_{s,t},\bar{h}_{s,t})$ the lift of $h\in \mathcal{H}$ to the space of rough paths, where the second iterated integral is defined as $\bar{h}_{s,t}:=\int_{s}^{t}\delta h_{s,r} \otimes \dot{h}_r\dd r$.

The MDP concerns the study of a particular scaling of fluctuations of the sequence $(X^\epsilon)_{\epsilon}$: more practically, it consists in the study of a LDP for the rescaled sequence $(X^\epsilon/\lambda(\epsilon))_{\epsilon}$. This is why we introduce the definition of LDP and we speak of LDPs in the sequel.
We say that a lower semi-continuous map $I:\mathcal{Y}\rightarrow [0,+\infty]$ is a \textit{good rate function} if for every $M\in [0,+\infty)$ the level sets $\{x\in \mathcal{Y}: I(x)\leq M\}$ are compact subsets of $\mathcal{Y}$.
We say that a family of measures $(\mu_\epsilon)_\epsilon$ satisfies a LDP on $\mathcal{Y}$ with speed $\lambda(\epsilon)$ and  good rate function $\tilde{I}: \mathcal{Y}\rightarrow [0,+\infty]$ 
if for every closed subset $F$ in $\mathcal{Y}$ 
\begin{align*}
\limsup_{\epsilon\rightarrow 0} \frac{1}{\lambda(\epsilon)^2}\log(\mu_\epsilon(F))\leq-\inf_{x\in F}\tilde{I}(x)\,,
\end{align*}
and if for every open subset $A$ in $\mathcal{Y}$
\begin{align*}
-\inf_{x\in A}\tilde{I}(x)\leq\liminf_{\epsilon\rightarrow 0} \frac{1}{\lambda(\epsilon)^2}\log(\mu_\epsilon(A))\,.
\end{align*}
We say that a sequence $(Y^\epsilon)_{\epsilon}$ satisfies a LDP on $\mathcal{Y}$ with good rate function $\tilde{I}$ if the sequence $(\mu^\epsilon)_{\epsilon}$ defined as $\mu^\epsilon:=\mathbb{P}\circ (Y^\epsilon)^{-1} $ satisfies a LDP on $\mathcal{Y}$ with good rate function $\tilde{I}$.

A classical way to show LDPs and MDPs for SPDEs is the weak convergence approach from Dupuis  and Ellis \cite{Dupuis_Ellis}. Instead, as noted in \cite{Ledoux_Qian_Zhang_p_var}, rough path calculus allows to use directly the contraction principle on the It\^o-Lyons map and avoid the weak convergence approach to establish a LDP. We use this consideration coupled with the exponential equivalence in this context. We state the main result of this section.

\begin{theorem}\label{th:MDP_abstract}
	Let $\mathcal{X}$ be a Banach space and let $\mathcal{Y}\subset C([0,T];\mathcal{X})$ be a Polish space endowed with the Borel $\sigma$-algebra $\mathcal{B}_{\mathcal{Y}}$. 
	Let $(\mathbf{W},\mathbf{g})$ be $p$-geometric compatible directions constructed as in Construction \ref{constr:Constr_p_q} and denote by $\mathbf{Z}$ their joint lift, where $\mathbf{W}\equiv(W,\mathbb{W})$ is the Stratonovich lift of a Brownian and $\mathbf{g}$ is a random $q$-geometric rough path with $q\in [1,2)$. 
	 
	Consider a continuous map $\Phi:\mathcal{RP}^p(\mathbb{R}^d;E)\rightarrow \mathcal{Y}$ so that
	\begin{enumerate}
		\item the sequence $(X^{\epsilon,\mathbf{g}})_{\epsilon}$ defined by
		\begin{align}\label{eq:X_eps_G_abstract_MDP}
		X^{\epsilon,\mathbf{g}}:=\frac{\Phi(\{\mathbf{g}+\tau_{\sqrt{\epsilon}}\mathbf{W}\})-\Phi(\mathbf{g})}{\sqrt{\epsilon}}\, 
		\end{align}
		converges $\mathbb{P}$-a.s. to $X^\mathbf{g}:=D\Phi[\mathbf{g}](\mathbf{W})$ in $\mathcal{Y}$ with speed of convergence $\sqrt{\epsilon}$, namely there exists a positive random variable independent on $\mathbf{Z}$ such that $\mathbb{E}[\exp(C^2)]<+\infty$ such that $\mathbb{P}$-a.s.
		\begin{align*}
		\|X^{\epsilon,\mathbf{g}}-X^\mathbf{g}\|_{\mathcal{Y}}\leq C\omega_{\mathbf{Z}}^{1/p} \sqrt{\epsilon}\, .
		\end{align*}
		\item the map $\mathbf{Z}\mapsto X^\mathbf{g}=D\Phi[\mathbf{g}](\mathbf{W})$ is continuous.
	\end{enumerate}
	Define the sequence of probability measures  $\tilde{\mu}_\epsilon:=\mathbb{P}\circ (X^{\mathbf{g},\epsilon}/\lambda(\epsilon))^{-1}$ on $(\mathcal{Y},\mathcal{B}_\mathcal{Y})$. 
	Then the sequence $(\tilde{\mu}_\epsilon)_{\epsilon}$ satisfies a LDP on $\mathcal{Y}$ with speed of convergence $\lambda(\epsilon)$ and good rate function $\tilde{I}: \mathcal{Y}\rightarrow [0,+\infty]$ given by
	\begin{align*}
	\tilde{I}(f)=\inf\bigg(\frac{1}{2}\int_{0}^{T}|h_s|^2 \dd s:h\in \mathcal{H}, \; D\Phi[\mathbf{g}](\mathbf{h})=f\bigg)\,,
	\end{align*}
	with the convention $\inf\emptyset=+\infty$.
\end{theorem}

\subsubsection{Proof of Theorem~\ref{th:MDP_abstract}.}
In this subsection we prove Theorem~\ref{th:MDP_abstract}. Define the sequence of probability measures $\mu_\epsilon:=\mathbb{P}\circ (X^\mathbf{g}/\lambda(\epsilon))^{-1}$.
The proof is then divided into three steps:
\begin{enumerate}
	\item $(\mu_\epsilon)_{\epsilon}$ satisfies a LDP with speed $\lambda(\epsilon)^2$ and a good rate function $\tilde{I}:\mathcal{Y}\rightarrow [0,+\infty]$.
	\item  $(\mu_\epsilon)_{\epsilon}$ is $\lambda(\epsilon)^{-1}-$exponentially equivalent to $(\tilde{\mu}_\epsilon)_{\epsilon}$. We recall what we mean by exponential equivalence: we say that $(X^{\mathbf{g},\epsilon}/\lambda(\epsilon))_{\epsilon}$ is $\lambda(\epsilon)^{-1}-$exponentially equivalent to $(X^{\mathbf{g}}/ \lambda(\epsilon))_{\epsilon}$ if for all $\delta>0$
	\begin{align*}
	\limsup_{\epsilon\rightarrow 0} \lambda(\epsilon)^{-2}\log\mathbb{P}\bigg(\frac{\|X^{\mathbf{g},\epsilon}-X^{\mathbf{g}}\|_{\mathcal{Y}}}{\lambda(\epsilon)}>\delta\bigg)=-\infty\,.
	\end{align*}
	\item We apply the following result (see Theorem 4.2.13 in \cite{Dembo_Zeitouni}) and conclude.  
	\begin{theorem}\label{th:teo_MDP_exp_abst}
		If $(\mu_\epsilon)_\epsilon$ satisfies a LDP with speed $\lambda(\epsilon)^2$ with good rate function $\mathcal{I}(\cdot)$, which are exponentially equivalent to $(\tilde{\mu}_\epsilon)_\epsilon$, then the same LDP holds for $(\tilde{\mu}_\epsilon)_\epsilon$.
	\end{theorem}
	Indeed Theorem~\ref{th:teo_MDP_exp_abst} asserts that the LDP for $(\mu^\epsilon)_{\epsilon}$ passes to $(\tilde{\mu}^\epsilon)_{\epsilon}$.
\end{enumerate}

Thus from Theorem~\ref{th:teo_MDP_exp_abst} the sequence $(\tilde{\mu}_\epsilon)_\epsilon$ satisfies a LDP on $\mathcal{Y}$ with speed $\lambda(\epsilon)^{2}$ and rate function $\tilde{I}$.
In Lemma~\ref{lemma:LDP_X_abst} the first step, which is a consequence of the following contraction principle (Appendix~C, Theorem~C.6 in \cite{FrizVictoir})
\begin{theorem}\label{th:contraction_principle}
	(Contraction principle) Let $\mathcal{Z}$ and $\mathcal{Y}$ be Hausdorff topological spaces. 
	Suppose that $f:\mathcal{Z}\rightarrow\mathcal{Y}$ is a continuous measurable map. 
	If a set of probability measures $(\mu_\epsilon)_{\epsilon>0}$ on $\mathcal{Z}$ satisfies a large deviation principle on $\mathcal{Z}$ with good rate function $\mathcal{I}$, then the image measures $(\mu_\epsilon\circ f^{-1})_{\epsilon>0}$ satisfy a large deviation principle on $\mathcal{Y}$ with good rate function
	\begin{align*}
	\mathcal{J}(y):= \inf \big(\mathcal{I}(x) :\,x\in \mathcal{Z}\;\mathrm{and}\;f(x)=y\big)\,.
	\end{align*}
\end{theorem}
\begin{lemma}\label{lemma:LDP_X_abst}
	The sequence of probability measures $(\mu_\epsilon)_\epsilon$ satisfies a large deviation principle on $\mathcal{Y}$ with speed $\lambda(\epsilon)^{2}$ and with good rate function
	\begin{align*}
	\tilde{I}(f)=\inf\bigg(\frac{1}{2}\int_{0}^{T}|h(s)|^2 ds:h\in \mathcal{H}, \; D\Phi[\mathbf{g}](\mathbf{h})=f\bigg)\,.
	\end{align*}
\end{lemma}
	\begin{proof} Fix $\mathbf{g}\in \mathcal{RP}^q(\mathbb{R^d}; E)$.
		Since the joint lift $\mathbf{Z}$ associated to $(\mathbf{W},\mathbf{g})$ is constructed via Young's integration, the map $\mathbf{h}\mapsto\mathbf{Z}$ is continuous.
		By assumption 2., the map $\mathbf{Z}\mapsto X^\mathbf{g}$ is continuous. Thus the map $\mathbf{h}\mapsto D\Phi[\mathbf{g}](\mathbf{h}) $ is continuous as a map from $\mathcal{RP}^p(\mathbb{R}^d;E)$ to $\mathcal{Y}$. 
	
	From the rough paths equivalent of Schiller's theorem for Gaussian processes,  $(\tau_{\sqrt{\epsilon}}\mathbf{W})_{\epsilon}$ satisfies a LDP  on $\mathcal{RP}^p(\mathbb{R}^d;E)$ with speed $\epsilon^{-1}$ and good rate function $\mathcal{I}$. From the assumptions on the scaling in \eqref{MDP:lambda_epsilon_right_section}, also the sequence $(\tau_{\lambda(\epsilon)^{-1}}\mathbf{W})_{\epsilon}$ satisfies a LDP  on $\mathcal{RP}^p(\mathbb{R}^d;E)$ with speed $\lambda(\epsilon)^2$ and good rate function $\mathcal{I}$.
	
	Since the map $D\Phi[\mathbf{g}](\,\cdot\,):\mathcal{RP}^p(\mathbb{R}^d;E)\longrightarrow\mathcal{Y}$ is continuous and  the sequence $(\tau_{\lambda(\epsilon)^{-1}}\mathbf{W})_{\epsilon}$ satisfies a LDP  on $\mathcal{RP}^p(\mathbb{R}^d;E)$, we can apply the contraction principle (Theorem~\ref{th:contraction_principle}) and obtain a LDP also for $(D\Phi[\mathbf{g}](\tau_{\lambda(\epsilon)^{-1}}\mathbf{W}))_{\epsilon}$.
\end{proof}
W prove the second step in the following theorem.
\begin{lemma}\label{lemma:exp_equiv}
	The sequence $(X^\mathbf{g}/\lambda(\epsilon))_{\epsilon}$ is $\lambda(\epsilon)^{-1}$-exponentially equivalent to $(X^{\mathbf{g},\epsilon}/\lambda(\epsilon))_\epsilon$ in $\mathcal{Y}$, where $X^{\mathbf{g},\epsilon}$ is defined in \eqref{eq:X_eps_G_abstract_MDP} and $X^\mathbf{g}=D\Phi[\mathbf{g}](\mathbf{W})$.
\end{lemma}
\begin{proof}
	By assumption, the sequence $(X^{\mathbf{g},\epsilon})_{\epsilon}$ converges to $D\Phi[\mathbf{g}](\mathbf{W})$ in $\mathcal{Y}$ with speed of convergence $\sqrt{\epsilon}$. Hence
	\begin{align*}
	\limsup_{\epsilon\rightarrow 0} \lambda(\epsilon)^{-2}\log\mathbb{P}\bigg(\frac{\|X^{\mathbf{g},\epsilon}-X^\mathbf{g}\|_{\mathcal{Y}}}{\lambda(\epsilon)}>\delta\bigg)\leq \limsup_{\epsilon\rightarrow 0} \lambda(\epsilon)^{-2}\log\mathbb{P}\bigg(\frac{\sqrt{\epsilon}C\omega_{\mathbf{Z}}^{1/p}}{\lambda(\epsilon)}>\delta\bigg)\,.
	\end{align*}
	From the Markov's inequality, we can bound  the logarithm of the probability above by
	\begin{align}
	\log\mathbb{P}\bigg(\frac{\sqrt{\epsilon}C\omega_{\mathbf{Z}}^{1/p}}{\lambda(\epsilon)}>\delta\bigg)
	&\leq\log\mathbb{P}\bigg(\frac{\sqrt{\epsilon}[C^2+\omega_{\mathbf{Z}}^{2/p}]}{2\lambda(\epsilon)}>\delta\bigg)\nonumber\\
	&\leq\log\bigg(\exp\bigg(-\frac{\lambda(\epsilon)2\delta T^{1-\frac{p}{2}}}{\sqrt{\epsilon}\eta }\bigg)\left[\mathbb{E}\big[\frac{\eta}{T^{1-\frac{p}{2}}}\exp(\omega_{\mathbf{Z}}^{1/p})\big] +\mathbb{E}[\exp(C^2)]\right]\bigg) \nonumber\\
	&=-\frac{2\lambda(\epsilon)\delta}{\sqrt{\epsilon}} \frac{T^{1-\frac{p}{2}}}{\eta}+\log\left(\mathbb{E}\big[2\exp(\omega_{\mathbf{Z}}^{2/p})\big]+2\mathbb{E}\big[\exp(C^2) \big] \right)\,. \label{eq:proof_abs_MDP_RHS}
	\end{align}
	Since the constant $\eta>0$ is independent on $\epsilon$ and depends only on the joint lift $\mathbf{Z}$, it does not influence the limit for $\epsilon\mapsto 0$. 
	From the exponential boundedness of $\mathbf{Z}$, the right hand side of the equality in \eqref{eq:proof_abs_MDP_RHS} converges to $-\infty$ for $\epsilon$  converging to $0$. Thus it follows the $\lambda(\epsilon)^{-1}$-exponential equivalence of $(X^{\mathbf{g},\epsilon})_{\epsilon}$ and $(X^{\mathbf{g}}/\lambda(\epsilon))_{\epsilon}$.
\end{proof}
The conclusion of Theorem~\ref{th:MDP_abstract} follows by applying Theorem~\ref{th:teo_MDP_exp_abst}.
\subsection{MDP for SPDEs with multiplicative linear noise.}
We apply Theorem~\ref{th:MDP_abstract} to the SPDEs discussed in the CLT section.

\subsection{Moderate deviations from the CLT centred in $\mathbf{0}$ for SPDEs, $X=D\Phi[\mathbf{0}](\mathbf{W})$.}\label{sec:MDP_in_0_spde}
We rewrite Theorem~\ref{th:MDP_abstract} in the framework of  the heat equation considered in Section~\ref{sec:CLT_SPDEs} and for the reaction-diffusion equation considered in Section~\ref{section:reaction_diffusion}. Consider the map
\begin{align*}
	\Phi_i:\mathcal{RD}^p(\mathbb{R}^d;E)&\longrightarrow \mathcal{Y}:=C([0,T];L^2)\cap L^2(H^2)\cap L^\infty(H^1)\\
	\mathbf{W}&\longmapsto \Phi_i(\mathbf{W})=\pi_i(u^0,\mathbf{W})\,,
\end{align*}
for $i=1,2$. Recall the notations: we denote by $\pi_1(u^0,\mathbf{W})$ the unique solution to the heat equation driven by a multiplicative noise $\mathbf{W}$ started in $u^0\in H^1$ and $\Phi_1$ is the corresponding It\^o-Lyons map. We denote by $\pi_2(u^0,\mathbf{W})$ the unique solution to the stochastic reaction-diffusion equation driven by a multiplicative noise $\mathbf{W}$ started in $u^0\in H^1$ and $\Phi_2$ is the corresponding It\^o-Lyons map.
That is, for $i=1,2$, $\Phi_i(\mathbf{W})$ is the unique solution to
\begin{align*}
\delta \Phi_i(\mathbf{W})_{s,t}=\int_{s}^{t}b_i(\Phi_i(\mathbf{W})_r)\dd r+W_{s,t}\Phi_i(\mathbf{W})_s+\mathbb{W}_{s,t}\Phi_i(\mathbf{W})_s+\Phi^{i,\natural}_{s,t}\,,
\end{align*}
with $\Phi_i(\mathbf{W})_0=u^0$ and $b_i(\cdot)$ is the drift to the corresponding equations.
Consider the $p-$geometric compatible direction $(\mathbf{W},\mathbf{0})$ as in Theorem~\ref{th:MDP_abstract}, where we denote by $\mathbf{Z}$ the joint lift.
 Define $X^{\epsilon,i}:=\Phi_i(\tau_{\sqrt{\epsilon}}\mathbf{W})-\Phi_i(\mathbf{0})/\sqrt{\epsilon}$ and $X^i$ to be the unique solution to 
\begin{align*}
\delta X^i_{s,t}=\int_{s}^{t}b_i'(\Phi_i(\mathbf{0})_r)X^i_r\dd r+W_{s,t}\Phi_i(\mathbf{0})_s+X^{i,\natural}_{s,t}\,,
\end{align*}
where we recall that $b_1'(\Phi_1(\mathbf{0}))X^1=b_1(X^1)$ from the linearity of the drift and for $b_2'(\Phi_2(\mathbf{0}))X^2$ is discussed in Section~\ref{section:reaction_diffusion}.
For $i=1,2$, the map 
\begin{align*}
D\Phi_i[\mathbf{0}](\cdot):\mathcal{RD}^p(\mathbb{R}^d;E)&\longrightarrow \mathcal{Y}:=C([0,T];L^2)\cap L^2(H^2)\cap L^\infty(H^1)\\
\mathbf{W}&\longmapsto D\Phi_i[\mathbf{0}](\mathbf{W})=X^i\,,
\end{align*}
 is continuous, from Proposition~\ref{pro:continuity_limit_equation} for both of the equations (the drifts are discussed in the respective sections).
From Theorem~\ref{teo:CLT}, the sequence $(X^{\epsilon,1})_{\epsilon}$ converges to $D\Phi_1[\mathbf{0}](\mathbf{W})$ in $L^\infty(H^1)\cap L^2(H^2)\cap \mathcal{V}^p(L^2)$ with speed of convergence $\sqrt{\epsilon}$. From the considerations in Section~\ref{section:reaction_diffusion}, also the sequence $(X^{\epsilon,2})_{\epsilon}$ converges to $D\Phi_2[\mathbf{0}](\mathbf{W})$ in $L^\infty(H^1)\cap L^2(H^2)\cap \mathcal{V}^p(L^2)$ with speed of convergence $\sqrt{\epsilon}$.
Thus the hypothesis of Theorem~\ref{th:MDP_abstract} are fulfilled and the MDP follows, for $i=1,2$.
\begin{theorem}\label{th:MDP}
	Let $\mathbf{W}$ be a random $p$-geometric rough driver on a probability space $(\Omega,\mathcal{F},\mathbb{P})$ which fulfils the hypothesis of Theorem~\ref{th:MDP_abstract} and so that
    $(\lambda(\epsilon)^{-1}\mathbf{W})_{\epsilon}$ satisfy a large deviation principle with speed $\lambda(\epsilon)^{-2}$ and good rate function 
	\begin{align*}
	\mathcal{I}(h)=\frac{1}{2}\int_{0}^{T}|h(s)|^2 \dd s,\quad h\in \mathcal{H}\,.
	\end{align*}
	In the notations of this section, the sequence $((\Phi_i(\tau_{\sqrt{\epsilon}}\mathbf{W})-\Phi_i(\mathbf{0}))/\sqrt{\epsilon} \lambda(\epsilon))_\epsilon$ satisfies a large deviation principle on $L^\infty(H^1)\cap L^2(H^2)\cap \mathcal{V}^p(L^2)$ with speed $\lambda(\epsilon)^2$ and with good rate function  
	\begin{align*}
	\tilde{I}(f)=\inf\bigg(\frac{1}{2}\int_{0}^{T}|h(s)|^2 \dd s:h\in \mathcal{H}, \; X^{h,i}=f\bigg)\,, \quad \inf \emptyset:=+\infty\,,
	\end{align*}
	where $X^{h,i}$ is the unique solution to the skeleton equation
	\begin{align*}
	\delta X^{h,i}_{s,t}=\int_{s}^{t}b_i'(\Phi_i(\mathbf{0})_r)X^{h,i}_r \dd r+h_{s,t}\Phi_i(\mathbf{0})_s+X^{h,i,\natural,}_{s,t}\,.
	\end{align*}
\end{theorem}
\begin{remark}
	\textbf{(MDP for SPDEs driven by an It\^o integral)}: in Remark~\ref{remark:ito_CLT}, we observe that the pathwise CLT holds also for a stochastic reaction-diffusion equation driven by an It\^o integral, with the Brownian motion rescaled by $\sqrt{\epsilon}$. We show that the convergence occurs with speed $\sqrt{\epsilon}$: we can ask whether a MDP is also inherited in a similar fashion. We adopt the notations of Section~\ref{section:reaction_diffusion} and we recall them briefly. We denote by $y^{i,\epsilon}$ the unique solution to the SPDE driven by a It\^o integral, where the Brownian motion is rescaled by a $\sqrt{\epsilon}$ factor. Denote by $u^{i,\epsilon}=\Phi_i(\tau_{\sqrt{\epsilon}}\mathbf{W}) $ the corresponding SPDE driven by a Stratonovich equation and by $y^i$ the associated deterministic equation.
	If we have a speed of convergence to the CLT, we can also inherit a MDP. Indeed, 
	\begin{align*}
	&\quad\mathbb{P}\left( \left|\frac{1}{\lambda(\epsilon)}\left(\frac{y^{i,\epsilon}-y^i}{\sqrt{\epsilon}}-D\Phi_i[\mathbf{0}](\mathbf{W})\right)\right|>\delta\right)\\
	&\leq \mathbb{P}\left( \left|\frac{1}{\lambda(\epsilon)}\left(\frac{y^{i,\epsilon}-u^{i,\epsilon}}{\sqrt{\epsilon}}+\frac{u^{i,\epsilon}-y^i}{\sqrt{\epsilon}}-D\Phi_i[\mathbf{0}](\mathbf{W})\right)\right|>\delta\right)\\
	&\leq \mathbb{P}\left( \left|\frac{1}{\lambda(\epsilon)}\left(\frac{y^{i,\epsilon}-u^{i,\epsilon}}{\sqrt{\epsilon}}\right)\right|+\left|\frac{1}{\lambda(\epsilon)}\left(\frac{u^{i,\epsilon}-y^i}{\sqrt{\epsilon}}-D\Phi_i[\mathbf{0}](\mathbf{W})\right)\right|>\delta\right)\\
	& \leq \mathbb{P}\left( \left|\frac{1}{\lambda(\epsilon)}\left(\frac{y^{i,\epsilon}-u^{i,\epsilon}}{\sqrt{\epsilon}}\right)\right|>\delta\right)+\mathbb{P}\left(\left|\frac{1}{\lambda(\epsilon)}\left(\frac{u^{i,\epsilon}-y^i}{\sqrt{\epsilon}}-D\Phi_i[\mathbf{0}](\mathbf{W})\right)\right|>\delta\right)\,.
	\end{align*}
	The first probability converges to $0$, since $(y^{i,\epsilon}-u^{i,\epsilon})/\sqrt{\epsilon}\lambda(\epsilon)$ converges also pathwise to $0$.  The second converges to $-\infty$ and therefore the exponential equivalence principle applies also in this case.
\end{remark}

\subsubsection{Moderate deviations from a CLT centred in $\mathbf{g}$ for SPDEs, $X^\mathbf{g}=D\Phi[\mathbf{g}](\mathbf{W})$.}\label{sec:MDP_in_g_spde}
We consider an application of the MDP to the heat equation in Section~\ref{sec:CLT_SPDEs} and we adopt the notations of Section~\ref{sec:MDP_in_0_spde} and of Theorem~\ref{th:MDP_abstract}. We show a MDP from a central limit theorem which is not centred in the null path $\mathbf{0}$.  We state the some additional properties needed on the compatible directions. In this Section~\ref{sec:MDP_in_g_spde}, we consider a random $p$-compatible direction $(\mathbf{g},\mathbf{W})$ with joint lift $\mathbf{Z}$, which fulfils the hypothesis of Theorem~\ref{th:MDP_abstract}. We recall that, due to the construction of the second iterated integrals of the joint lift via the Young's integration, the exponential boundedness in \eqref{eq:exponential_bound} holds for the joint lift $\mathbf{Z}$ (it is inherited from the exponential bound of the rough paths lift of the Brownian motion). Moreover, the map $\mathbf{W}\mapsto \mathbf{Z}$ is locally Lipschitz continuous.

From Theorem~\ref{teo:CLT}, the sequence $(Y^{\mathbf{g},\epsilon})_{\epsilon}$ converges pathwise in $\mathcal{Y}$ for $\epsilon \rightarrow 0$ to a limit $Y^\mathbf{g}$ with speed $\sqrt{\epsilon}$, where $Y^\mathbf{g}$ is the unique solution in the sense of Definition~\ref{def:sol_additive} to equation
\begin{align}\label{eq:Dphi_Y_not_centered}
\delta Y^{\mathbf{g}}_{s,t}=\int_{s}^{t}b'(u^\mathbf{g}_r)Y^{\mathbf{g}}_r \dd r+W_{s,t}u^{\mathbf{g}}_s+g_{s,t}Y^{\mathbf{g}}_{s}+([gW]+[Wg])u^{\mathbf{g}}_s+[gg]_{s,t}Y^{\mathbf{g}}_s+Y^{\mathbf{g},\natural,}_{s,t}\,,
\end{align}
and $u^{\mathbf{g}}=\Phi(\mathbf{g})$ is the unique solution to the heat equation.
\begin{theorem}\label{th:MDP_non_central}
	Fix $(\mathbf{g},\mathbf{W})$ be  a random $p$-compatible direction with joint lift $\mathbf{Z}$, which fulfils the hypothesis of Theorem~\ref{th:MDP_abstract} and recall that the sequence
	$(\lambda(\epsilon)^{-1}\mathbf{W})_{\epsilon}$ satisfies a LDP with speed $\lambda(\epsilon)^{-2}$ and good rate function 
	\begin{align*}
	\mathcal{I}(h)=\frac{1}{2}\int_{0}^{T}|h(s)|^2 ds,\quad h\in \mathcal{H}\, ;\quad \mathcal{I}(h)=+\infty \quad \mathrm{else}.
	\end{align*}
	Under the assumptions and the notations of this Section~\ref{sec:MDP_in_g_spde}, the sequence $(X^\epsilon)_{\epsilon}$ satisfies a MDP on $\mathcal{Y}$ with speed $\lambda(\epsilon)^2$ and with good rate function $\tilde{I}$, i.e. 
	the sequence $((\Phi(\{\tau_{\sqrt{\epsilon}}\mathbf{W}+\mathbf{g}\})-\Phi(\mathbf{g}))/\sqrt{\epsilon} \lambda(\epsilon))_\epsilon$ satisfies a LDP on $\mathcal{Y}$ with speed $\lambda(\epsilon)^2$ and with good rate function  
	\begin{align*}
	\tilde{I}(f)=\inf\bigg(\frac{1}{2}\int_{0}^{T}|h(s)|^2 ds:h\in \mathcal{H}, \; Y^{h,\mathbf{g}}=f\bigg)\,, 
	\end{align*}
	where $Y^{h,\mathbf{g}}$ is the unique solution to the skeleton equation
	\begin{align*}
	\delta Y^{h,\mathbf{g}}_{s,t}=\int_{s}^{t}b'(u^\mathbf{g}_r)Y^{h,\mathbf{g}}_r \dd r+h_{s,t}u^{\mathbf{g}}_s+g_{s,t}Y^{h,\mathbf{g}}_{s}+([gh]+[hg])u^{\mathbf{g}}_s+[gg]_{s,t}Y^{h,\mathbf{g}}_s+Y^{h,\mathbf{g},\natural,}_{s,t}\,.
	\end{align*}
\end{theorem}
\begin{proof}
We verify the hypothesis of Theorem~\ref{th:MDP_abstract}. Recall that $Y^{\mathbf{g}}$ is the unique solution to equation \eqref{eq:Dphi_Y_not_centered} and consider $Z^{\mathbf{g},\epsilon}:=D\Phi [\mathbf{g}](\tau_{1/\lambda(\epsilon)}\mathbf{W})$, which is the unique solution to 
\begin{equation*}
\begin{aligned}
\delta Z^{\mathbf{g},\epsilon}_{s,t}=\int_{s}^{t}b'(u^\mathbf{g}_r)Z^{\mathbf{g},\epsilon}_r \dd r&+\frac{1}{\lambda(\epsilon)}W_{s,t}u^{\mathbf{g}}_s+g_{s,t}Z^{\mathbf{g},\epsilon}_{s}\\
&+\frac{1}{\lambda(\epsilon)}([gW]+[Wg])_{s,t}u^{\mathbf{g}}_s+[gg]_{s,t}Z^{\mathbf{g},\epsilon}_s+Z^{\mathbf{g},\epsilon,\natural}_{s,t}\,,
\end{aligned}
\end{equation*}
where $u^{\mathbf{g}}=\Phi(\mathbf{g})$. From Proposition~\ref{pro:continuity_limit_equation}, the map $\mathbf{W}\mapsto D\Phi[\mathbf{g}](\mathbf{W})$ is continuous. 
From Theorem~\ref{teo:CLT}, the speed of convergence of $(X^{\epsilon,\mathbf{g}})$ to $D\Phi[\mathbf{g}](\mathbf{W})$ is $\sqrt{\epsilon}$, with respect to the norm $\mathcal{Y}$. Since all the assumptions are fulfilled, the MDP follows.
\end{proof}

\subsubsection{MDP for the stochastic Landau-Lifschitz-Gilbert equation on a one dimensional domain.}
Assume the conditions on the compatible directions in Theorem~\ref{th:MDP_abstract} and fix $k\in\mathbb{N}$. In the notations and hypothesis of Section~\ref{sec:LLG}, consider the It\^o-Lyons map $\Phi_3$ associated to the one dimensional LLG, namely
\begin{align*}
\Phi_3:\mathcal{RD}_S^p(\mathbb{R}^d;H^{k+1})&\longrightarrow L^\infty(H^1)\cap L^2(H^2)\cap \mathcal{V}^p(L^2) \subset C([0,T];H^{-1})\\
\mathbf{W}&\longmapsto \Phi_3(\mathbf{W})=\pi_3(u^0,\mathbf{W})\,,
\end{align*}
which is locally Lipschitz continuous as proved in \cite{LLG1D}. From Theorem~\ref{th:local_lip_cont_der_LLG}, the map $\Phi_3$ admits a pathwise directional derivative $D\Phi_3[\mathbf{0}](\mathbf{W})$ in $\mathbf{0}$ in the direction $\mathbf{W}$ and the sequence $(X^\epsilon)_{\epsilon}$ converges to $D\Phi_3[\mathbf{0}](\mathbf{W})$ in $L^\infty(H^k)\cap L^2(H^{k+1})\cap\mathcal{V}^p(H^{k-1})$ with speed $\sqrt{\epsilon}$. Since $D\Phi_3[\mathbf{0}](\mathbf{W})$  is the strong solution to a semilinear stochastic partial differential equation with Young additive noise (the well posedness follows from e.g. \cite{hocquet2017energy}), a modification of Proposition~\ref{pro:lip_semilinear} shows that the map $\mathbf{Z}\mapsto D\Phi_3[\mathbf{0}](\mathbf{W})$ is locally Lipschitz continuous (one can follow the proof of the continuity of the It\^o-Lyons map of the LLG for the proof of the higher order regularity). Hence we are in the context of Theorem~\ref{th:MDP_abstract} and we can state a MDP, with $\mathcal{Y}:=L^\infty(H^k)\cap L^2(H^{k+1})\cap\mathcal{V}^p(H^{k-1})$.
\begin{theorem}
	The family of probability measures $(\mu_\epsilon)_\epsilon$ on $(\mathcal{Y},\mathcal{B}_{\mathcal{Y}})$ defined by $\mu_\epsilon(\,\cdot\,)=\mathbb{P}(X^\epsilon\in \cdot)$, satisfies a LDP on $\mathcal{Y}$ with speed $\lambda(\epsilon)^2$ and with good rate function
	\begin{align*}
	\tilde{I}(f)=\inf\bigg(\frac{1}{2}\int_{0}^{T}|h_s|^2 \dd s:h\in \mathcal{H}, \; D\Phi_3[\mathbf{0}](\mathbf{h})=f\bigg)\,,
	\end{align*}
	with the convention $\inf \emptyset:=+\infty$.
\end{theorem}	

\appendix

\section{Appendix A: useful results from rough path theory.} \label{section:appendix_RP}
We refer to the monographs \cite{FrizHairer,FrizVictoir} for details on rough integration theory. We recall some basic facts.
\paragraph{Young integration} Given two paths $X\in\mathcal{V}^p(\mathbb{R}^d)$ and $Y\in\mathcal{V}^q(\mathcal{L}(\mathbb{R}^{d},\mathbb{R}^n))$, with $1/p+1/q>1$ and $p>2$, we can define the \textit{Young integral} of $Y$ against $X$ as
\begin{align}\label{eq:def_Young_integral}
\mathcal{I}^{\mathcal{Y}}(Y,X):=\lim_{|\mathcal{P}|\rightarrow 0}\sum_{[s,t]\in\mathcal{P}}Y_sX_{s,t}\,,
\end{align}
where $\mathcal{P}$ denotes a finite partition of $[0,T]$ and $|\mathcal{P}|$ the length of the longest interval in $\mathcal{P}$. The main property of this integral is that the map
\begin{align*}
(Y,X)\mapsto\mathcal{I}^{\mathcal{Y}}(X,Y)
\end{align*}
is bilinear and continuous as a map from $\mathcal{V}^p(\mathbb{R}^d)\times \mathcal{V}^q(\mathbb{R}^d)$ to $\mathcal{V}^p(\mathbb{R}^d)$ (see Proposition 6.12 in \cite{FrizVictoir}). However Young's integration theory does not allow to consider the integral of a path of a Brownian motion against itself: indeed the paths of Brownian motion are known to be merely of $2+$ variation. The classical It\^o integration theory allows to integrate a Brownian motion against itself (as a limit in probability), but it does not preserve the continuity of the map which to the noise and the integrand associated the integral.
In contrast the theory of rough paths allows to obtain an integral of a Brownian motion against itself and to preserve the continuity of the map that to the noise and the integrand associates the integral. 
\subsection{Useful results on rough paths.}
A fundamental result, which allows for the construction of the rough integral, is the classical sewing Lemma.
\begin{lemma}[Sewing lemma]\label{lemma_sewing}
	\label{lemma-lambda}Fix an interval $I$, a Banach space $E$ and a parameter
	$\zeta > 1$. Consider a map $G : I^3 \to E$ such that $ G \in \{\delta H ; \, H: I^2\to E\}$ and for every $s < u < t \in I$,
	\[ |G_{s u t} | \leqslant \omega (s, t)^{\zeta}, \]
	for some control $\omega$ on $I$. Then there exists a unique element
	$g \in \mathcal{V}_2^{1 / \zeta} (I ; E)$ such that $\delta g = G$
	and for every $s < t \in I$,
	\begin{equation*}
	| g_{s t} | \leqslant C_\zeta\omega (s, t)^{\zeta}\,,
	\end{equation*}
	for some universal constant $C_{\zeta}$.
\end{lemma}
We recall the statement of the rough Gronwall's Lemma (\cite{deya2016priori}), which can be considered as counterpart of the classical statement in the $p$-variation setting.
\begin{lemma}[Rough Gronwall lemma]
	\label{lem:gronwall}
	Let $E\colon[0,T]\to \R_+$ be a path such that there exist constants $\kappa,\ell>0,$ a super-additive map $\varphi $ and a control $\omega $ such that:
	\begin{equation*}
	\delta E_{s,t}\leq \left(\sup_{s\leq r\leq t} E_r\right)\omega (s,t)^{\kappa }+\varphi (s,t)\,,
	\end{equation*}
	for every $s\le t \in [0,T]$ under the smallness condition $\omega (s,t)\leq \ell$. Then, there exists a constant $\tau _{\kappa,\ell}>0$ such that
	\begin{equation*}
	\sup_{0\leq t\leq T}E_t\leq \exp\left(\frac{\omega (0,T)}{\tau _{\kappa,\ell}}\right)\left[E_0+\sup_{0\leq t\leq T}\left|\varphi (0,t)\right|\right].
	\end{equation*}
\end{lemma}

We introduce a product formula, which is the equivalent of the Stratonovich product rule in this framework: we employ Proposition 4.1 in \cite{hocquet2018ito} and the modification introduced in \cite{LLG1D}.
\begin{proposition}[Product formula]
	\label{pro:product}
	Fix a open bounded set $D\subset \mathbb{R}^d$, an integer $n\ge1$
	and let $a=(a^{i})_{i=1}^n\colon [0,T]\to L^2(D;\R^n)$, (resp. $b=(b^i)_{i=1}^n\colon [0,T]\to L^2(D;\R^n)$), be a bounded path, given as a weak solution of the system
	\begin{align*}
	\delta a_{s,t}&=\int_{s}^{t}[\partial_i f^1+f^0]\dd t+A_{s,t}a_s+\mathbb{A}_{s,t}a_s+a^{\natural}_{s,t}\,,\\
	( \mathrm{resp}\quad\delta b_{s,t}&=\int_{s}^{t} [\partial_i g^1+g^0]\dd t+B_{s,t}b_s+\mathbb{B}_{s,t}b_s+b^{\natural}_{s,t})\,,
	\end{align*}
	for $i\in \{0,\cdots d\}$ on $[0,T]\times D$, 
	for some $f^0,f^1, g^0,g^1\in L^1(L^1)$. Assume further that for $j=0,1$ the pointwise products $f^j(\cdot )\partial_i b(\cdot -x), \, \,\partial_i a(\cdot -x)g^j(\cdot ) \in L^1(L^1)$, for any $x\in \mathbb{R}^d$ with $|x|\leq 1$.
	We assume that both  
	\[
	\mathbf{A}=\left (A^{i,j}_{s,t}(x),\mathbb
	A^{i,j}_{s,t}(x)\right )_{\substack{1\leq i,j\leq n;\\ s\le t \in [0,T];x\in D }}
	\quad 
	\mathbf{B}=\left (B^{i,j}_{s,t}(x),\mathbb{B}^{i,j}_{s,t}(x)\right )_{\substack{1\leq i,j\leq n;\\ s\le t \in [0,T];x\in D }}
	\]
	are $n$-dimensional geometric rough drivers
	of finite $(p,p/2)$-variation with $p\in[2,3)$ and with coefficients in $H^{1}(D)$.
	Then the following holds:
	\begin{enumerate}[label=(\roman*)]
		\item \label{Q_shift}
		The two parameter mapping $\boldsymbol\Gamma^{\mathbf{A},\mathbf{B}}\equiv(\Gamma^{A,B},\bbGamma^{\mathbf{A},\mathbf{B}})$ defined for $s\le t \in [0,T] $ as
		\begin{equation*}
		\begin{aligned}
		&\Gamma_{s,t}^{A,B}:= A_{s,t}\otimes \mathbf 1 + \mathbf 1\otimes B_{s,t} \,,\quad 
		\\
		&\bbGamma_{s,t}^{\mathbf{A},\mathbf{B}}:=\mathbb{A}_{s,t}\otimes \mathbf 1+ A_{s,t}\otimes B_{s,t}+ \mathbf 1\otimes \mathbb{B}_{s,t}\,,
		\end{aligned}
		\end{equation*}
		where $\mathbf 1\equiv \mathbf 1_{n\times n}\in \mathcal L(\R^n)$ is the identity, is an $n^2$-dimensional rough driver,
		\item \label{prod_rv}
		The product $v^{\otimes 2}_t(x)=(a^i_t(x)b^j_t(x))_{1\leq i, j\leq n}$ is bounded as a path in $L^1(D;\R^{n\times n}).$ Moreover, it is a weak solution, in $L^1,$ of the system
		\begin{equation*}
		d (a\otimes b)= (a\otimes g +f \otimes b )d t + d \boldsymbol\Gamma^{\mathbf{A},\mathbf{B}} [a\otimes b]\,.
		\end{equation*}
	\end{enumerate}
\end{proposition}
 The following result is a small modification of the result in M.~Hofmanová, J.~M.~Leahy, T. Nilssen. \cite{tornstein_embedding} (Lemma A.2).
\begin{lemma}\label{lemma:embedding_tornstein}
	Let $B_0\subset B \subset B_1$ be Banach spaces, $B_0$ and $B_1$ reflexive with compact embedding of $B_0$ in $B$. Let $p>1$, $k\in [2,3)$ and we define
	\begin{equation*}
	X:=L^p([0,T]; B_0)\cap\{g\in C([0,T];B_1):\omega(s,t)\leq L \quad \mathrm{implies} \quad |\delta g_{s,t}|_{B_{1}}\leq \omega_{s,t}^{1/k})\},
	\end{equation*}
	endowed with the natural norm. Then the embedding of $X$ in $L^p([0,T];B)$ is compact.
\end{lemma}

\section{Appendix B: well posedness and continuity of the It\^o-Lyons map for the heat equation.}\label{Appendix_B_new}
This section is devoted to existence, uniqueness and local Lipschitz continuity of the It\^o-Lyons map for the heat equation \eqref{eq:heat_equation} on $\mathbb{T}^d$, with $d=1,2,3$ and with initial condition in $H^1$. The same proof can be achieved by different methods: we still provide a proof via this method since with a small adaptation of the following procedure, many other equations can be covered.
 The solution exists and is unique in the sense of Definition~\ref{def:sol_rough}.
We denote by $u=\pi_1(u^0,\mathbf{G})$ a solution to \eqref{eq:heat_equation} in the sense of Definition~\ref{def:sol_rough} started in $u^0$ and driven by $\mathbf{G}$.
If there exists a solution $u$ in the sense of Definition~\ref{def:sol_rough}, then the solution is unique and its It\^o-Lyons map is locally Lipschitz continuous.
\begin{proposition}\label{pro:lip_semilinear}
	Let $\mathbf{G}\in \mathcal{RD}^p(\mathbb{R}^n;E)$ and $u^0\in H^1$. If there exists a solution $u=\pi_1(u^0,\mathbf{G})$ to \eqref{eq:heat_equation} in the sense of Definition~\ref{def:sol_rough}, then it is unique. Moreover, the It\^o-Lyons map associated to the unique solution $u$ to \eqref{eq:heat_equation} 
	\begin{align*}
	\Phi_1:\quad\mathcal{RD}^p(\mathbb{R}^n;E)&\longrightarrow  L^\infty(H^1)\cap L^2(H^2) \cap \mathcal{V}^p(L^2)\\
	\mathbf{G}&\longmapsto u=\pi_1(u^0;\mathbf{G})\, 
	\end{align*}
	is locally Lipschitz continuous, in the sense that for all $\mathbf{G}, \mathbf{W}\in \mathcal{RD}^p(\mathbb{R}^n;E)$ there exists a constant $C\equiv C(\|u^0\|_{H},\mathbf{G},\mathbf{W},T)>0$ so that 
	\begin{align*}
	\|\Phi_1(\mathbf{G})-\Phi_1(\mathbf{W})\|_{L^\infty(H^1)\cap L^2(H^2) \cap \mathcal{V}^p(L^2)}\lesssim C\,\rho(\mathbf{G},\mathbf{W})\,.
	\end{align*}
\end{proposition}
\begin{proof}
	We follow the steps of the proof in \cite{LLG1D}. We recall the proof for the reader's convenience and because it introduces techniques useful for the CLT. The algebraic part concerning the noise is also similar. We can consider the square of the difference of two equations $u=\pi_1(u^0, \mathbf{G})$ and $v=\pi_1(u^0,\mathbf{W})$. 
	We show the local Lipschitz continuity of $u$ in $L^\infty(H^1)\cap L^2(H^2)\cap \mathcal{V}^p(L^2)$ with respect to the noise. If there exists a strong solution, the local Lipschitz continuity of the It\^o-Lyons map implies uniqueness, by setting $\mathbf{G}=\mathbf{W}$. Let $z:=u-v$, from the product formula in Proposition~\ref{pro:product}, for $k=0,1$ and for all $s\leq t\in [0,T]$, it follows that
	\begin{equation}\label{eq:tested_equation}
	\begin{aligned}
	\delta  (\partial^k_x z\otimes  \partial^k_x z)_{s,t}&=\int_{s}^{t}  \partial^k_x [b(u_r)-b(v_r)]\odot \partial^k_x z_r\dd r\\
	&\quad+ \partial^k_x [G_{s,t}u_s-W_{s,t}v_s] \odot \partial^k_x z_s+ \partial^k_x [\mathbb{G}_{s,t}u_s-\mathbb{W}_{s,t}v_s]\odot \partial^k_x z_s\\
	&\quad+\partial^k_x  [G_{s,t}u_s-W_{s,t}v_s]\otimes \partial^k_x [G_{s,t}u_s-W_{s,t}v_s] + \partial^k_x z^{\otimes 2,\natural}_{s,t} \,.
	\end{aligned}
	\end{equation}
	What follows holds for all $s\leq r\leq t \in [0,T]$.
	We aim to apply the sewing Lemma~\ref{lemma_sewing}, thus we want to represent $\delta( \partial^k_x z^{\otimes 2,\natural}_{s,t} )_{s,r,t}$ as linear combination of elements of $p/3$-variation or higher. 
	We first denote by $Z:=G-W$, $\mathbb{Z}=\mathbb{G}-\mathbb{W}$ (beware that $(Z,\mathbb{Z})$ is not a rough path) and rewrite the product in the last line as
	\begin{align*}
	\partial^k_x[G_{s,t}u_s-W_{s,t}v_s]\otimes \partial^k_x[G_{s,t}u_s-W_{s,t}v_s]&=\partial^k_x [Z_{s,t}u_s+W_{s,t}z_s]\otimes \partial^k_x [Z_{s,t}u_s+W_{s,t}z_s]\,.
	\end{align*}
	We introduce an algebraic relation for dealing with the elements of the form $A_{s,t}a_s\otimes B_{s,t}b_s$, for $\mathbf{A},\mathbf{B}\in \mathcal{RD}^p(\mathbb{R}^n;E)$ and $a,b\in\mathbb{R}^n$: we use the definition of $\delta(\cdot)_{s,r,t}$ and then we add and subtract some quantities:
	\begin{equation}\label{eq;algebraic_relation}
	\begin{aligned}
	\delta\left(A_{s,t}a_s\otimes B_{s,t}b_s\right)_{s,r,t}
	&=A_{s,t}a_s\otimes B_{s,t}b_s-A_{s,r}a_s\otimes B_{s,r}b_s-A_{r,t}a_r\otimes B_{r,t}b_r\\
	&=[A_{s,r}+A_{r,t}]a_s\otimes [B_{s,r}+B_{r,t}]b_s-A_{s,r}a_s\otimes B_{s,r}b_s-A_{r,t}a_r\otimes B_{r,t}b_r\\
	&=A_{s,r}a_s\otimes B_{r,t}b_s+A_{r,t}a_s\otimes B_{s,r}b_s+A_{r,t}a_s\otimes B_{r,t}b_s-A_{r,t}a_r\otimes B_{r,t}b_r\,.
	\end{aligned}
	\end{equation}
	We apply $\delta(\cdot)_{s,r,t}$ to the two-index map $\partial^k_x z^{\otimes 2,\natural}$. This leads to the expression 
	\begin{align*}
	\delta (\partial^k_x z^{\otimes 2,\natural})_{s,r,t}&= \sum_{n=0}^{k} \binom{k}{n} \partial^{k-n}_x Z_{r,t}\delta (\partial_x^n u\odot \partial^k_x z)_{s,r}+\sum_{n=0}^{k} \binom{k}{n} \partial^{k-n}_x W_{r,t}\delta (\partial_x^n z\odot \partial^k_x z)_{s,r}  \\
	& +\sum_{n=0}^{k} \binom{k}{n}  \partial^{k-n}_x \mathbb{Z}_{r,t}\delta (\partial_x^n  u\odot \partial^k_x z)_{s,r}+\sum_{n=0}^{k} \binom{k}{n}  \partial^{k-n}_x \mathbb{W}_{r,t}\delta (\partial_x^n z \odot \partial^k_x z)_{s,r}\\
	&- \partial^{k}_x [ \delta \mathbb{Z}_{s,r,t}  u_s]\odot  \partial^k_x z_s-\partial^{k}_x  [\delta \mathbb{W}_{s,r,t} \partial^{n}_x z_s]\odot\partial^k_x z_s\\
	&-\partial_{x}^{k}[Z_{s,r}u_s+W_{s,r}z_s]\otimes\partial_x^k[ Z_{r,t}u_s+W_{r,t}z_s]-\partial_{x}^{k}[Z_{r,t}u_s+W_{r,t}z_s]\otimes\partial_x^k[ Z_{s,r}u_s+W_{s,r}z_s]\\
	&-\partial_{x}^{k}[Z_{r,t}u_s+W_{r,t}z_s]\otimes\partial_x^k[ Z_{r,t}u_s+W_{r,t}z_s]+\partial_{x}^{k}[Z_{r,t}u_r+W_{r,t}z_r]\otimes\partial_x^k[ Z_{r,t}u_r+W_{r,t}z_r]\,,
	\end{align*}
	where we used the algebraic relation \eqref{eq;algebraic_relation}.
	Thus we compensate the following terms of at most $p/2$-variation
	\begin{equation}\label{eq:line_proof_p_1}
	\begin{aligned}
	&- \partial^{k}_x [ \delta \mathbb{Z}_{s,r,t}  u_s]\odot  \partial^k_x z_s-\partial^{k}_x  [\delta \mathbb{W}_{s,r,t} \partial^{n}_x z_s]\odot\partial^k_x z_s\\
	&-\partial_{x}^{k}[Z_{s,r}u_s+W_{s,r}z_s]\otimes\partial_x^k[ Z_{r,t}u_s+W_{r,t}z_s]-\partial_{x}^{k}[Z_{r,t}u_s+W_{r,t}z_s]\otimes\partial_x^k[ Z_{s,r}u_s+W_{s,r}z_s]\,,
	\end{aligned}
	\end{equation}
	by adding and subtracting properly chosen quantities from the following terms of only $p/2$-variation,
	\begin{align}\label{eq:things_to add_and subtract}
	\sum_{n=0}^{k} \binom{k}{n} \partial^{k-n}_x Z_{r,t}\delta (\partial_x^n u\odot \partial^k_x z)_{s,r}+\sum_{n=0}^{k} \binom{k}{n} \partial^{k-n}_x W_{r,t}\delta (\partial_x^n z\odot \partial^k_x z)_{s,r}\,.
	\end{align}
	Indeed if we can subtract from $\delta (\partial_x^n u\odot \partial^k_x z)_{s,r}$ all of the elements of $p$-variation of $\delta (\partial_x^n u\odot \partial^k_x z)_{s,r}$, which we can denote by $T_1$, then  $\delta (\partial_x^n u\odot \partial^k_x z)_{s,r}-T_1$ is of $p/2$-variation. The same consideration holds for $\delta (\partial_x^n z\odot \partial^k_x z)_{s,r}$. We compute those compensations.
	The elements of only $p$-variation coming from $\delta (\partial_x^n u\odot \partial^k_x z)_{s,r}$ are
	\begin{align*}
	T_1:= \partial^n_x [G_{s,r}u_s] \odot \partial^k_x z_s+\partial^n_x z_s\odot \partial^k_x [Z_{s,r}u_s+W_{s,r}z_s]\,,
	\end{align*}
	and the ones coming from  $\delta (\partial_x^n z\odot \partial^k_x z)_{s,r}$ are
	\begin{align*}
	T_2:= \partial^n_x [Z_{s,r}u_s+W_{s,r}z_s] \odot \partial^k_x z_s+\partial^n_x z_s\odot \partial^k_x [Z_{s,r}u_s+W_{s,r}z_s]\,.
	\end{align*}
	Going back to \eqref{eq:things_to add_and subtract}, the following equality holds
	\begin{align}\label{eq:line_proof}
	&\quad\quad \sum_{n=0}^{k} \binom{k}{n} \partial^{k-n}_x Z_{r,t}\delta (\partial_x^n u\odot \partial^k_x z)_{s,r}+\sum_{n=0}^{k} \binom{k}{n} \partial^{k-n}_x W_{r,t}\delta (\partial_x^n z\odot \partial^k_x z)_{s,r}\nonumber\\
	&=\sum_{n=0}^{k} \binom{k}{n} \partial^{k-n}_x Z_{r,t}[\delta (\partial_x^n u\odot \partial^k_x z)_{s,r}-T_1]+\sum_{n=0}^{k} \binom{k}{n} \partial^{k-n}_x W_{r,t}[\delta (\partial_x^n z\odot \partial^k_x z)_{s,r}-T_2]\\
	&\quad +\sum_{n=0}^{k} \binom{k}{n} \partial^{k-n}_x Z_{r,t}T_1+\sum_{n=0}^{k} \binom{k}{n} \partial^{k-n}_x W_{r,t}T_2\,. \nonumber
	\end{align}
	The elements in \eqref{eq:line_proof} are now of $p/3$-variation and thus suitable for the remainder estimate. We write explicitly the last line 
	\begin{align*}
	&\quad\quad\sum_{n=0}^{k} \binom{k}{n} \partial^{k-n}_xZ_{r,t}T_1+ \sum_{n=0}^{k} \binom{k}{n} \partial^{k-n}_xW_{r,t}T_2\\
	&= \sum_{n=0}^{k} \binom{k}{n} [\partial^{k-n}_x Z_{r,t}\partial^n_x [G_{s,r}u_s] \odot \partial^k_x z_s+\partial^{k-n}_x Z_{r,t}\partial^n_x z_s\odot \partial^k_x [Z_{s,r}u_s+W_{s,r}z_s]]\\
	&\quad + \sum_{n=0}^{k} \binom{k}{n}  [\partial^{k-n}_xW_{r,t}\partial^n_x [Z_{s,r}u_s+W_{s,r}z_s] \odot \partial^k_x z_s+\partial^{k-n}_xW_{r,t}\partial^n_x z_s\odot \partial^k_x [Z_{s,r}u_s+W_{s,r}z_s]]\\
	&=\partial_x^k[Z_{r,t}G_{s,r}u_s] \odot \partial^k_x z_s+\partial_x^k[Z_{r,t} z_s]\odot \partial^k_x [Z_{s,r}u_s+W_{s,r}z_s]\\
	&\quad  +\partial^{k}_x[W_{r,t}[Z_{s,r}u_s+W_{s,r}z_s]] \odot \partial^k_x z_s+\partial^{k}_x[W_{r,t} z_s]\odot \partial^k_x [Z_{s,r}u_s+W_{s,r}z_s]\,,
	\end{align*}
	which compensate with the terms of $p/2$-variation in \eqref{eq:line_proof_p_1}. Indeed, by rearranging the terms and using Chen's relation,
	\begin{align*}
	&\quad \partial_x^k[Z_{r,t}G_{s,r}u_s] \odot \partial^k_x z_s + \partial^{k}_x[W_{r,t}[Z_{s,r}u_s+W_{s,r}z_s]] \odot \partial^k_x z_s\\
	&= \partial_x^k[(G_{r,t}-W_{r,t})G_{s,r}u_s] \odot \partial^k_x z_s + \partial^{k}_x[W_{r,t}[(G_{s,r}-W_{s,r})u_s+W_{s,r}z_s]] \odot \partial^k_x z_s\\
	&= \partial_x^k[(\delta \mathbb{G}_{s,r,t}-\delta \mathbb{W}_{s,r,t})u_s] \odot \partial^k_x z_s + \partial^{k}_x[W_{r,t}W_{s,r}z_s] \odot \partial^k_x z_s\\
	&= \partial^{k}_x [ \delta \mathbb{Z}_{s,r,t}  u_s]\odot  \partial^k_x z_s+\partial^{k}_x  [\delta \mathbb{W}_{s,r,t} \partial^{n}_x z_s]\odot\partial^k_x z_s\,,
	\end{align*}
	the elements in the first line of \eqref{eq:line_proof_p_1} erase.
	In conclusion, we can rewrite $\delta \partial^k_x z^{\otimes 2,\natural}$ as
	\begin{equation}\label{eq:remainder_WZ}
	\begin{aligned}
	\delta (\partial^k_x z^{\otimes 2,\natural})_{s,r,t}&= \sum_{n=0}^{k} \binom{k}{n} \partial^{k-n}_x Z_{r,t}[\delta (\partial_x^n u\odot \partial^k_x z)_{s,r}-T_1]+\sum_{n=0}^{k} \binom{k}{n} \partial^{k-n}_x W_{r,t}[\delta (\partial_x^n z\odot \partial^k_x z)_{s,r}-T_2]\\
	&  +\sum_{n=0}^{k} \binom{k}{n}  \partial^{k-n}_x \mathbb{Z}_{r,t}\delta (\partial_x^n  u\odot \partial^k_x z)_{s,r}+\sum_{n=0}^{k} \binom{k}{n}  \partial^{k-n}_x \mathbb{W}_{r,t}\delta (\partial_x^n z \odot \partial^k_x z)_{s,r}\\
	&+\partial_{x}^{k}[Z_{r,t}u_s+W_{r,t}z_s]\otimes\partial_x^k[ Z_{r,t}u_s+W_{r,t}z_s]-\partial_{x}^{k}[Z_{r,t}u_r+W_{r,t}z_r]\otimes\partial_x^k[ Z_{r,t}u_r+W_{r,t}z_r]\,,
	\end{aligned}
	\end{equation}
	where all the terms appearing are at least of $p/3$-variation. 
	Consider $k=0$. Then $\delta ( z^{\otimes 2,\natural})_{s,r,t}$ depends on $\delta (u\odot u)$, $\delta (u\odot z)_{s,r}$ and $\delta (z\odot z)_{s,r}$. In the following, denote by
	\begin{align*}
	\omega_{b(z,u)}(s,r):=\int_{s}^{r}\sup_{\phi\in H^1,\|\phi\|\leq 1}[|\langle b(z)\otimes u,\phi\rangle|+|\langle z\otimes b(u),\phi\rangle|]\,,
	\end{align*}
	which is a control associated to the drift of $\delta (u\odot z)_{s,r}$. The other controls associated to the drifts of other equations are defined analogously.	
	
	Thus, by substituting the equations into equality \eqref{eq:remainder_WZ},  the remainders $u^{\otimes 2,\natural}$, $(uz)^{\otimes 2,\natural}$ and $z^{\otimes 2,\natural}$ appear. By similar computations as above, it follows that
	\begin{align*}
	\omega_{(uz)^{\otimes 2, \natural};W^{1,\infty}}^{3/p}\lesssim \omega^{1/p}\omega_{b(u,z)}+\omega^{3/p}\sup_{s\leq r\leq t}[\|u_r\|_{L^2}^2+\|z_r\|_{L^2}^2]\,,
	\end{align*}
	and the same bound holds for $\omega_{(zu)^{\otimes 2, \natural};W^{1,\infty}}^{3/p}$. Analogously, it holds that
	\begin{align*}
	\omega_{u^{2,\natural,i,j}}^{3/p}\lesssim \omega^{1/p}\omega_{b(u^i,u^j)}+\omega^{3/p}\sup_{s\leq r\leq t}\|u_r\|_{L^2}^2\,.
	\end{align*}
	Observe also that the inequality holds
	\begin{align*}
	\omega^{1/p}_{z^iz^j}\lesssim \omega_{b(z,z)}+\omega_{\mathbf{Z}}^{1/p}\sup_{s\leq r\leq t}\|u_r\|_{L^2}^2+\omega^{1/p}\sup_{s\leq r\leq t}\|z_r\|_{L^2}+\omega^{3/p}_{z^{2,\natural,i,j}}\,.
	\end{align*}
	Hence estimate the noise, by recalling that the noises have $L^\infty$ space regularity,
	\begin{align*}
	\| \delta z^{\otimes 2,\natural}_{s,r,t}\|_{(W^{1,\infty})^*}&\lesssim \omega^{1/p}_{\mathbf{Z}}\omega^{2/p}\left[\sup_{s\leq r\leq t}[\|u_r\|_{L^2}^2+\|z_r\|_{L^2}^2]+\omega_{b(u,u)}+\omega_{b(z,u)}+\omega_{b(u,z)}\right]\\
	&\quad+\omega^{2/p}\omega_{b(z,z)}+\omega^{1/p}\omega^{3/p}_{z^{\otimes 2,\natural}}\,.
	\end{align*}
	By taking the power $p/3$, taking the sum on the partitions and by using the sub-additivity of the controls (also that the product of controls is a control), we can write
	\begin{align*}
	\omega^{3/p}_{z^{2,\natural,i,j}}&\lesssim\omega^{1/p}_{\mathbf{Z}}\omega^{2/p}\left[\sup_{s\leq r\leq t}[\|u_r\|_{L^2}^2+\|z_r\|_{L^2}^2]+\omega_{b(u,u)}+\omega_{b(z,u)}+\omega_{b(u,z)}\right]+\omega^{2/p}\omega_{b(z,z)}+\omega^{1/p}\omega^{3/p}_{z^{2,\natural,i,j}}\,.
	\end{align*}
	By choosing a suitable partition such that $\omega<1$, we can absorb $\omega_{z^{\otimes 2,\natural}}^{1/p}$ to the left hand side and obtain the following estimate,
	\begin{align*}
	\omega^{3/p}_{z^{\otimes 2,\natural,i,j}}&\lesssim\omega^{1/p}_{\mathbf{Z}}\omega^{2/p}\left[\sup_{s\leq r\leq t}[\|u_r\|_{L^2}^2+\|z_r\|_{L^2}^2]+\omega_{b(u,u)}+\omega_{b(z,u)}+\omega_{b(u,z)}\right]+\omega^{2/p}\omega_{b(z,z)}\,.
	\end{align*}
	In conclusion, the estimate of the remainder becomes
	\begin{align}\label{eq:rem_steimate_spde}
	\|z^{\otimes 2,\natural}\|_{(W^{1,\infty})^*}\lesssim \omega^{1/p}_{\mathbf{Z}}\omega^{2/p}+\omega^{3/p}\sup_{s\leq r\leq t}\|  z_r\|_{L^2}^2+\omega^{1/p}\int_{s}^{t}\|z_r\|^2_{H^1} \dd r \,,
	\end{align}
	where the time integral on the right hand side comes from estimating the drift $\omega_{b(z,z)}$.
	We consider again the equation tested by $\mathbf{1}\in W^{1,\infty}$ in \eqref{eq:tested_equation}. We write down the drift
	\begin{align*}
	\int_{s}^{t}\langle  [b(u_r)-b(v_r)]\otimes  z_r+ z_r\otimes [b(u_r)-b(v_r)], \mathbf{1}\rangle \dd r&=-2\int_{s}^{t} \| \nabla z_r\|^2_{L^2}\dd r \,.
	\end{align*}
	From this form of the drift and the remainder estimate in \eqref{eq:rem_steimate_spde},
	\begin{align*}
	\|z_t\|^2_{L^2}+2\int_{s}^{t}\|\nabla z_r\|^2_{L^2}\dd r&\lesssim \|z_s\|^2_{L^2}+\omega_{\mathbf{Z}}^{1/p}\sup_{0\leq r\leq T}(\|u_r\|^2_{L^2}+\|v_r\|^2_{L^2})\\
	&\quad+\omega^{1/p}_{\mathbf{Z}}\omega^{2/p}+\omega^{3/p}\sup_{s\leq r\leq t}\|  z_r\|_{L^2}^2+\omega^{1/p}\int_{s}^{t}\|z_r\|^2_{H^1} \dd r \,.
	\end{align*}
	We apply the rough Gronwall's Lemma~\ref{lem:gronwall} to $E_t:=\|z_t\|^2_{L^2}+\int_{0}^{t}\|\nabla z_r\|^2_{L^2}\dd r$, $\phi(0,t):=\omega_{\mathbf{Z}}^{1/p}(0,t) $ and $\bar{\omega}(0,t):=\omega(0,t)$. Then there exists a constant $\tau>0$ such that
	\begin{align*}
	\sup_{0\leq t\leq T}\|z_t\|_{L^2}^2+\int_{0}^{T} \|\nabla z_r\|^{2}_{L^2}\dd r \lesssim_{\tau}  \exp\left(\frac{\bar{\omega}}{\tau}\right)[|z_0|^2+\omega_{\mathbf{Z}}^{1/p}(0,T)]\,,
	\end{align*}
	that implies the continuity of the It\^o-Lyons map in $L^\infty(L^2)\cap L^2(H^1)$.
	The map $\mathbf{W}\mapsto u$ is locally Lipschitz continuous in $\mathcal{V}^p(H^{-1})$, indeed from the previous computations
	\begin{align*}
	\|\delta  z_{s,t}\|_{H^{-1}} \lesssim \|b(u)-b(v)\|_{L^\infty(H^{-1})}\| z\|_{L^\infty(H^1)}(t-s)+\omega_{\mathbf{Z}}^{1/p}(s,t) +\omega^{1/p}\|z\|_{L^\infty(H^1)}^2(t-s)\,,
	\end{align*}
	and from the previous bound on $\|z\|_{L^\infty(H^1)}^2$, the convergence follows.
	We show the local Lipschitz continuity of $\nabla u$ in $L^\infty(L^2)\cap L^2(H^1)$ with respect to the noise. In this step, we use the fact that the It\^o-Lyons map is locally Lipschitz continuous in $L^\infty(L^2)\cap L^2(H^1)$.
	Denote by $\nabla z:=\nabla u-\nabla v$ and we look at the squared equation via the product formula in Proposition~\ref{pro:product}
	\begin{align*}
	\delta (\nabla z\otimes  \nabla z)_{s,t}&=\int_{s}^{t} \nabla[b(u_r)-b(v_r)]\odot\nabla z_r \dd r\\
	&\quad+  \nabla [G_{s,t}u_s-W_{s,t}v_s] \odot\nabla z_s+ \nabla[\mathbb{G}_{s,t}u_s-\mathbb{W}_{s,t}v_s]\odot \nabla z_s\\
	&\quad+\nabla [G_{s,t}u_s-W_{s,t}v_s]\otimes\nabla[G_{s,t}u_s-W_{s,t}v_s]+\nabla z^{\otimes 2,\natural}_{s,t}\,.
	\end{align*}
	We aim to apply the sewing Lemma~\ref{lemma_sewing} to the three index map $\delta \nabla z^{\otimes 2,\natural}_{s,r,t}$ obtained by applying $\delta(\cdot)_{s,r,t}$ to the two index map $\nabla z^{\otimes 2,\natural}$. With the expression derived before for $\partial^k_x z^{\otimes 2,\natural}$, for each direction $x$ fixed, we can derive again a bound for the remainder:
	\begin{align*}
	\|\nabla z^{\otimes 2,\natural}\|_{(W^{1,\infty})*}\lesssim \omega^{1/p}_{\mathbf{Z}}\omega^{2/p}+\omega^{3/p}\sup_{s\leq r\leq t}\|  z_r\|_{H^1}^2+\omega^{1/p}\int_{s}^{t}\|z_r\|^2_{H^2} \dd r \,,
	\end{align*}
	where $\omega$ is a control depending on the noise, the drift and the initial condition. We test the above equation by $\mathbf{1}\in W^{1,\infty}$,
	\begin{align*}
	\delta \langle \nabla z\otimes  \nabla z, \mathbf{1}\rangle_{s,t}&=\int_{s}^{t}\langle  \nabla[b(u_r)-b(v_r)]\otimes \nabla z_r+\nabla z_r\otimes\nabla [b(u_r)-b(v_r)], \mathbf{1}\rangle \dd r\\
	&\quad+ 2\langle \nabla [G_{s,t}u_s-W_{s,t}v_s] \odot\nabla z_s, \mathbf{1}\rangle+\langle \nabla[\mathbb{G}_{s,t}u_s-\mathbb{W}_{s,t}v_s]\odot \nabla z_s, \mathbf{1}\rangle\\
	&\quad+\langle\nabla [G_{s,t}u_s-W_{s,t}v_s]\otimes\nabla[G_{s,t}u_s-W_{s,t}v_s], \mathbf{1}\rangle+\langle \nabla z^{\otimes 2,\natural}_{s,t}, \mathbf{1}\rangle\,,
	\end{align*}
	and from the rough Gronwall's Lemma~\ref{lem:gronwall}, we conclude that
	\begin{align*}
	\sup_{0\leq t\leq T}\|\nabla z_t\|^2_{L^2}+\int_{0}^{T}\|\Delta z_r\|^2_{L^2}\dd r\lesssim \exp \left(\omega\right)[\|z^0\|^2_{L^2}+\omega^{1/p}_{\mathbf{Z}}(0,T)]\,,
	\end{align*}
	which implies the local Lipschitz continuity of the It\^o-Lyons map in $L^\infty(H^1)\cap L^2(H^2)$. Analogously, we can infer the convergence also in $\mathcal{V}^p(L^2)$.
\end{proof}
\begin{remark}
	If the driving signal is Gaussian, this same continuity of the It\^o-Lyons map allows to derive a support theorem and a large deviation principle for the solution to the SPDEs.
\end{remark}
\begin{remark}\label{eq:rough_stnadard_machinery}
	We refer to the procedure in the proof of Proposition~\ref{pro:lip_semilinear} of obtaining an estimate of the remainder term independent on the remainder itself as in \eqref{eq:rem_steimate_spde} as \textit{rough standard machinery}. Indeed this procedure is repetitive and will be applied many times in the subsequent, we refer to it in this shortened way (see \cite[Sec. 3.2.2]{LLG1D} for more details).
\end{remark}

To show existence of a solution to \eqref{eq:heat_equation}, we approximate the noise by means of smooth approximations (since we are working with geometric rough paths, there exists an approximation of the rough paths with respect to the $p$-variation norm). Let $(G^n,\mathbb{G}^n)_n$ be such an approximation of $(G,\mathbb{G})$. From classical theory, there exists a unique smooth solution $u^n$ to the approximated equation
\begin{equation*}
\delta u^n_{s,t}=\int_{s}^{t}b(u^n_r)\dd r+G^n_{s,t}u^n_s+\mathbb{G}^n_{s,t}u^n_s+u^{n,\natural}_{s,t}\,,
\end{equation*}
with regularity $L^\infty(H^1)\cap L^2(H^2)\cap C^1(L^2)$. 
To show existence of a solution $u$ to \eqref{eq:heat_equation}, we want to uniformly bound $(u^n)_n$ in $L^\infty(H^1)\cap L^2(H^2)\cap \mathcal{V}^p(L^2)$. 
We compute the estimate of the remainder and highlight its dependence on the drift.
\begin{proposition}\label{pro:remainder_drift_abstract}
	The following estimate of the remainder holds:
	\begin{align*}
	\| u^{\otimes 2,\natural}_{s,t}\|_{(W^{1,\infty})^*}\lesssim_{\mathbf{G},u^0}\omega^{3/p}\sup_{s\leq r\leq t}\| u_r\|_{L^2}^2+\omega^{1/p}\omega_{b(u,u); (W^{1,\infty})^*}\,,
	\end{align*}
	\begin{align*}
	\|\partial_x u^{2,\natural}_{s,t}\|_{(W^{1,\infty})^*}&\lesssim_{\mathbf{G},\partial_x \mathbf{G},u^0,\nabla u^0}\omega^{3/p}[\sup_{s\leq r\leq t}\| u_r\|_{L^2}^2+\sup_{s\leq r\leq t}\|\partial_x u_r\|_{L^2}^2]\\
	&\quad+\omega^{1/p}[\omega_{b(u,u); (W^{1,\infty})^*}+\omega_{b(\partial_x u,u); (W^{1,\infty})^*}+\omega_{b(u,\partial_x u); (W^{1,\infty})^*}+\omega_{b(\partial_xu,\partial_xu); (W^{1,\infty})^*}]\,,
	\end{align*}
	where for all $s\leq t\in [0,T]$, $k,j\in \{0,1\}$ 
	\begin{align*}
	\omega_{b(\partial_x^ ku,\partial_x^ju); (W^{1,\infty})^*}(s,t)=\int_{s}^{t}\sup_{\phi\in W^{1,\infty}, \|\phi\|_{W^{1,\infty}}\leq 1}|\langle b(\partial_x^ k u,\partial_x^ j u),\phi \rangle|\,.
	\end{align*}
	
\end{proposition}
\begin{proof}
	The estimate in this proposition follows by Proposition~\ref{pro:lip_semilinear}: one can compare a solution $u=\pi_1(u^0,\mathbf{G})$ with the null equation $0=v=\pi_1(0,\mathbf{0})$ and substitute $z=u-0$, $\mathbf{Z}=\mathbf{G}-\mathbf{0}$, $\mathbf{W}=\mathbf{0}$ into \eqref{eq:remainder_WZ}. This leads to the following expression of the remainder, for $k=0,1$,
	\begin{equation*}
	\begin{aligned}
	\delta (\partial^k_x u^{2,\natural})_{s,r,t}&= \sum_{n=0}^{k} \binom{k}{n} \partial^{k-n}_x G_{r,t}[\delta (\partial_x^n u\odot \partial^k_x u)_{s,r}-T_1]+\sum_{n=0}^{k} \binom{k}{n}  \partial^{k-n}_x \mathbb{G}_{r,t}\delta (\partial_x^n  u\odot \partial^k_x u)_{s,r}\\
	&+\partial_{x}^{k}[G_{r,t}u_s]\otimes\partial_x^k[ G_{r,t}u_s]-\partial_{x}^{k}[G_{r,t}u_r]\otimes\partial_x^k[ G_{r,t}u_r]\,.
	\end{aligned}
	\end{equation*}
	The drifts appearing  in the remainder estimate for $k=0$ are coming from $\delta (u\odot u)$.
	For  $k=1$, the drifts appearing are coming from the equations for
	\begin{align*}
	\delta (u\odot u)\,,\quad \delta (\partial_x u\odot u)\,,\quad \delta (u\odot \partial_x u)\,,\quad \delta (\partial_xu\odot \partial_xu)\,.
	\end{align*}
\end{proof}
We explicit the bounds for the drift in the case of the heat equation.
\begin{proposition}\label{pro:drift_heat_remainder}
	For all $s\leq t\in [0,T]$ the following bounds on the drift hold:
	\begin{align*}
	\omega_{b(u,u);(W^{1,\infty})^*}(s,t)\lesssim \int_{s}^{t}\|\nabla u_r\|^2_{L^2}\dd r+(t-s)\sup_{s\leq r\leq t}\|u_r\|^2_{L^2}\,,
	\end{align*}
	\begin{align*}
	\omega_{b(\partial_x u,\partial_x u);(W^{1,\infty})^*}(s,t)\lesssim \int_{s}^{t}\|\Delta u_r\|^2_{L^2}\dd r+(t-s)\sup_{s\leq r\leq t}\|\nabla u_r\|^2_{L^2}\,.
	\end{align*}
	The same bound holds for $\omega_{b(\partial_x u,u);(W^{1,\infty})^*},\omega_{b(u,\partial_x u);(W^{1,\infty})^*}$.
\end{proposition}
\begin{proof}
	For every $\phi\in W^{1,\infty}$ the following equality holds for the Laplacian
	\begin{align*}
	\langle \Delta u_r\odot u_r,\phi  \rangle = -\langle \nabla u_r\odot \nabla u_r,\phi  \rangle-\langle \nabla  u_r\odot u_r,\nabla \phi  \rangle\,,
	\end{align*}
	Hence $|\langle \Delta u_r\odot u_r,\phi  \rangle|\lesssim [\|\nabla u_r\|^2_{L^2}+\|u_r\|^2_{L^2}]\|\phi\|_{W^{1,\infty}}$. The lower order terms in the drift are bounded analogously and lead to the same bound.
	We give an estimate for the drift of $\delta (\partial_xu\odot \partial_xu)$. For all $\phi\in W^{1,\infty}$, 
	\begin{align*}
	\langle\nabla[ \Delta u_r]\odot \nabla u_r,\phi  \rangle = -\langle \Delta u_r\odot \nabla^2 u_r,\phi  \rangle-\langle \Delta u_r\odot \nabla  u_r,\nabla \phi  \rangle\,,
	\end{align*}
	and thus we can bound
	\begin{align*}
	|\langle\nabla[ \Delta u_r]\odot \nabla u_r,\phi  \rangle |\lesssim \|\Delta u_r\|_{L^2} ^2\|\phi\|_{L^\infty}+\|\Delta u_r\|_{L^2} \|\nabla  u_r\|_{L^2}\|\nabla \phi\|_{L^\infty}\,.
	\end{align*}
	The drifts coming from $\delta (\partial_x u\odot u)\,,\delta (u\odot \partial_x u)$ can be estimated by similar computations.
	
\end{proof}
We state the a priori estimates.
\begin{theorem}
	The sequence $(u^n)_n$ is uniformly bounded in $L^\infty(H^1)\cap L^2(H^2)\cap\mathcal{V}^p(L^2)$.
\end{theorem}
\begin{proof}
	From estimates analogous as in Proposition~\ref{pro:lip_semilinear} (which can be applied at the level of the approximations) and from Proposition~\ref{pro:drift_heat_remainder} and the rough Gronwall's Lemma~\ref{lem:gronwall},
	\begin{align*}
	\sup_{0\leq t\leq T}\|u^n\|^2_{L^2}+\int_{0}^{t}\|\nabla u^n_r\|^2_{L^2}\dd r\lesssim\|u^{0}\|^2_{L^2}(1+\omega_{G}^{1/p})\,,
	\end{align*}
	where we used that $\|u^{n,0}\|_{L^2}\leq\|u^0\|_{L^2}$. Analogously, for the higher order estimate we obtain
	\begin{align*}
	\sup_{0\leq t\leq T}\|\nabla u^n\|^2_{L^2}+\int_{0}^{t}\|\Delta u^n_r\|^2_{L^2}\dd r\lesssim \|\nabla u^{0}\|^2_{L^2}(1+\omega_{G}^{1/p})\,.
	\end{align*}
	For the bound in $\mathcal{V}^p(L^2)$, the bound holds
	\begin{align*}
	\|\delta u^n_{s,t}\|_{L^2}\lesssim (t-s)\|\nabla u^n\|_{L^\infty(L^2)}+\omega_{\mathbf{G}}^{1/p}(s,t)+\omega^{1/p}(s,t)\,,
	\end{align*}
	for a generic control $\omega^{1/p}$.

\end{proof}
From the uniform bound of the sequence in $L^\infty(H^1)\cap L^2(H^2)\cap \mathcal{V}^p(L^2)$, we can apply the compactness argument and infer that there exists a solution to \eqref{eq:heat_equation} (use e.g. Lemma~\ref{lemma:embedding_tornstein}). 

\begin{proposition}
	Under the assumptions and in the notation of this section, there exists a unique strong solution $u=\pi_1(u^0,\mathbf{G})$ to \eqref{eq:heat_equation} in the sense of Definition~\ref{def:sol_rough}. Moreover, the map $\mathbf{G}\mapsto u$ is locally Lipschitz continuous from $\mathcal{RD}^p(W^{2,\infty})$ to $L^\infty(H^1)\cap L^2(H^2)\cap\mathcal{V}^p(L^2)$.
\end{proposition}
\begin{remark}
	As a consequence of the continuity of the It\^o-Lyons map, a support theorem and a large deviation principle follow.
\end{remark}

\section{Appendix C: well posedness of the limit equation and continuity of the It\^o-Lyons map.}\label{Appendix_B}
In this section we discuss the well posedness of the linear equation 
	\begin{equation}\label{equazione_r_appendix}
\delta Y_{s,t}=\int_{s}^{t}b'(u_r)Y_r\dd r+g_{s,t}Y_s+\mathbb{g}_{s,t}Y_s+([Wg]+[gW])_{s,t}u_s+W_{s,t}u_s+Y^\natural_{s,t}\,,
\end{equation}
where $u$ is the unique solution to 
\begin{align}\label{equazione_g_appendix}
\delta u_{s,t}=\int_{s}^{t}b(u_r)\dd r+g_{s,t}u_s+\mathbb{g}_{s,t}u_s+u^\natural_{s,t}\,.
\end{align}
Assume to be in the physical dimensions $d=1,2,3$ and denote by $H:=H^1(\mathbb{T}^d;\mathbb{R}^n)$, $V=H^2(\mathbb{T}^d;\mathbb{R}^n)$ and $E:=W^{2,\infty}(\mathbb{T}^d;\mathbb{R}^n)$. We state the notion of solution.
	\begin{definition}\label{def:sol_additive_appendix}
	Let $(\mathbf{W},\mathbf{g})$ be random $p$-compatible directions with joint lift $\mathbf{Z}$ (where the directions are rough drivers). Let $u^0\in H^1$ and $u\in L^\infty(H^1)\cap L^2(H^2) \cap\mathcal{V}^p(L^2)$ be the unique solution in the sense of Definition~\ref{def:sol_rough} to \eqref{equazione_g_appendix} and such that $u_0=u^0$.
	We say that $Y\in L^\infty(H^1)\cap L^2(H^2)\cap \mathcal{V}^p(L^2)$ is a solution with initial condition $0\in L^2$ if there exists a two index map $Y^\natural\in \mathcal{V}^{p/3}_2(L^2)$ defined implicitly by the equality \eqref{equazione_r_appendix} in $L^2$, with $Y_0=0$ as an equality in $L^2$.
	\end{definition}
    In what follows, we denote by $Y^{\mathbf{W}}:=\pi(y^0,\mathbf{g},y,(\mathbf{W},\mathbf{g}))$ the unique solution to \eqref{equazione_r_appendix} coupled with $u=\pi(u^0,\mathbf{g})$  with respect to the compatible direction $(\mathbf{W},\mathbf{g})$, in the sense of Definition~\ref{def:sol_additive_appendix}. The solution $u=\pi(u^0,\mathbf{g})$ is intended in the sense of Definition~\ref{def:sol_rough}.
	We work under the following assumptions on the Fréchet derivative of $b$.
	\begin{assumption}\label{assumption:drift_linear_b_prime}
	Let $b:H^2\rightarrow L^2$ be Fréchet differentiable, with Fréchet derivative $b':H^2\rightarrow \mathcal{L}(H^2;L^2)$.
	Assume that for all $u\in L^\infty(H^1)\cap L^2 (H^2)$, for all $f,h\in L^\infty(H^1)\cap L^2 (H^2)$ and for all $\phi\in W^{1,\infty}$ there exists  $m,n\in\mathbb{N}$ such that
	\begin{align*}
	\int_{0}^{t}|\langle [b'(u_r)h_r]\otimes  f_r,\phi\rangle|\dd r&\leq [\|u\|_{L^\infty(H^1)\cap L^2(H^2)}]^m\|h\|_{L^\infty(H^1)}\|f\|_{L^\infty(H^1)}\|\phi\|_{L^\infty}(t-s)\\
	&\quad + [\|u\|_{L^\infty(H^1)\cap L^2(H^2)}]^m\|\nabla h\|_{L^2(L^2)}\|f\|_{L^2(H^1)}\|\phi\|_{W^{1,\infty}}\,,
	\end{align*}
	\begin{align*}
	\int_{0}^{t}|\langle \nabla [b'(u_r)h_r]\otimes \nabla f_r,\phi\rangle|\dd r&\leq [\|u\|_{L^\infty(H^1)\cap L^2(H^2)}]^n\|h\|_{L^\infty(H^1)}\|f\|_{L^\infty(H^1)}\|\phi\|_{L^\infty}(t-s)\\
	&\quad + [\|u\|_{L^\infty(H^1)\cap L^2(H^2)}]^n\|\Delta h\|_{L^2(L^2)}\|f\|_{L^2(H^1)}\|\phi\|_{W^{1,\infty}}\,.
	\end{align*}
	Moreover, it also holds that
	\begin{align*}
	\int_{0}^{t}\langle b'(u_r) f_r,f_r \rangle\dd r \leq -\|\nabla f\|^2_{L^2(L^2)}+\|u\|_{L^\infty(H^1)\cap L^2(H^2)}\|f\|_{L^\infty(L^2)}^2 \,,
	\end{align*}
	\begin{align*}
	\int_{0}^{t} \langle \nabla[b'(u_r) f_r],\nabla f _r\rangle \dd r\leq -\|\Delta  f\|^2_{L^2(L^2)}+\|u\|_{L^\infty(H^1)\cap L^2(H^2)}\|\nabla f\|_{L^\infty(L^2)}^2 \,.
	\end{align*}
\end{assumption}

\begin{proposition}\label{pro:continuity_limit_equation}
	Let $(\mathbf{g},\mathbf{W})$ be random $p$-compatible directions as in Construction \ref{constr:Constr_p_q} with joint lift $\mathbf{Z}$ (where the directions are rough drivers).
	In the setting of this section, equation \eqref{eq:Dphi_Y_not_centered} admits a unique solution $Y^\mathbf{g}\in L^\infty(H^1)\cap L^2(H^2)\cap C([0,T];L^2)$ in the sense of Definition~\ref{def:sol_additive_appendix}. Moreover, the map 
	\begin{align*}
	D\Phi[\mathbf{g}](\cdot):\mathcal{RP}^p(\mathbb{R}^d;E)&\longrightarrow \mathcal{Y}:=C([0,T];L^2)\cap L^2(H^2)\cap L^\infty(H^1)\\
	\mathbf{W}&\longmapsto D\Phi[\mathbf{g}](\mathbf{W})=Y^\mathbf{W}\,,
	\end{align*}
	is locally Lipschitz continuous.
\end{proposition}
\begin{proof}
	Assume that there exists $Y^{\mathbf{W}}=\pi(u^0,\mathbf{g},u,(\mathbf{g},\mathbf{W})), Y^{\mathbf{G}}=\pi(u^0,\mathbf{g},u,(\mathbf{g},\mathbf{G}))$. We study the local Lipschitz continuity of the map $\mathbf{W}\mapsto Y^{\mathbf{W}}$ and, at the same time, the uniqueness of the equation. We introduce the notation $z:=Y^{\mathbf{W}}- Y^{\mathbf{G}}$, $[Zg]+[gZ]:=([gW]+[Wg])-([gG]+[Gg])$ and consider the squared equation $\partial_x^k z\otimes \partial_x^k  z$
	\begin{align*}
	\delta (\partial_x^k z\otimes  \partial_x^k z )_{s,t}&=\int_{s}^{t}\partial_x^k[b'(u_r)z_r]\odot \partial_x^kz_r \dd r+\partial_x^k[Z_{s,t}u_s]\odot\partial_x^kz_s+\partial_x^k[g_{s,t}z_s]\odot \partial_x^k z_s\\
	&\quad+\partial_x^k[[[Zg]+[gZ]]_{s,t}u_s]\odot \partial_x^k z_s+\partial_x^k[[gg]_{s,t}z_s]\odot \partial_x^kz_s\\
	&\quad+\partial_x^k[Z_{s,t}u_s+g_{s,t}z_s]\otimes \partial_x^k[Z_{s,t}u_s+g_{s,t}z_s]+\partial_x^kz^{2,\natural}_{s,t}\\
	&=\int_{s}^{t}\partial_x^k[b'(u_r)z_r]\odot \partial_x^kz_r \dd r+I_{s,t}+\mathbb{I}_{s,t}+\tilde{\mathbb{I}}_{s,t}+\partial_x^kz^{2,\natural}_{s,t}\,,
	\end{align*}
	where $I$ correspond to the terms of $p$-variation in the first line, $\mathbb{I}$ are the terms of $p/2$-variation in the second line and $\tilde{\mathbb{I}}$ are the terms of $p/2$-variation in the third line.
	We apply $\delta(\cdot)_{s,r,t}$ to the remainder term $\partial_x^k z^{2,\natural}$ and obtain
	\begin{align*}
	\delta I_{s,r,t}&=-\sum_{n=0}^{k} \binom{k}{n} \partial_x^{k-n} Z_{r,t}\delta (\partial_x^ nu\odot \partial_x^k z)_{s,r}- \sum_{n=0}^{k} \binom{k}{n} \partial_x^{k-n}  g_{r,t}\delta (\partial_x^n z\odot \partial_x^k z)_{s,r}\,,
	\end{align*}
	\begin{align*}
	\delta \mathbb{I}_{s,r,t}&=\partial^k_x[\delta [[Zg]+[gZ]]_{s,r,t}u_s]\odot \partial^k_x z_s-\sum_{n=0}^{k} \binom{k}{n} \partial_x^{k-n} [[Zg]+[gZ]]_{r,t}\delta (\partial_x^n u\odot z)_s\\
	&\quad+\partial^k_x[\delta [gg]_{s,r,t}z_s]\odot  \partial^k_x z_s-\sum_{n=0}^{k} \binom{k}{n} \partial_x^{k-n} [gg]_{r,t}\delta (\partial_x^n z\odot z)_{s,r}\,,
	\end{align*}
	\begin{align*}
	\delta \tilde{\mathbb{I}}_{s,r,t}&=\partial_{x}^{k}[Z_{r,t}u_s+g_{r,t}z_s]\odot \partial_{x}^{k}[Z_{s,r}u_s+g_{s,r}z_s]\\
	&\quad+\partial_{x}^{k}[Z_{r,t}u_s+g_{r,t}z_s]\otimes\partial_x^k[ Z_{r,t}u_s+g_{r,t}z_s]-\partial_{x}^{k}[Z_{r,t}u_r+g_{r,t}z_r]\otimes\partial_x^k[ Z_{r,t}u_r+g_{r,t}z_r]\,.
	\end{align*}
	The following terms of only $p/2$-variation
	\begin{equation}\label{eq:da_compensare_MDP}
	\begin{aligned}
	\partial^k_x[\delta [[Zg]+[gZ]]_{s,r,t}u_s]\odot \partial^k_x z_s+\partial^k_x[\delta [gg]_{s,r,t}z_s]\odot  \partial^k_x z_s\\
	+\partial^k_x[Z_{r,t}u_s+g_{r,t}z_s]\odot\partial^k_x [Z_{s,r}u_s+g_{s,r}z_s]
	\end{aligned}
	\end{equation}
	compensate by adding and subtracting the elements $T_1$ from $\delta (\partial_x^n u\odot \partial^k_x z)_{s,r}$, where
	\begin{align*}
	T_1:=\partial_x^n[g_{s,r} u_s] \odot \partial_x^k z_s +\partial_x^n u_s \odot \partial_x^k[Z_{s,r}u_s+g_{s,r}z_s]
	\end{align*}
	and $T_2$ from $\delta (\partial_x^n z\odot  \partial_x^k z)_{s,r}$, where
	\begin{align*}
	T_2:=\partial_x^n z_s \odot \partial_x^k [Z_{s,r}u_s+g_{s,r}z_s]+ \partial_x^n [Z_{s,r}u_s+g_{s,r}z_s]\odot \partial_x^k z_s\,.
	\end{align*}
	We write explicitly
	\begin{align*}
	&\quad-\sum_{n=0}^{k} \binom{k}{n} \partial_x^{k-n}Z_{r,t}T_1- \sum_{n=0}^{k} \binom{k}{n} \partial_x^{k-n}g_{r,t}T_2\\
	&=-\sum_{n=0}^{k} \binom{k}{n} [\partial_x^{k-n}Z_{r,t}\partial_x^n[g_{s,r} u_s] \odot \partial_x^k z_s -\partial_x^{k-n}Z_{r,t}\partial_x^n u_s \odot \partial_x^k[Z_{s,r}u_s+g_{s,r}z_s]]\\
	&\quad-\sum_{n=0}^{k} \binom{k}{n}[\partial_x^{k-n}g_{r,t}\partial_x^n z_s \odot \partial_x^k [Z_{s,r}u_s+g_{s,r}z_s]- \partial_x^{k-n}g_{r,t}\partial_x^n [Z_{s,r}u_s+g_{s,r}z_s]\odot \partial_x^k z_s]\\
	&=-\partial_x^{k}[Z_{r,t}g_{s,r} u_s] \odot \partial_x^k z_s -\partial_x^{k}[Z_{r,t}u_s ]\odot \partial_x^k[Z_{s,r}u_s+g_{s,r}z_s]\\
	&\quad -\partial_x^{k}[g_{r,t} z_s ]\odot \partial_x^k [Z_{s,r}u_s+g_{s,r}z_s]-\partial_x^{k}[g_{r,t} Z_{s,r}u_s+g_{r,t}g_{s,r}z_s]\odot \partial_x^k z_s]\,,
	\end{align*}
	which erases with the elements in \eqref{eq:da_compensare_MDP}. In conclusion, we can rewrite the remainder as
	\begin{align*}
	\delta(\partial_x^k  z^{2,\natural})_{s,r,t}&=\sum_{n=0}^{k} \binom{k}{n} \partial_x^{k-n} Z_{r,t}[\delta (\partial_x^ nu\odot \partial_x^k z)_{s,r}-T_1]+ \sum_{n=0}^{k} \binom{k}{n} \partial_x^{k-n}  g_{r,t}[\delta (\partial_x^n z\odot \partial_x^k z)_{s,r}-T_2]\\
	&\quad+\sum_{n=0}^{k} \binom{k}{n} \partial_x^{k-n} [[Zg]+[gZ]]_{r,t}\delta (\partial_x^n u\odot z)_s+\sum_{n=0}^{k} \binom{k}{n} \partial_x^{k-n} [gg]_{r,t}\delta (\partial_x^n z\odot z)_{s,r}\\
	&\quad+\partial_{x}^{k}[Z_{r,t}u_s+g_{r,t}z_s]\otimes\partial_x^k[ Z_{r,t}u_s+g_{r,t}z_s]-\partial_{x}^{k}[Z_{r,t}u_r+g_{r,t}z_r]\otimes\partial_x^k[ Z_{r,t}u_r+g_{r,t}z_r]\,.
	\end{align*}
	Each of the above term presents has as its component $Z$ or $z$: this will ensure the continuity of the map. In order to bound $\delta(z^{2,\natural})_{s,r,t}$, we need to know estimates for $\delta (z\odot z)_{s,r}$, $\delta (u\odot z)_{s,r}$: in particular we need to bound the drift. This step follows from Assumption \ref{assumption:drift_linear_b_prime}.
	In order to bound $\delta(\partial_x z^{2,\natural})_{s,r,t}$, we need the estimates of the drifts coming from
	\begin{align*}
	\delta (z\odot z)_{s,r}\, \quad \delta (\partial_x z\odot z)_{s,r}\,,\quad \delta (z\odot \partial_x z)_{s,r}\,,\quad \delta (\partial_x z\odot \partial_x z)_{s,r}\,,
	\end{align*}
	\begin{align*}
	\delta (\partial_x u\odot z)_{s,r}\,,\quad \delta (u\odot \partial_x z)_{s,r}\,,\quad \delta (\partial_x u\odot\partial_x z)_{s,r}\,.
	\end{align*}
	We estimate the drift coming from $\delta (\partial_x z\odot \partial_x z)_{s,r}$, since the others can be bounded analogously. Also this bound follows from  Assumption \ref{assumption:drift_linear_b_prime}. The estimate of the remainder at the first level of the energy becomes
	\begin{align*}
	\|  z^{2,\natural}_{s,t}\|_{(W^{1,\infty})^*}&\lesssim_{u,  u^0,\mathbf{W},\mathbf{G},\mathbf{Z}} \omega_{Z}^{1/p}\left[\int_{s}^{t}[\|\nabla z\|^2_{L^2}+\|\nabla u\|^2_{L^2}]\dd r+\omega^{2/p}\left[\|u\|_{L^\infty([s,t];L^2)}^2+\|z\|_{L^\infty([s,t];L^2)}^2\right]\right]\\
	&\quad+\omega^{1/p}\left[\int_{s}^{t}\|\nabla z\|^2_{L^2}\dd r+\omega^{2/p}\sup_{s\leq r\leq t}\|z_r\|^2_{L^2}\right]\,,
	\end{align*}
	where the controls $\omega$ in the above inequality are evaluated in $s\leq t\in [0,T]$.
	Hence the first level estimate, after the application of the rough Gronwall's Lemma~\ref{lem:gronwall} becomes
	\begin{align*}
	\sup_{0\leq t\leq T}\|z_r\|^2_{L^2}+\int_{0}^{T}\|\nabla z_r\|^2_{L^2}\dd r\lesssim_{u,  u^0,\mathbf{W},\mathbf{G},\mathbf{Z},Y^\mathbf{W},Y^\mathbf{G}} \omega^{1/p}_{\mathbf{Z}}(0,T)+\omega^{2/p}_{[Zg]+[gZ]}(0,T)\,.
	\end{align*}
	We estimate the second level energy. The estimate for the drift has the form
	\begin{align*}
	\| \partial_x  z^{2,\natural}_{s,t}\|_{(W^{1,\infty})^*}&\lesssim \omega_{Z;W^{1,\infty}}^{1/p}\left[\int_{s}^{t}[\|\Delta z\|^2_{L^2}+\|\Delta u\|^2_{L^2}]\dd r+\omega^{2/p}\left[\|u\|_{L^\infty([s,t];H^1)}^2+\|z\|_{L^\infty([s,t];H^1)}^2\right]\right]\\
	&\quad +\omega^{1/p} \left[\int_{s}^{t}\|\Delta z\|^2_{L^2}\dd r+\omega^{2/p}\sup_{s\leq r\leq t}\|z_r\|^2_{H^1}\right]\,,
	\end{align*}
	where the controls $\omega$ in the above inequality are evaluated in $s\leq t\in [0,T]$.
	The second level energy has the form
	\begin{align*}
	\sup_{0\leq t\leq T}\|\nabla z_r\|^2_{L^2}+\int_{0}^{T}\|\Delta z_r\|^2_{L^2}\dd r\lesssim_{u,  u^0,\mathbf{W},\mathbf{G},\mathbf{Z},Y^\mathbf{W},Y^\mathbf{G}} \omega^{1/p}_{\mathbf{Z}}(0,T)+\omega^{2/p}_{[Zg]+[gZ]}(0,T)\,.
	\end{align*}
	Remark that we have also used that in the Definition~\ref{def:sol_rough} of solution for $y$, we have an energy estimate: from this, the dependence on $\mathbf{g},u,\mathbf{u^0}$.
	From the local Lipschitz continuity of the Young's integral, the local Lipschitz continuity of the It\^o-Lyons map $\mathbf{W}\mapsto D\Phi[\mathbf{g}](\mathbf{W})$ follows.
	If there exists a solution $Y^\mathbf{G}$ to \eqref{eq:Dphi_Y_not_centered}, then the equation is unique and its It\^o-Lyons map is locally Lipschitz continuous with respect to the noise $\mathbf{W}$.
	We discuss now the existence of a solution to \eqref{eq:Dphi_Y_not_centered} in the sense of Defintion \ref{def:sol_additive}. Let $\mathbf{W}^n$ be a sequence of smooth approximations of the rough path $\mathbf{W}$ and denote by $Y^n:=\pi(\mathbf{W}^n)$ (recall that the initial condition of this equation is $0$).
	By looking at the estimate of the noise for the Wong-Zakai, we can derive an a priori bound on the sequence $(Y^n)_n$ in $L^\infty(H^1)\cap L^2(H^2)\cap C([0,T];L^2)$. From Lemma~\ref{lemma:embedding_tornstein}, there exists a subsequence converging to $Y$ in $L^\infty(L^2)\cap L^2(H^1)\cap C(H^{-1})$. Since the drift is linear, the identification of the limit in the drift is standard. From the lower semi-continuity of the norm, the limit solution is also an element of $L^\infty(H^1)\cap L^2(H^2)\cap C([0,T];L^2)$.

\end{proof}

\bibliographystyle{plain}

\end{document}